\documentclass{amsart}

\usepackage{amsmath}
\usepackage{amssymb}
\usepackage{amsthm}
\usepackage{arydshln}
\usepackage{bbm}
\usepackage{bm}
\usepackage{colonequals}
\usepackage{datetime2}
\usepackage{mathrsfs}
\usepackage{mathtools}
\usepackage{stmaryrd}
\usepackage[all]{xy}
\usepackage{yhmath}

\usepackage{imakeidx}
\makeindex[name=symbols,title=List of Symbols,columns=2]
\newcommand{\symdef}[2]{\index[symbols]{#1@#2}}

\usepackage[dvipsnames]{xcolor}
\usepackage[
  bookmarks=true,
  bookmarksnumbered=true,
  colorlinks=true,
  linkcolor=RoyalBlue,  
  citecolor=ForestGreen,
  urlcolor=Magenta,     
  linktoc=all
]{hyperref}

\theoremstyle{plain}
\newtheorem{thm}{Theorem}[section]
\newtheorem*{thm*}{Theorem}
\newtheorem{prop}[thm]{Proposition}
\newtheorem{lem}[thm]{Lemma}
\newtheorem{cor}[thm]{Corollary}

\theoremstyle{definition}
\newtheorem{defn}[thm]{Definition}

\theoremstyle{remark}
\newtheorem{rem}[thm]{Remark}
\newtheorem{example}[thm]{Example}
\newtheorem*{claim*}{Claim}

\newcommand{\asym}{\mathrm{asym}}

\newcommand{\ram}{\mathrm{ram}}
\newcommand{\red}{\mathrm{red}}
\newcommand{\res}{\mathrm{res}}
\newcommand{\rs}{\mathrm{rs}}

\newcommand{\spl}{\mathbf{spl}}

\newcommand{\supp}{\mathrm{supp}}
\newcommand{\sym}{\mathrm{sym}}
\newcommand{\ur}{\mathrm{ur}}

\newcommand{\Wal}{\mathrm{Wal}}

\newcommand{\IV}{\mathrm{IV}}

\newcommand{\C}{\mathbb{C}}
\newcommand{\F}{\mathbb{F}}
\newcommand{\Q}{\mathbb{Q}}
\newcommand{\R}{\mathbb{R}}
\newcommand{\Z}{\mathbb{Z}}
\newcommand{\Gm}{\mathbb{G}_{\mathrm{m}}}

\newcommand{\bbH}{\mathbb{H}}

\newcommand{\bfa}{\mathbf{a}}

\newcommand{\x}{\mathbf{x}}
\newcommand{\y}{\mathbf{y}}
\newcommand{\bfA}{\mathbf{A}}
\newcommand{\bfB}{\mathbf{B}}
\newcommand{\B}{\mathbf{B}}
\newcommand{\bfC}{\mathbf{C}}

\newcommand{\bfG}{\mathbf{G}}
\newcommand{\G}{\mathbf{G}}
\newcommand{\bfH}{\mathbf{H}}

\newcommand{\J}{\mathbf{J}}
\newcommand{\bfN}{\mathbf{N}}
\newcommand{\bfS}{\mathbf{S}}
\newcommand{\bfT}{\mathbf{T}}
\newcommand{\T}{\mathbf{T}}

\newcommand{\bfV}{\mathbf{V}}
\newcommand{\bfX}{\mathbf{X}}
\newcommand{\bfZ}{\mathbf{Z}}

\newcommand{\mfg}{\mathfrak{g}}

\newcommand{\mfj}{\mathfrak{j}}
\newcommand{\mfp}{\mathfrak{p}}
\newcommand{\mfs}{\mathfrak{s}}

\newcommand{\mfG}{\mathfrak{G}}

\newcommand{\mcA}{\mathcal{A}}
\newcommand{\mcB}{\mathcal{B}}

\newcommand{\mcE}{\mathcal{E}}
\newcommand{\mcF}{\mathcal{F}}

\newcommand{\mcK}{\mathcal{K}}

\newcommand{\mcO}{\mathcal{O}}

\newcommand{\bmfg}{\boldsymbol{\mathfrak{g}}}

\newcommand{\bmfs}{\boldsymbol{\mathfrak{s}}}

\newcommand{\ol}{\overline}
\newcommand{\ul}{\underline}
\renewcommand{\t}{\tilde}
\newcommand{\h}{\hat}

\newcommand{\nat}{\natural}

\DeclareMathOperator{\cInd}{c-Ind}

\DeclareMathOperator{\id}{id}

\DeclareMathOperator{\sgn}{sgn}

\DeclareMathOperator{\tr}{tr}

\DeclareMathOperator{\val}{val}

\DeclareMathOperator{\Aut}{Aut}
\DeclareMathOperator{\Cok}{Cok}
\DeclareMathOperator{\Gal}{Gal}
\DeclareMathOperator{\Hom}{Hom}
\DeclareMathOperator{\Ind}{Ind}
\DeclareMathOperator{\Ker}{Ker}
\DeclareMathOperator{\Lie}{Lie}
\DeclareMathOperator{\Nr}{Nr}
\DeclareMathOperator{\Res}{Res}
\DeclareMathOperator{\Tr}{Tr}

\DeclareMathOperator{\GL}{GL}
\DeclareMathOperator{\SL}{SL}
\DeclareMathOperator{\Sp}{Sp}

\title{Twisted character formula for toral supercuspidal representations}
\author{Masao Oi}
\address{Department of Mathematics, National Taiwan University, Astronomy Mathematics Building 5F, No.\ 1, Sec.\ 4, Roosevelt Rd., Taipei 10617, Taiwan.}
\email{masaooi@ntu.edu.tw}

\begin{document}

\begin{abstract}
We establish an explicit formula for twisted Harish-Chandra characters of toral supercuspidal representations of $p$-adic reductive groups under several technical assumptions.
Our setup especially includes the case of a quasi-split group equipped with an involution.
\end{abstract}


\maketitle

\begingroup
\renewcommand{\thefootnote}{}
\footnotetext{\DTMnow}
\footnotetext{Note: The material of this paper was originally incorporated as part of \cite{Oi23-TECR} and has been extracted and reformulated in a slightly more general setting.}
\endgroup

\setcounter{tocdepth}{2}
\tableofcontents


\section{Introduction}\label{sec:intro}

The aim of this paper is to establish an explicit formula of \textit{twisted (Harish-Chandra) characters} for a specific class of supercuspidal representations of $p$-adic reductive groups called \textit{toral supercuspidal representations}.

Let $\bfG$ be a connected reductive group over a $p$-adic field $F$.
We consider an irreducible admissible representation $\pi$ of $\bfG(F)$.
Recall that, in representation theory of finite groups, the trace characters of representations play a very important role.
However, in our context, it does not make sense to take the trace of $\pi$ in the usual way since $\pi$ is mostly infinite-dimensional.
Nevertheless, we can still associate to $\pi$ a $\C$-valued function $\Theta_\pi$ on (the regular semisimple locus of) $\bfG(F)$ called the \textit{Harish-Chandra} character of $\pi$, which encodes all the information of $\pi$ in principle just as the trace character of a representation of a finite group.
Based on the philosophy ``\textit{characters tell all}'' (\cite{SS09}), it is natural to try to seek an explicit formula for Harish-Chandra characters of irreducible admissible representations of $p$-adic reductive groups.
Of course, it is impossible to describe $\Theta_\pi(\gamma)$ `explicitly' if the given information is merely that $\pi$ is an irreducible admissible representation and $\gamma$ is a regular semisimple element.
In order to obtain a concrete expression of $\Theta_\pi(\gamma)$, we must specify concretely what kind of $\pi$ and $\gamma$ are under consideration, by providing explicit data (labeling) for them.
Keeping this in mind, let us first introduce the pivotal works of Adler--Spice \cite{AS08,AS09} and DeBacker--Spice \cite{DS18}.

In representation theory of $p$-adic reductive groups, \textit{supercuspidal} representations are often referred to as the ``building blocks''.
In \cite{Yu01}, J.-K.\ Yu introduced a method for constructing a lot of supercuspidal representations in a highly explicit manner.
The supercuspidal representation obtained by Yu's construction are called \textit{tame} supercuspidal representations.
The input data of Yu's construction is a $5$-tuple $(\vec{\bfG}, \vec{\vartheta},\vec{r},\x,\rho_0)$ consisting of very concrete objects (see Section \ref{sec:rsc} for details).
On the other hand, in \cite{AS08}, Adler--Spice introduced the notion of a \textit{normal $r$-approximation} for elements of $p$-adic reductive groups.
Roughly speaking, for an element $\gamma\in \bfG(F)$, a normal $r$-approximation to $\gamma$ is a product decomposition $\gamma=\prod_{0\leq i <r}\gamma_i\cdot \gamma_{\geq r}$ along the $p$-adic depth of elements of $p$-adic reductive groups (the depth of each $\gamma_i$ is $i$ unless $\gamma_i=1$ and the depth of $\gamma_{\geq r}$ is greater than or equal to $r\in\R_{\geq0}$).
A normal $r$-approximation can be thought of as an enhancement of the \textit{topological Jordan decomposition} in the sense of \cite{Spi08}, which decomposes $\gamma$ into the product of a $p$-adically semisimple element $\gamma_0$ and a $p$-adically unipotent element $\gamma_{+}$ which commute.
Although it is not always possible to find a normal $r$-approximation for a given element, it is shown that any elliptic regular semisimple element admits a normal $r$-approximation.

With these languages, now it makes sense to investigate an explicit formula of $\Theta_\pi(\gamma)$ for a tame supercuspidal representation $\pi$ of $\bfG(F)$ in terms of its input data $(\vec{\bfG}, \vec{\vartheta},\vec{r},\x,\rho_0)$ and a normal $r$-approximation $\gamma=\prod_{0\leq i <r}\gamma_i\cdot \gamma_{\geq r}$.
A complete answer to this problem was given by Adler--Spice \cite{AS09} under a certain compactness assumption on the Yu-datum $(\vec{\bfG}, \vec{\vartheta},\vec{r},\x,\rho_0)$.
We do not explain what the ``compactness'' means here because from now on we restrict ourselves to a further special class of tame supercuspidal representations called \textit{toral} supercuspidal representations, for which the assumption always holds.
(In fact, the compactness assumption has been removed by recent work of Spice \cite{Spi18,Spi21}; see remarks at the end of this introduction.)
The Yu-datum of a toral supercuspidal representation is of a particularly simple form (see Section \ref{subsec:toral-sc}).
The key point is that such a Yu-datum can be thought of just as a pair $(\bfS,\vartheta)$ of a tamely ramified elliptic maximal torus $\bfS$ of $\bfG$ and a character $\vartheta\colon\bfS(F)\rightarrow\C^\times$ satisfying a certain regularity condition.
Let us call such a pair $(\bfS,\vartheta)$ a \textit{toral pair}.
(This is a special case of Kaletha's re-parametrization theorem for \textit{regular supercuspidal representations}.)

The formula of Adler--Spice \cite{AS09}, which was later elaborated by DeBacker--Spice \cite{DS18}, is as follows:
\begin{thm}[Adler--DeBacker--Spice character formula]\label{thm:ADS}
    Suppose that $p\gg0$.
    Let $\pi$ be a toral supercuspidal representation of $\bfG(F)$ corresponding to a toral pair $(\bfS,\vartheta)$.
    Let $r\in\R_{\geq0}$ be the depth of $\pi$.
    For any elliptic regular semisimple element $\gamma\in \bfG(F)$ with a normal $r$-approximation $\gamma=\prod_{0\leq i <r}\gamma_i\cdot \gamma_{\geq r}$, 
    \[
        \Phi_\pi(\gamma)=
        \sum_{\begin{subarray}{c}g\in \bfS(F)\backslash \bfG(F)/\bfG_{\gamma_{<r}}(F)\\ {}^{g}\gamma_{<r}\in \bfS(F) \end{subarray}} \vartheta({}^{g}\gamma_{<r})\cdot\varepsilon({}^{g}\gamma_{<r})\cdot\hat{\iota}^{\bfG_{{}^{g}\gamma_{<r}}}_{X^\ast}(\log({}^{g}\gamma_{\geq r})),
    \]
    where $\Phi_\pi$ denotes the normalized Harish-Chandra character of $\pi$.
\end{thm}
Let us explain the right-hand side in a bit more detail.
We put $\gamma_{<r}:=\prod_{0\leq i <r}\gamma_i$.
Then, by the definition of a normal approximation, $\gamma_{<r}$ (``$r$-head'' of $\gamma$) commutes with $\gamma_{\geq r}$ (``$r$-tail'' of $\gamma$).
The key feature of this formula is that the $r$-head and the $r$-tail contribute to $\Theta_{\pi}(\gamma)$ in a separated way.
The contribution of the $r$-head is given by the character $\vartheta$ and a certain function $\varepsilon$ on $\bfS(F)$ introduced by DeBacker--Spice in \cite{DS18}, which is determined by the root-theoretic information of the pair $(\bfS,\vartheta)$.
On the other hand, the contribution of the $r$-tail is given in terms of the normalized Fourier transform of the Lie algebra orbital integral $\hat{\iota}^{\bfG_{{}^{g}\gamma_{<r}}}_{X^\ast}$.
Here, $X^\ast$ is an element of the dual Lie algebra determined by the pair $(\bfS,\vartheta)$.
The symbol ``$\log(-)$'' denotes a logarithm map from a $p$-adic reductive group to its Lie algebra (the assumption ``$p\gg0$'' is needed to guarantee that such a nice map exists).
Note that the orbital integral is taken not in the original connected reductive group $\bfG$, but in the connected centralizer group $\bfG_{{}^{g}\gamma_{<r}}$ of the $r$-head part ${}^{g}\gamma_{<r}$.
Therefore, we may interpret this formula as a manifestation of Harish-Chandra's philosophy of \textit{descent} in representation theory and harmonic analysis of $p$-adic reductive groups.
For more details, see \cite[Theorem 4.28]{DS18}; however, we caution that our description of Theorem \ref{thm:ADS} is (over-)simplified for exposition.

We remark that although $\gamma$ is assumed to be elliptic in Theorem \ref{thm:ADS}, it is practically harmless since supercuspidal representations can be distinguished from each other only by their Harish-Chandra character values on elliptic regular semisimple elements (a consequence of the \textit{elliptic inner product pairing}).

Now, let us move on to the ``twisted'' situation.
We fix an $F$-rational automorphism $\theta$ of $\bfG$ of finite order.
For an irreducible admissible representation $\pi$ of $\bfG(F)$, we define its \textit{$\theta$-twist} to be $\pi^\theta:=\pi\circ\theta$.
We say that $\pi$ is $\theta$-stable if $\pi^\theta$ is isomorphic to $\pi$; this is equivalent to that $\pi$ can be extended to a representation $\tilde{\bfG}$ of the semi-direct group $\bfG(F)\rtimes \langle\theta\rangle$.
When $\pi$ is $\theta$-stable, we can introduce a variant of the Harish-Chandra character called the \textit{$\theta$-twisted Harish-Chandra character} $\Theta_{\tilde{\pi}}$, which is a function on the regular semisimple locus of $\tilde{\bfG}(F):=\bfG(F)\rtimes\theta$.
The twisted Harish-Chandra characters plays an important role in the context of the Langlands correspondence.
In particular, twisted Harish-Chandra characters are necessary to formulate various basic cases of the Langlands functoriality called the \textit{twisted endoscopic lifting}, which includes the base change, automorphic induction, and so on.
Indeed, the motivation of this paper is an application to the twisted endoscopy; our result will be utilized in our subsequent work \cite{Oi23-TECR}, which compares Arthur's construction of the local Langlands correspondence for quasi-split classical groups \cite{Art13} with Kaletha's explicit construction of regular supercuspidal local Langlands correspondence \cite{Kal19}.

Let $\pi$ be a toral supercuspidal representation of $\bfG(F)$ associated to a pair $(\bfS,\vartheta)$.
Our aim is to, by assuming that $\pi$ is $\theta$-stable, establish an ``Adler--DeBacker--Spice-like'' explicit formula of $\Theta_{\tilde\pi}(\delta)$ for $\delta\in\tilde{\bfG}(F)$.
For this, we first have to think about a twisted version of a normal $r$-approximation, that is, how to define a normal $r$-approximation for elements of $\tilde{\bfG}(F)$ such that a computation of twisted Harish-Chandra characters can be developed based on it and also that its existence is ensured for a reasonable class of elements (especially, elliptic regular semisimple elements) of $\tilde{\bfG}(F)$.

The key observation from the untwisted setting is the following.
Let $\gamma=\gamma_{<r}\cdot\gamma_{\geq r}$ be a normal $r$-approximation to a regular semisimple element $\gamma\in \bfG(F)$.
As mentioned before, a normal $r$-approximation is a refinement of a topological Jordan decomposition; the $r$-head $\gamma_{<r}$ is furthermore decomposed into a topologically semisimple part $\gamma_{0}$ and a topologically unipotent (and shallower-than-$r$) part $\gamma_{<r}^{+}$ such that $\gamma=\gamma_{0}\cdot(\gamma_{<r}^{+}\cdot\gamma_{\geq r})$ gives a topological Jordan decomposition of $\gamma$.
If we put $\gamma_{+}\colonequals \gamma_{<r}^{+}\cdot\gamma_{\geq r}$, then $\gamma_{+}$ belongs to $\G_{\gamma_{0}}$ and the decomposition $\gamma_{+}=\gamma_{<r}^{+}\cdot\gamma_{\geq r}$ is a normal $r$-approximation to $\gamma_{+}$ in $\G_{\gamma_{0}}$.
In summary, a normal $r$-approximation to $\gamma$ in $\bfG$ induces a topological Jordan decomposition $\gamma=\gamma_0\cdot\gamma_+$ and a normal $r$-approximation to $\gamma_+$ in $\bfG_{\gamma_0}$.

Let us consider performing this procedure in the reverse direction.
For a given element $\gamma\in\bfG(F)$, we take a topological Jordan decomposition $\gamma=\gamma_{0}\cdot\gamma_{+}$ and a normal $r$-approximation $\gamma_{+}=\gamma_{<r}^{+}\cdot\gamma_{\geq r}$ in $\G_{\gamma_{0}}$.
In fact, then the resulting decomposition $\gamma=(\gamma_{0}\cdot\gamma_{<r}^{+})\cdot\gamma_{\geq r}$ might not necessarily be a normal $r$-approximation to $\gamma$ in $\G$ (because a ``good'' element of $\G_{\gamma_{0}}$ might not be ``good'' in $\G$; see Definitions \ref{defn:good-element}--\ref{defn:approx}).
The point here is that, however, the property that $\gamma_{+}=\gamma_{<r}^{+}\cdot\gamma_{\geq r}$ is a normal $r$-approximation in $\G_{\gamma_{0}}$ is enough so that the arguments in \cite{AS08,AS09,DS18} work well.

Based on this observation, we arrive at the following ``construction'' (rather than a conceptual definition) of a normal $r$-approximation in the twisted setting.
Let $\delta\in\t\G(F)$ be an elliptic regular semisimple element.
\begin{enumerate}
\item
Take a topological Jordan decomposition $\delta=\delta_{0}\cdot\delta_{+}$ (the meaning of a ``topological Jordan decomposition'' is the same as in the untwisted setting, i.e., $\delta_{0}$ is $p$-adically semisimple and $\delta_{+}$ is $p$-adically unipotent such that $\delta_0\delta_+=\delta_+\delta_0$).
Note that, while $\delta_{0}$ lies in $\t\G=\G\rtimes\theta$, $\delta_{+}$ lies in the ``untwisted'' part $\G$, in fact, even the connected centralizer $\G_{\delta_{0}}$ of the conjugate action by $\delta_0$ on $\bfG$.
\item
Take a normal $r$-approximation $\delta_{+}=\delta_{<r}^{+}\cdot\delta_{\geq r}$ in $\G_{\delta_{0}}$.
Note that $\delta_{+}\in\G_{\delta_{0}}$ is no longer ``twisted'', hence the work of Adler--Spice \cite{AS08} is enough to find a normal $r$-approximation to $\delta_{+}$ in $\G_{\delta_{0}}$.
\item
We call the resulting $\delta=\delta_{0}\cdot\delta_{<r}^{+}\cdot\delta_{\geq r}$ a \textit{normal $r$-approximation} to $\delta$. 
\end{enumerate}

We remark that we need to assume that $p$ is odd and prime to the order of $\theta$ so that this construction works.

Now let us state our main result.
The setup is as follows.
Let $\G$ be a tamely ramified quasi-split connected reductive group over $F$ and $\theta$ an $F$-rational involution of $\G$ preserving an $F$-splitting of $\G$.
We assume that the restricted root system with respect to $(\G,\bfS,\theta)$ does not contain any restricted root of type $2$ or $3$.
We also assume a technical condition on the restricted root system, which is satisfied when $\theta$ is an involution, for example (see Section \ref{subsec:twisted-Weil-asym}).
Let $\tilde{\bfS}\subset\tilde{\bfG}$ be a tame elliptic twisted maximal torus (see Section \ref{subsec:twisted-tori} for the definition of a twisted maximal torus); by fixing a base point $\ul{\eta}\in\tilde{\bfS}(F)$, we may write $\tilde{\bfS}=\bfS\cdot\ul{\eta}$ for a tame elliptic maximal torus $\bfS$ of $\bfG$.
The $\ul{\eta}$-conjugation on $\bfG$ preserves $\bfS$, hence induces an $F$-rational automorphism $\theta_\bfS$ of $\bfS$.
We consider a toral pair $(\bfS,\vartheta)$ such that $\vartheta=\vartheta\circ\theta_\bfS$.
Then the corresponding toral supercuspidal representation $\pi:=\pi_{(\bfS,\vartheta)}$ is a $\theta$-stable toral supercuspidal representation.
Let $r$ be the depth of $\pi$.

\begin{thm}[Theorem \ref{thm:TCF-final}]\label{thm:twisted-ADS}
    Suppose $p\gg0$.
    Let $\delta\in\tilde{\bfG}(F)$ be an elliptic regular semisimple element with a normal $r$-approximation $\delta=\delta_{0}\cdot\delta_{<r}^{+}\cdot\delta_{\geq r}$ constructed in the manner as above.
    Then
    \[
    \Phi_{\t\pi}(\delta)
    =
    C_{\ul{\eta}}\cdot
    \sum_{\begin{subarray}{c}g\in \bfS(F)\backslash \bfG(F)/\bfG_{\delta_{<r}}(F) \\ {}^{g}\delta_{<r}\in \t{\bfS}(F)\end{subarray}} \t\varepsilon({}^{g}\delta_{<r})\cdot\t\vartheta({}^{g}\delta_{<r})\cdot\h\iota^{\G_{{}^{g}\delta_{<r}}}_{X^{\ast}}\bigl(\log({}^{g}\delta_{\geq r})\bigr),
    \]
    where $\Phi_{\t\pi}$ is the normalized twisted Harish-Chandra character of $\t\pi$.
\end{thm}

We do not explain the meanings of the symbols $C_{\ul{\eta}}$, $\tilde{\varepsilon}$, and $\tilde{\vartheta}$ on the right-hand side; these are something computable explicitly and written in terms of arithmetic or root-theoretic quantities as in the untwisted case.
Instead, we emphasize that the formula looks formally the same as the one in the untwisted setting (Theorem \ref{thm:ADS}).
Indeed, much of the proof can be carried out by following the arguments of \cite{AS09} and \cite{DS18}.

The starting point of the proof is the twisted version of Harish-Chandra's integration formula.
Since $\pi$ is defined to be the compact induction of a finite-dimensional representation $\sigma$ of an open compact-modulo-center subgroup $K_\sigma\subset \bfG(F)$, the integration formula gives
\[
\Theta_{\t{\pi}}(\delta)
=
\frac{\deg\pi}{\dim\sigma}
\int_{\bfG(F)/\bfZ_{\G}(F)}
\int_{\mcK} \dot{\Theta}_{\t{\sigma}} ({}^{\dot{g}k}\delta)\, dk\,d\dot{g}.
\]
Here, our assumption on $(\bfS,\vartheta)$ implies that $K_\sigma$ is also $\ul{\eta}$-stable, hence we may consider the twisted character $\Theta_{\tilde{\sigma}}$ (see the proof of Theorem \ref{thm:TCF-0} for the other symbols used in the above formula).
By imitating the computation of \cite{AS09} using the properties of a normal $r$-approximation, we can rewrite the formula as follows:
\[
\Theta_{\t{\pi}}(\delta)
=
\sum_{\begin{subarray}{c} g\in \bfS(F)\backslash \bfG(F)/\bfG_{\eta}(F) \\ {}^{g}\eta\in\t{\bfS}(F)\end{subarray}}
\Theta_{\t{\sigma}}({}^{g}\eta)
\cdot
\hat{\mu}^{\G_{{}^{g}\eta}}_{X^{\ast}} \bigl(\log({}^{g}\delta_{\geq r})\bigr),
\]
where we write $\eta:=\delta_{<r}$ in short ($\hat{\mu}$ is an unnormalized Fourier transform of a Lie algebra orbital integral).
Thus the proof of Theorem \ref{thm:twisted-ADS} comes down to computing $\Theta_{\t{\sigma}}({}^{g}\eta)$.
To simplify the notation, we write $\eta$ for ${}^{g}\eta\in\t{\bfS}(F)$ by abuse of notation.

Since $\sigma$ is induced from a representation $\rho$ of a smaller group $K\subset K_\sigma$, which is also $\ul{\eta}$-stable, we can furthermore express $\Theta_{\t{\sigma}}(\eta)$ as the averaged sum of the twisted characters $\Theta_{\t{\rho}}({}^{k}\eta)$ over $k\in K\backslash K_\sigma$ by the Frobenius character formula.
In fact, this averaged sum is given by the product of $\Theta_{\t{\rho}}(\eta)$ and a constant determined by $\vartheta$ and the topological Jordan decomposition $\eta=\eta_0\cdot\eta_+$:
\[
\Theta_{\t{\sigma}}(\eta)
=
\Theta_{\t{\rho}}(\eta)
\cdot|\t{\mfG}_{\G_{\eta_{0}}}(\vartheta,\eta_{+})|
\cdot\mfG_{\G_{\eta_{0}}}(\vartheta,\eta_{+})
\]
(Corollary \ref{cor:5.3.2}).
Here, $|\t{\mfG}_{\G_{\eta_{0}}}(\vartheta,\eta_{+})|$ can be thought of as a certain volume term, which can be utilized to normalize $\hat{\mu}$.
On the other hand, $\mfG_{\G_{\eta_{0}}}(\vartheta,\eta_{+})$ is a root of unity which constitutes $\tilde{\varepsilon}$ in Theorem \ref{thm:twisted-ADS}.

Our discussion so far are just imitation of those of \cite{AS09} and \cite{DS18}.
However, we must face genuinely new difficulties to compute the term $\Theta_{\t{\rho}}(\eta)$.
To explain this, let us first briefly explain the construction of $\rho$ (see Section \ref{subsec:toral-sc} for the details) and also the computation in the untwisted setting by Adler--DeBacker--Spice, i.e., how to determine $\Theta_{\rho}(s)$ for $s\in \bfS(F)$.

The open compact-mod-center subgroup $K$ of $\bfG(F)$ where $\rho$ is defined can be written as the product $K=SJ$, where $S=\bfS(F)$ and $J$ is a subgroup normalized by $S$.
The point of the construction is that $J$ has a finite quotient which has a structure of a finite Heisenberg group $H(V)$ associated to a symplectic $\F_p$-vector space $V$.
Since the image of the $r$-th Moy--Prasad filtration $S_r$ of $S$ in $H(V)$ equals the center of $H(V)$, we may apply the Stone-von Neumann theorem to get an irreducible representation $\omega$ of $\Sp_{\F_p}(V)\ltimes H(V)$ (so-called the \textit{Heisenberg--Weil} representation).
As the conjugate action of $S$ on $J$ in fact preserves the symplectic structure of $V$, we have a natural homomorphism $S\ltimes J\rightarrow \Sp_{\F_p}(V)\ltimes H(V)$, hence $\omega$ can be regarded as a representation of $S\ltimes J$.
The tensor product of $\omega$ and $\vartheta$ (regarded as a character of $S\ltimes J$ via inflation) factors through $S\ltimes J\twoheadrightarrow SJ=K$.
We define $\rho$ to be the resulting representation of $K$.

By this definition, to determine $\Theta_{\rho}(s)$, it is enough to compute $\tr(\omega([s]))$, where $[s]\in\Sp_{\F_p}(V)$ denotes the conjugate action of $s$ on $V$.
A character formula for Heisenberg--Weil representations in the finite field setting has been already known by the work of G\'erardin \cite{Ger77}; this is what Adler--Spice \cite{AS09} and DeBacker--Spice \cite{DS18} utilized.
The point of the discussion is that the symplectic space part $V$ of the finite Heisenberg quotient $H(V)$ of $J$ has an orthogonal decomposition derived from the root space decomposition of $\Lie \bfG$ with respect to $\bfS$.
To be more precise, we let $\Phi$ be the set of absolute roots of $\bfS$ in $\bfG$, which is equipped with an action of $\Gamma:=\Gal(\overline{F}/F)$.
If we let $\dot{\Phi}:=\Phi/\Gamma$, then we have
\[
   V=\bigoplus_{\dot{\alpha}\in \dot{\Phi}} V_\alpha,
\]
where the action of $S$ on $V_\alpha$ is given by $\alpha$ (each $V_\alpha$ could be zero).
We call $\alpha\in\Phi$ \textit{symmetric} (resp.\ \textit{asymmetric}) if $-\alpha\in\dot{\alpha}=\Gamma\alpha$ (resp.\ $-\alpha\notin\dot{\alpha}=\Gamma\alpha$).
When $\alpha$ is symmetric, $V_\alpha$ is a symplectic subspace.
When $\alpha$ is asymmetric, $V_\alpha\oplus V_{-\alpha}$ forms a symplectic subspace.
Furthermore, the resulting decomposition
\[
    V=\bigoplus_{\ddot{\alpha}\in \ddot{\Phi}} V_{\ddot{\alpha}}
\]
is orthogonal, where $\Sigma:=\Gamma\times\{\pm1\}$, $\ddot{\Phi}:=\Phi/\Sigma$, and $V_{\ddot{\alpha}}=V_\alpha$ or $V_{\ddot{\alpha}}=V_\alpha\oplus V_{-\alpha}$ according to whether $\alpha$ is symmetric or asymmetric.
This implies that, noting that the conjugate action of $s$ on $V$ preserves each $V_{\ddot{\alpha}}$, we can compute $\tr(\omega([s]))$ as the product of the characters of Heisenberg--Weil representations $\omega_{\ddot{\alpha}}$ of $H(V_{\ddot{\alpha}})$ for $\ddot{\alpha}\in\ddot{\Phi}$:
\[
    \tr(\omega([s]))
    =
    \prod_{\ddot{\alpha}\in\ddot{\Phi}}\tr(\omega_{\ddot{\alpha}}([s])).
\]
Each symplectic space $V_{\ddot{\alpha}}$ is enough understandable, so that the character $\tr(\omega_{\ddot{\alpha}}([s]))$ can be explicated using \cite{Ger77}.
For example, when $\alpha$ is asymmetric, $V_\alpha$ and $V_{-\alpha}$ form a polarization of $V_{\ddot{\alpha}}$ and $[s]$ preserves this polarization.
In fact, G\'erardin's character formula for such symplectic automorphism becomes particularly simple.
Consequently, we obtain a certain sign function constituting $\varepsilon$ of the ADS formula (Theorem \ref{thm:ADS}).

Now let us go back to the twisted situation, that is, the computation of $\Theta_{\tilde{\rho}}(\eta)$ for $\eta\in\t{S}:=\t{\bfS}(F)$.
Again in this situation, the conjugate action of $\eta$ on $J$ preserves the symplectic structure of $V$, hence it is enough to compute $\tr(\omega([\eta]))$ for the symplectic automorphism $[\eta]$ of $V$ induced by $\eta$.
The difference from the untwisted case is that $[\eta]$ does not necessarily acts on each symplectic subspace $V_{\ddot{\alpha}}$.
To be more precise, note that the action of $\eta$ on $V$ induces a permutation on the set of roots $\Phi$.
If we write $\theta\colon\Phi\rightarrow\Phi$ for this permutation by abuse of notation (this does not depend on $\t{\eta}\in\t{S}$), then $[\eta]$ maps $V_{\ddot{\alpha}}$ onto $V_{\Sigma\theta(\alpha)}$.
Therefore, the smallest pieces of an orthogonal decomposition of $V$ we can utilize in the twisted setting are not individual $V_{\ddot{\alpha}}$'s, but their sums $V_{\ddot{\alpha}}+V_{\Sigma\theta(\alpha)}+ V_{\Sigma\theta^2(\alpha)}+\cdots$.
Accordingly, the product expression of $\tr(\omega([\eta]))$ is of the form
\[
    \tr(\omega([\eta]))
    =
    \prod_{\ddot{\alpha}\in\Theta\backslash\ddot{\Phi}}\tr(\omega_{\ddot{\alpha}}([\eta]^{m_\alpha})),
\]
where the index set is over $\Theta\,(:=\!\langle\theta\rangle)$-orbits of $\ddot{\Phi}$ and $m_\alpha$ denotes the order of $\Sigma\backslash(\Sigma\times\Theta)\alpha$ (see Corollary \ref{cor:twisted-HW-Yu}).

In fact, this is exactly the source of the complication.
For example, let us consider the case where $\alpha$ is asymmetric (hence $V_{\ddot{\alpha}}=V_{\alpha}\oplus V_{-\alpha}$ gives a polarization) and moreover $V_{\Sigma\theta(\alpha)}=V_{\ddot{\alpha}}$.
In this case, we have $\theta(\dot{\alpha})=\pm\dot{\alpha}$.
If $\theta(\dot{\alpha})=\dot{\alpha}$, then $V_{\alpha}$ and $V_{-\alpha}$ are preserved by $[\eta]$, thus the character of $[\eta]$ can be computed easily as in the untwisted case.
However, it could also happen that $\theta(\dot{\alpha})=-\dot{\alpha}$, then $V_{\alpha}$ and $V_{-\alpha}$ are swapped by $[\eta]$.
This means that the symplectic automorphism $[\eta]$ of $V_{\ddot{\alpha}}$ no longer preserves the polarization, which makes the description of G\'erardin's formula complicated.

What is happening in this example is that the $\theta$-orbit of $\alpha$ is symmetric with respect to the $\Gamma$-action even though $\alpha$ itself is asymmetric.
Therefore, to compute each $\tr(\omega_{\ddot{\alpha}}([\eta]^{m_\alpha}))$, we have to care about the relationship between the symmetry of the root system $\Phi$ and the twist $\theta$.
What we give eventually in this paper is to present an ``algorithm'' to determine each $\tr(\omega_{\ddot{\alpha}}([\eta]^{m_\alpha}))$ based on a case-by-case argument on the type of $\alpha$.
Hence, we do not really give a complete description of the twisted character in general.
However, depending on the situation ($\G$, $\theta$, and $\bfS$), the algorithm may terminate rather quickly. 
For example, it is the case when $\theta$ is involutive.


\medbreak
\noindent{\bfseries Organization of this paper.}\quad
In Section \ref{sec:notation}, we list our notation.
In Section \ref{sec:twisted-spaces}, we explain our setup and establish a version of the theory of good product expansion by Adler--Spice \cite{AS08} in the twisted space setting.
In Section \ref{sec:rsc}, we review Yu's construction of tame supercuspidal representations with emphasis on the toral case.
In Section \ref{sec:TCF1}, we establish a preliminary version of a twisted Adler--DeBacker--Spice character formula for toral supercuspidal representations.
In the main theorem of this section (Theorem \ref{thm:TCF-pre}), the contribution of the shallow part of $\delta$ remains not to be computed.
In Section \ref{sec:HW-twisted}, we compute the contribution of the shallow part by appealing to G\'erardin's character formula.
In Section \ref{sec:TCF-final}, we establish our main theorem (Theorem \ref{thm:TCF-final}).
What we have to do here is basically just to combine the results of Sections \ref{sec:TCF1} and \ref{sec:HW-twisted}, but we also need to be careful about normalization of various terms to get a cleaner formula.
Some of the results needed in this section are summarized in Appendix \ref{sec:twisted-HW}.

\medbreak
\noindent{\bfseries Acknowledgment.}\quad
I am deeply grateful to Atsushi Ichino, Tasho Kaletha, and Yoichi Mieda, who have dedicated a considerable amount of time to have discussions with me on this project.
I am also very grateful to Loren Spice, who helped me a lot by answering my technical questions on the good product expansion.
I also would like to thank Jeffrey Adler, Alexander Bertoloni Meli, Charlotte Chan, Guy Henniart, Tamotsu Ikeda, Wen-Wei Li, Sug Woo Shin, Kazuki Tokimoto, Sandeep Varma, and Alex Youcis for their encouragement and helpful discussions on this paper.

This work was carried out with the support from the Program for Leading Graduate Schools MEXT, JSPS Research Fellowship for Young Scientists, and KAKENHI Grant Number 17J05451, 19J00846, 20K14287.
This work was also supported by the Yushan Young Fellow Program, Ministry of Education, Taiwan.
I am grateful to the Hakubi Center of Kyoto University for providing a wonderful research environment during 2019--2024.

\section{Notation and assumptions on $p$}\label{sec:notation}

Here we summarize our basic notation and assumptions on $p$.
(See also the list of symbols at the end of this paper.)

\subsection{Notation}

\subsubsection{$p$-adic fields}
Let $p$ be a prime number.
We fix a $p$-adic field $F$, i.e., $F$ is a finite extension of $\Q_{p}$.
We fix an algebraic closure $\overline{F}$ of $F$ and always work within $\overline{F}$ when we consider algebraic extensions of $F$.
For any finite extension $E$ of $F$, we write $\mcO_{E}$, $\mfp_{E}$, $k_{E}$, and $\Gamma_{E}$ for the ring of integers of $E$, the maximal ideal of $\mcO_{E}$, the residue field, and the absolute Galois group $\Gal(\overline{F}/E)$ of $E$, respectively.
When $E=F$, we omit the subscript $F$ from these symbols. 
For any finite extension $E$ of $F$, we write $W_{E}$, $I_{E}$, and $P_{E}$ for the Weil group of $E$, its inertia subgroup, and its wild inertia subgroup, respectively.
For any $r\in\R_{>0}$, let $I_{F}^{r}$ denote the $r$-th upper ramification filtration of $I_{F}$.
We fix a valuation $\val_{F}$ of $F$ such that $\val_{F}(F^\times)=\Z$.
We extend it to $\overline{F}$ and again write $\val_{F}$ for it.
We define an absolute value $|\cdot|_{\overline{F}}$ of $\overline{F}$ by $|\cdot|_{\overline{F}}\colonequals q^{-\val_{F}(\cdot)}$, where $q$ is the cardinality of the residue field $k_{F}$.

We fix an additive character $\psi_{F}\colon F\rightarrow\C^{\times}$ of level $1$, i.e., $\psi_{F}|_{\mfp}\equiv\mathbbm{1}$ and $\psi_{F}|_{\mcO}\not\equiv\mathbbm{1}$.

\subsubsection{Algebraic varieties and algebraic groups}
In this paper, we use a bold letter for an algebraic variety and use an italic letter for the set of its $F$-valued points when it is defined over $F$.
For example, if $\bfX$ is an algebraic variety defined over $F$, then $X\colonequals \bfX(F)$.

For any algebraic group $\G$, we write $X^{\ast}(\G)$ and $X_{\ast}(\G)$ for the group of characters and cocharacters of $\G$, respectively.
We let $\bfZ_{\G}$ denote the center of $\G$.
When $\G$ is defined over $F$, so is $\bfZ_{\G}$ and the set of its $F$-valued points is denoted by $Z_{\G}$.

For any torus $\bfS$ equipped with an automorphism $\theta_{\bfS}$, we let $\bfS^{\theta_{\bfS}}$ and $\bfS_{\theta_{\bfS}}$ denote the invariant and coinvariant of $\bfS$ with respect to $\theta_{\bfS}$, respectively.

\subsubsection{Centralizers and normalizers}
Suppose that $\G$ is an algebraic group and $\bfX$ is an algebraic variety equipped with left and right actions of $\G$, for which we write $\G\times\bfX\times\G\rightarrow\bfX\colon(g_{1}, x, g_{2})\mapsto g_{1}\cdot x\cdot g_{2}$.
Then we define the conjugate action of $\G$ on $\bfX$ by $\G\times\bfX\rightarrow\bfX\colon (g,x)\mapsto g\cdot x\cdot g^{-1}$.
We introduce the following notation:
\begin{itemize}
\item
For $g\in\G$, let $[g]$ denote the conjugation automorphism $\bfX\rightarrow\bfX\colon x\mapsto g\cdot x\cdot g^{-1}$.
We also often write ${}^{g}x\colonequals [g](x)=g\cdot x\cdot g^{-1}$.
\item
For $x\in\bfX$, let $\G^{x}$ denote the full stabilizer of $x$ in $\G$ with respect to the conjugate action, i.e., $\G^{x}\colonequals \{g\in\G\mid [g](x)=x\}$.
\item
For $x\in\bfX$, let $\G_{x}$ denote the connected stabilizer of $x$ in $\G$ with respect to the conjugate action, i.e., $\G_{x}\colonequals \G^{x,\circ}$.
\end{itemize}
Note that, when $\G$ and $\bfX$ are $F$-rational, $[g]$ is also $F$-rational if $g\in \G(F)$.
Similarly, $\G^{x}$ and $\G_{x}$ are $F$-rational if $x\in\bfX(F)$.

For any subsets $\bfH\subset \bfG$ and $Y\subset \bfX$, we put 
\begin{itemize}
\item
$\bfZ_{\bfH}(Y)\colonequals \{g\in\bfH\mid \text{$[g](y)=y$ for any $y\in Y$}\}$ and 
\item
$\bfN_{\bfH}(Y)\colonequals \{g\in\bfH\mid [g](Y)\subset Y\}$.
\end{itemize}
When $Y$ is a singleton $\{y\}$, we simply write $\bfZ_{\bfH}(y)\colonequals \bfZ_{\bfH}(Y)=\bfN_{\bfH}(Y)$.
If $\G$ and $\bfX$ are defined over $F$ and $Y$ is a subset of $\bfX(F)$, then $\bfZ_{\G}(Y)$ and $\bfN_{\G}(Y)$ are defined over $F$ and the sets of their $F$-valued points are denoted by $Z_{\G}(Y)$ and $N_{\G}(Y)$, respectively.

\subsubsection{Reductive groups}
For any connected reductive group $\G$ and its maximal torus $\bfS$, we let $\Phi(\G,\bfS)$ and $\Phi^{\vee}(\G,\bfS)$ denote the set of roots and coroots of $\bfS$ in $\G$, respectively.
Note that, when both $\G$ and $\bfS$ are defined over $F$, the sets $\Phi(\G,\bfS)$ and $\Phi^{\vee}(\G,\bfS)$ are equipped with an action of $\Gamma$.
We let $\Omega_{\G}(\bfS)$ be the Weyl group of $\bfS$ in $\G$, i.e., $\Omega_{\G}(\bfS)\colonequals \bfN_{\G}(\bfS)/\bfS$.
We sometimes loosely write $\Omega_{\G}$ for $\Omega_{\G}(\bfS)$ when the choice of a maximal torus $\bfS$ is clear from the context (e.g., when $\bfS$ is a maximal torus belonging to a splitting of $\G$).

We write $\bmfg$ for the Lie algebra of $\G$.
When $\G$ is defined over $F$, $\bmfg$ is an algebraic variety over $F$, hence we write $\mfg\colonequals \bmfg(F)$ following the convention explained above.

\subsubsection{Bruhat--Tits theory}

Suppose that $\G$ is a connected reductive group over $F$.
We follow the notation on Bruhat--Tits theory used by \cite{AS08, AS09, DS18}.
(See, for example, \cite[Section 3.1]{AS08} for details.)
Especially, $\mcB(\G,F)$ (resp.\ $\mcB^{\red}(\G,F)$) denotes the enlarged (resp.\ reduced) Bruhat--Tits building of $\G$ over $F$.
For $\x\in\mcB(\G,F)$ (or $\x\in\mcB^{\red}(\G,F)$), we let $G_{\x}$ be the stabilizer of $\x$ in $G$.

We define $\widetilde{\R}$ to be the set $\R\sqcup\{r+\mid r\in\R\}\sqcup\{\infty\}$ with a natural order.
Then, for any $r\in\widetilde{\R}_{\geq0}$ and $\x\in\mcB^{\red}(\G,F)$, we can consider the $r$-th Moy--Prasad filtration $G_{\x,r}$ of $G$ with respect to the point $\x$.
For any $r,s\in\widetilde{\R}_{\geq0}$ satisfying $r<s$, we write $G_{\x,r:s}$ for the quotient $G_{\x,r}/G_{\x,s}$.
We put $G_{r}\colonequals \bigcup_{\x\in\mcB^{\red}(\G,F)}G_{\x,r}$ for $r\in\widetilde{\R}_{\geq0}$.
Similarly, we have the Moy--Prasad filtration $\{\mfg_{\x,r}\}_{r}$ on the Lie algebra $\mfg=\bmfg(F)$, their quotients $\mfg_{\x,r:s}$, and the unions $\mfg_{r}$.
We also have the Moy--Prasad filtration on the dual Lie algebra $\mfg^{\ast}\colonequals \Hom_{F}(\mfg,F)$ defined by 
\[
\mfg^{\ast}_{\x,r}\colonequals \{Y^{\ast}\in\mfg^{\ast} \mid \langle \mfg_{\x,(-r)+},Y^{\ast}\rangle\subset \mfp\}
\]
for any $r\in\R_{\geq0}$ and $\x\in\mcB^{\red}(\G,F)$ ($\mfg^{\ast}_{\x,r+}$ is defined to be $\bigcup_{s>r}\mfg^{\ast}_{\x,s}$).

Suppose that $\bfS$ is a tame (i.e., $F$-rational and split over a tamely ramified extension of $F$) maximal torus of $\G$.
By fixing an $S$-equivariant embedding of $\mcB(\bfS,F)$ into $\mcB(\G,F)$, we may regard $\mcB(\bfS,F)$ as a subset of $\mcB(\G,F)$.
Then, for any point $\x\in\mcB(\G,F)$, the property that ``$\x$ belongs to the image of $\mcB(\bfS,F)$'' does not depend on the choice of such an embedding (see the second paragraph of \cite[Section 3]{FKS23} for details).
For any point $\x\in\mcB(\G,F)$ which belongs to $\mcB(\bfS,F)$, we have $S_{\mathrm{b}}\subset G_{\x}$, where $S_{\mathrm{b}}$ denotes the maximal bounded subgroup of $S$.
When $\bfS$ is elliptic in $\G$, the image of $\mcB(\bfS,F)$ in $\mcB^{\red}(\G,F)$ consists of only one point.
If $\x\in\mcB(\G,F)$ belongs to the image of $\mcB(\bfS,F)$, we say that $\x$ is associated to $\bfS$.

We also fix a family of mock-exponential maps $\mfg_{\x,r}\rightarrow G_{\x,r}$ for $x\in\mcB(\G,F)$ and $r\in\t\R_{>0}$ and simply write ``$\exp$'' for it (see \cite[Appendix A]{AS09}; cf.\ \cite[Section 3.4]{Hak18}).
We write ``$\log$'' for the inverse of $\exp$.
It is guaranteed that a mock exponential map in the sense of \cite[Appendix A]{AS09} always exists when $p$ does not divide the order of the absolute Weyl group of $\G$.

\subsubsection{Finite sets with Galois actions}\label{subsubsec:fin-Gal-sets}
We put $\Sigma\colonequals \Gamma\times\{\pm1\}$.
Suppose that $\Phi$ is a finite set with an action of $\Sigma$, e.g., the set of roots of an $F$-rational maximal torus in a connected reductive group ($-1$ acts on $\Phi$ via $\alpha\mapsto -\alpha$ in this case).
Following \cite{AS09}, we put $\dot{\Phi}\colonequals \Phi/\Gamma$ and $\ddot{\Phi}\colonequals \Phi/\Sigma$.
Also, whenever there is no risk of confusion, we simply write $\dot{\alpha}:=\Gamma\alpha$ and $\ddot{\alpha}:=\Sigma\alpha$.

For each $\alpha\in\Phi$, we put $\Gamma_{\alpha}$ (resp.\ $\Gamma_{\pm\alpha}$) to be the stabilizer of $\alpha$ (resp.\ $\{\pm\alpha\}$) in $\Gamma$.
Let $F_{\alpha}$ (resp.\ $F_{\pm\alpha}$) be the subfield of $\overline{F}$ fixed by $\Gamma_{\alpha}$ (resp.\ $\Gamma_{\pm\alpha}$).
Hence we have $\Gamma_{\alpha}=\Gamma_{F_{\alpha}}$ and $\Gamma_{\pm\alpha}=\Gamma_{F_{\pm\alpha}}$:
\[
F\subset F_{\pm\alpha}\subset F_{\alpha}
\quad
\longleftrightarrow
\quad
\Gamma\supset \Gamma_{\pm\alpha}\supset \Gamma_{\alpha}.
\]
We abbreviate the residue field $k_{F_{\alpha}}$ of $F_{\alpha}$ (resp.\ $k_{F_{\pm\alpha}}$ of $F_{\pm\alpha}$) as $k_{\alpha}$ (resp.\ $k_{\pm\alpha})$.

We say that $\alpha\in\Phi$ is \textit{asymmetric} if $F_{\alpha}=F_{\pm\alpha}$ and that $\alpha$ is \textit{symmetric} if $F_{\alpha}\supsetneq F_{\pm\alpha}$.
We remark that $\alpha$ is symmetric if and only if the $\Gamma$-orbit of $\alpha$ contains $-\alpha$.
By noting that the extension $F_{\alpha}/F_{\pm\alpha}$ is necessarily quadratic if $\alpha$ is symmetric, we say that $\alpha$ is \textit{(symmetric) unramified} (resp.\ \textit{ramified}) if $F_{\alpha}/F_{\pm\alpha}$ is unramified (resp.\ ramified).
We write $\Phi_{\asym}$, $\Phi_{\ur}$, $\Phi_{\ram}$, and $\Phi_{\sym}$ for the set of asymmetric elements, symmetric unramified elements, symmetric ramified elements, and symmetric elements of $\Phi$, respectively.

For $\alpha\in\Phi_{\sym}$, we let $\kappa_{\alpha}\colon F_{\pm\alpha}^{\times}\rightarrow\C^{\times}$ denote the quadratic character of $F_{\pm\alpha}^{\times}$ corresponding to the quadratic extension $F_{\alpha}/F_{\pm\alpha}$ under the local class field theory.

Note that, if $\alpha$ is symmetric, $\dot{\alpha}=\ddot{\alpha}$.
This implies that the sets $\dot{\Phi}_{\sym}$ and $\ddot{\Phi}_{\sym}$ can be naturally identified (and, of course, the same is true for $\Phi_{\ur}$ or $\Phi_{\ram}$).



\subsubsection{Finite fields}
Suppose that $\ul{k}$ is a finite field of odd characteristic $p$.
Then the multiplicative group $\ul{k}^{\times}$ is cyclic of even order, hence there exists a unique nontrivial quadratic character $\ul{k}^{\times}\rightarrow\{\pm1\}$.
We write $\sgn_{\ul{k}^{\times}}(-)$ for this character.

Next, we furthermore suppose that $[\ul{k}:\F_{p}]$ is even.
Then, there uniquely exists a subextension $\ul{k}_{\pm}$ satisfying $[\ul{k}:\ul{k}_{\pm}]=2$.
We let $\ul{k}^{1}$ denote the kernel of the norm map $\Nr_{\ul{k}/\ul{k}_{\pm}}\colon \ul{k}^{\times}\rightarrow\ul{k}_{\pm}^{\times}$.
By noting that $\ul{k}^{1}$ is also cyclic of even order, we write $\sgn_{\ul{k}^{1}}(-)$ for the unique nontrivial quadratic character of $\ul{k}^{1}$.

\subsection{Assumptions on $p$}
In this paper, we consider a connected reductive group $\bfG$ over $F$ equipped with a finite order automorphism $\theta$.
From Section \ref{subsec:approx}, we assume that $p$ is odd and does not divide the order of $\theta$.
In Section \ref{subsec:rsc}, we add the assumption that $p$ does not divide the order of the absolute Weyl group of $\G$.

\section{Twisted spaces}\label{sec:twisted-spaces}

In this section, we first review some basics of twisted spaces, on which the theory of twisted endoscopy is based.
Then, we reproduce the theory of good product expansion by Adler--Spice (\cite{AS08}) in the context of twisted spaces.

\subsection{Twisted spaces}\label{subsec:twisted-spaces}

Recall that the theory of twisted endoscopy (\cite{KS99}) starts with fixing a triple $(\G,\theta,\bfa)$ consisting of 
\begin{itemize}
\item
a connected reductive group $\G$ over $F$,
\item
an $F$-rational quasi-semisimple automorphism $\theta$ of $\G$ (i.e., $\theta$ preserves a Borel pair), and
\item
$\bfa\in H^{1}(W_{F},\bfZ_{\hat{\G}})$, where $\h\G$ is the Langlands dual group of $\G$ over $\C$.
\end{itemize}
In this paper, we focus on the case where $(\G,\theta,\bfa)$ satisfies the following conditions:
\begin{itemize}
\item
$\G$ is a quasi-split group with an $F$-splitting $\spl_{\G}=(\B,\T,\{X_{\alpha}\}_{\alpha})$;
\item
$\theta$ is of finite order (let $l$ be the order of $\theta$) and preserves $\spl_{\G}$;
\item
$\bfa$ is trivial.
\end{itemize}

\begin{example}\label{ex:GL}
We particularly keep the following example in mind.
Let $\G$ be the general linear group $\GL_{n}$ over $F$.
Let $\theta$ be the $F$-rational automorphism of $\G$ defined by
\[
\theta(g)\colonequals  J_{n}{}^{t}\!g^{-1}J_{n}^{-1},
\]
where $J_{n}$ is an anti-diagonal matrix of size $n$ whose $(i,n+1-i)$-th entry is given by $(-1)^{i-1}$ and ${}^{t}\!g$ denotes the transpose of $g$.
Then $\theta$ is involutive and preserves the standard $F$-splitting of $\GL_{n}$.
This is the case considered in Arthur's theory of the endoscopic classification of representations of quasi-split classical (symplectic or special orthogonal) groups (\cite{Art13}).
\end{example}

Following Labesse (\cite{Lab04}) and M{\oe}glin--Waldspurger (\cite{Wal08,MW16-1}), we work with the formalism of \textit{twisted spaces} as follows.
We put 
\[
\t\G\colonequals \G\theta \,(:=\bfG\rtimes\theta),
\]
namely, $\t{\G}$ is the connected component of the disconnected reductive group $\G\rtimes\langle\theta\rangle$ containing the element $1\rtimes\theta$.
(In the following, the symbol ``$\rtimes$'' is often omitted when there is no risk of confusion.)
This is a twisted space in the sense of Labesse, that is, an algebraic variety over $F$ which is a bi-$\G$-torsor.
As an algebraic variety, it is isomorphic to $\G$ by the map $\G\rightarrow\t{\G}\colon g\mapsto g\theta:=g\cdot1\rtimes\theta$.
The right and left actions of $\G$ on $\t\G$ is given by
\[
g_{1}\cdot (g\theta)\cdot g_{2}
=
(g_{1}g\theta(g_{2}))\theta.
\]
Thus the conjugate action of $\G$ on $\t{\G}$ is given by
\[
[g_{1}](g\theta)
\colonequals 
g_{1}\cdot (g\theta)\cdot g_{1}^{-1}
=
(g_{1}g\theta(g_{1})^{-1})\theta.
\]
Note that the $\theta$-twisted conjugacy in $\G$ (as in \cite{KS99}) is amount to the $\G$-conjugacy in $\t\G$.
The conjugate action of $\t\G$ on $\G$ is given by, for $\delta=g\theta\in\t\G$, 
\[
[\delta]
= 
[g]\circ\theta
\colon\G\rightarrow\G.
\]

\subsection{Twisted maximal torus}\label{subsec:twisted-tori}

We next investigate the notion of a twisted maximal torus.

\begin{defn}[{\cite[Section 4.1]{MW18}}]\label{defn:twisted-tori}
Let $(\t\bfS,\bfS)$ be a pair of 
\begin{itemize}
\item an $F$-rational maximal torus $\bfS$ of $\G$ and
\item an $F$-rational $\bfS$-twisted subspace $\t\bfS$ of $\t\G$ (i.e., subvariety of $\t{\G}$ which is a bi-$\bfS$-torsor under the bi-$\bfS$-action on $\t\bfS\subset\t\G$).
\end{itemize}
We say that $(\t\bfS,\bfS)$ is an \textit{$F$-rational twisted maximal torus} of $\t\G$ if the following two conditions are satisfied:
\begin{enumerate}
\item
There exists a Borel subgroup $\bfB_{\bfS}$ of $\G$ (not necessarily $F$-rational) containing $\bfS$ and satisfying $\t\bfS=\bfN_{\t\G}(\bfS)\cap \bfN_{\t\G}(\bfB_{\bfS})$.
\item
The set $\t{S}=\t\bfS(F)$ of $F$-valued points of $\t\bfS$ is not empty.
\end{enumerate}
\end{defn}

By the condition (1) of Definition \ref{defn:twisted-tori}, every $\eta\in\t{\bfS}$ acts on $\bfS$ by the conjugation $[\eta]$.
Since $\t\bfS$ is an $\bfS$-twisted space and $\bfS$ is commutative, this action is independent of the choice of $\eta$.
We let $\theta_{\bfS}$ denote this automorphism of $\bfS$.\symdef{theta-S}{$\theta_{\bfS}$}
Note that we can take $\eta$ to be $F$-rational by the condition (2) of Definition \ref{defn:twisted-tori}, hence $\theta_{\bfS}$ is $F$-rational.

\begin{lem}\label{lem:twist-order}
The order of $\theta_{\bfS}$ is $l$.
\end{lem}

\begin{proof}
Let $\eta\in\t{S}\subset \t{G}=G\theta$.
Since the order of $\theta$ is $l$, we have $\eta^{l}\in G$.
As $\eta$ normalizes $\bfS$ and a Borel subgroup $\bfB_{\bfS}$ containing $\bfS$, so does $\eta^{l}$.
By noting that $\bfS=N_{\G}(\bfS)\cap N_{\G}(\bfB_{\bfS})$, this implies that $\eta^{l}\in S$.
Hence, we get $\theta_{\bfS}^{l}=[\eta]|_{\bfS}^{l}=[\eta^{l}]|_{\bfS}=\id_{\bfS}$.
In other words, the order of $\theta_{\bfS}$ divides $l$.
If the order of $\theta_{\bfS}$ is strictly smaller than $l$, say $l'$, then $[\eta]^{l'}$ is an automorphism of $\bfG$ preserving $(\bfS,\bfB_\bfS)$ and acting on $\bfS$ via the identity.
Thus the outer automorphism determined by $[\eta]^{l'}$ is trivial, in other words, $[\eta]^{l'}$ is an inner automorphism of $\bfG$.
However, this means that $\theta^{l'}$ is also an inner automorphism of $\bfG$.
Since $\theta^{l'}$ preserves a splitting of $\bfG$, $\theta^{l'}$ must be trivial.
Thus we get a contradiction; the order of $\theta_\bfS$ must be $l$.
\end{proof}

\begin{example}
If we let $\tilde{\bfT}:=\bfT\theta$, then $(\tilde{\bfT},\bfT)$ forms an $F$-rational twisted maximal torus of $\t\G$.
Note that $\tilde{\bfT}$ contains $1\rtimes\theta\in \t{G}$ (hence $\theta_\bfT=\theta|_\bfT$), but in general $\tilde{\bfS}$ may not necessarily contain $1\rtimes\theta$.
\end{example}

When $(\bfS,\t\bfS)$ is an $F$-rational twisted maximal torus of $\t\G$, we often simply say that ``$\t\bfS$ is an $F$-rational twisted maximal torus of $\t\G$''.
For an $F$-rational twisted maximal torus $\t\bfS$ of $\t\G$, we put $\bfS^{\nat}\colonequals \bfS^{\theta_{\bfS},\circ}$.\symdef{S-nat}{$\bfS^{\nat}$}
Note that, for any $\eta\in\t\bfS$, we have $\bfS^{\nat}=\bfS_{\eta}\subset\bfG_{\eta}$.
The relationship between $\bfS$ and $\bfS^{\nat}$ is described as follows:

\begin{prop}\label{prop:twisted-tori}
For any $F$-rational twisted maximal torus $\t\bfS$ of $\t\G$, we have
\begin{enumerate}
\item
$\bfZ_{\G}(\bfS^{\nat})=\bfS$,
\item
$\bfZ_{\G}(\t\bfS)^{\circ}=\bfS^{\nat}$, and
\item
for any $\eta\in\t{\bfS}$, $\G_{\eta}$ is a connected reductive group with a maximal torus $\bfS^{\nat}$.
\end{enumerate}
\end{prop}

\begin{proof}
Let $\eta\in\t{\bfS}$.
Let $\B_{\bfS}$ be a Borel subgroup of $\G$ containing $\bfS$ and satisfying $\t\bfS=\bfN_{\t\G}(\bfS)\cap \bfN_{\t\G}(\bfB_{\bfS})$.
Then $[\eta]$ defines an automorphism of $\G$ preserving the Borel pair $(\B_{\bfS},\bfS)$.
We apply Steinberg's result (\cite{Ste68}), which is summarized in \cite[Theorem 1.1.A]{KS99}, to the automorphism $[\eta]$ as follows.
By \cite[Theorem 1.1.A (2)]{KS99}, $\bfS\cap\G_{\eta}$ is a maximal torus of $\G_{\eta}$.
Since $\bfS_{\eta}\subset\bfS\cap\G_{\eta}\subset \bfS^{\eta}$, the connectedness of $\bfS\cap\G_{\eta}$ implies that $\bfS_{\eta}=\bfS\cap\G_{\eta}$.
Thus we get the assertion (3).
Moreover, by \cite[Theorem 1.1.A (4)]{KS99}, we get the assertion (1).

Let us check the assertion (2).
As the inclusion $\bfZ_{\G}(\t\bfS)^{\circ}\supset\bfS^{\nat}$ is obvious, we show the converse inclusion.
Let $g\in \bfZ_{\G}(\t\bfS)^{\circ}$.
Then we have $gs\eta g^{-1}=s\eta$ for any $s\in\bfS$ since $\t\bfS=\bfS\eta$.
Note that, as $\eta\in\t\bfS$, we have $\bfZ_{\G}(\t\bfS)^{\circ}\subset \bfZ_{\G}(\t\eta)^{\circ}=\G_{\eta}$.
Thus $gs\eta g^{-1}=s\eta$ (for any $s\in\bfS$) if and only if $gs g^{-1}=s$ (for any $s\in\bfS$), which implies that $g\in\bfS$.
Hence we get $g\in \bfS\cap\G_{\eta}=\bfS_{\eta}=\bfS^{\nat}$.
\end{proof}

Recall that an element $\delta\in\t\G$ is said to be
\begin{itemize}
\item
\textit{semisimple} if $[\delta]$ is a quasi-semisimple automorphism of $\G$;
\item
\textit{regular semisimple} if $\delta$ is semisimple and $\G_{\delta}$ is a torus; 
\item
\textit{strongly regular semisimple} if $\delta$ is semisimple and $\G^{\delta}$ is abelian
\end{itemize}
(see \cite[Sections 3.2 and 3,3]{KS99}).

Let $\bfA_{\t{\G}}$ denote the maximal split subtorus of $\bfZ_{\G}^{\theta}$.\symdef{A-G-tilde}{$\bfA_{\tilde{\G}}$}
Note that then, for any $F$-rational twisted maximal torus $\t\bfS$ of $\t\G$, $\bfA_{\t\G}$ is contained in $\bfS^{\nat}$.

\begin{defn}\label{defn:elliptic}
\begin{enumerate}
\item
Let $\t\bfS$ be an $F$-rational twisted maximal torus of $\t\G$.
We say that $\t\bfS$ is \textit{elliptic} if $\bfS^{\nat}$ is anisotropic modulo $\bfA_{\t{\G}}$.
\item
For any semisimple element $\delta\in\t{G}$, we say that $\delta$ is \textit{elliptic} if there exists an $F$-rational elliptic twisted maximal torus $\t\bfS$ of $\t\G$ such that $\delta\in\t{S}$.
\end{enumerate}
\end{defn}

\begin{rem}\label{rem:elliptic}
If $(\t\bfS,\bfS)$ is an $F$-rational twisted maximal torus of $\t\G$ whose $\bfS$ is elliptic, then $(\t\bfS,\bfS)$ is elliptic.
Indeed, as we have maps
\[
\xymatrix{
\bfS^{\nat}/\bfA_{\t\G} & \bfS^{\nat}/(\bfZ_{\G}\cap \bfS^{\nat}) \ar@{->>}[l] \ar@{^{(}->}[r]& \bfS/\bfZ_{\G},
}
\]
the ellipticity of $\bfS$ in $\G$ (which is equivalent to the anisotropicity of $\bfS$ modulo $\bfZ_{\G}$) implies that the anisotropicity of $\bfS^{\nat}$ modulo $\bfA_{\t\G}$.
\end{rem}

\begin{lem}\label{lem:twisted-torus}
Let $\bfS$ be an $F$-rational maximal torus of $\G$.
If there exists a semisimple element $\eta\in\t{G}$ and a Borel subgroup $\bfB_{\bfS}$ containing $\bfS$ such that $(\bfB_{\bfS},\bfS)$ is preserved by $[\eta]$, then $(\t\bfS,\bfS)\colonequals (\bfS\eta,\bfS)$ is an $F$-rational twisted maximal torus of $\t\G$.
\end{lem}

\begin{proof}
Since $\t\G=\G\eta$ and $[\eta]$ preserves $(\bfB_{\bfS},\bfS)$, we have
\[
\bfN_{\t\G}(\bfS)\cap \bfN_{\t\G}(\bfB_{\bfS})
=\bigl(\bfN_{\G}(\bfS)\cap \bfN_{\G}(\bfB_{\bfS})\bigr)\eta
=\bfS\eta
=\t\bfS.
\]
Moreover, obviously $\t{S}=\t\bfS(F)$ is not empty as it contains $\eta$.
\end{proof}

\subsection{Steinberg's result on the structure of descended groups}\label{subsec:Steinberg}

Let $\t\bfS$  be an $F$-rational twisted maximal torus of $\t\G$ and $\bfB_{\bfS}$ a Borel subgroup which contains $\bfS$ and is preserved by the conjugate action of $\t{\bfS}$.
By fixing an element $g_{\bfS}\in\G$ satisfying $[g_{\bfS}](\bfB_{\bfS},\bfS)=(\bfB,\T)$, we get an isomorphism $[g_{\bfS}]\colon(\t\bfS,\bfS)\xrightarrow{\sim}(\t\T,\T)$ (which may not be defined over $F$).
Note that the isomorphism $[g_{\bfS}]\colon\bfS\xrightarrow{\sim}\T$ is independent of the choice of $g_{\bfS}\in\G$ and that the automorphism $\theta_{\bfS}$ of $\bfS$ is transported to $\theta_\bfT=\theta$ on $\T$ via $[g_{\bfS}]$, i.e., $\theta\circ[g_{\bfS}]=[g_{\bfS}]\circ\theta_{\bfS}$.

For any $\eta\in\t{S}$, its connected centralizer $\G_{\eta}$ is a connected reductive group with a maximal torus $\bfS^{\nat}$ by Proposition \ref{prop:twisted-tori}.
In this subsection, we review some facts about the structure of the root system $\Phi(\G_{\eta},\bfS^{\nat})$ following \cite[Section 3.3]{Wal08}.

We put 
\begin{itemize}
\item
$Y^{\ast}(\T)\colonequals X^{\ast}(\T)/(X^{\ast}(\T)\cap(1-\theta)X^{\ast}(\T)_{\Q})$ and
\item
$Y_{\ast}(\T)\colonequals X_{\ast}(\T)/(X_{\ast}(\T)\cap(1-\theta)X_{\ast}(\T)_{\Q})$.
\end{itemize}
We write $p^{\ast}\colon X^{\ast}(\T)\rightarrow Y^{\ast}(\T)$ and $p_{\ast}\colon X_{\ast}(\T)\rightarrow Y_{\ast}(\T)$ for the natural surjections.
Then we have the following:
\begin{enumerate}
\item
$Y^{\ast}(\T)\cong X^{\ast}(\T^{\nat})$ is the $\Z$-dual to $X_{\ast}(\T)^{\theta}\cong X_{\ast}(\T^{\nat})$;
\item
$Y_{\ast}(\T)$ is the $\Z$-dual to $X^{\ast}(\T)^{\theta}$.
\end{enumerate}

We put $\Theta\colonequals \langle\theta\rangle$.
Note that the action of $\Theta$ on $(\G,\T)$ induces an action on $\Phi(\G,\T)$ given by $\alpha\mapsto \theta(\alpha):=\alpha\circ\theta^{-1}$.
For any $\alpha\in\Phi(\G,\T)$, we let $l_{\alpha}$ be the cardinality of the $\Theta$-orbit of $\alpha$ in $\Phi(\G,\T)$ and define an element $N(\alpha)\in X^\ast(\bfT)$ by
\[
N(\alpha)\colonequals \sum_{i=0}^{l_{\alpha}-1} \theta^{i}(\alpha).
\]
We also define $l_{\alpha^{\vee}}$ and $N(\alpha^{\vee})$ for any $\alpha^{\vee}\in\Phi^{\vee}(\G,\T)$ in the same manner.
For $\alpha\in\Phi(\G,\T)$, we shortly write $\alpha_{\res}\colonequals p^{\ast}(\alpha)$.
We define a set $\Phi_{\res}(\G,\T)$ by
\[
\Phi_{\res}(\G,\T)
=\{p^{\ast}(\alpha) \mid \alpha\in \Phi(\G,\T)\}\subset Y^{\ast}(\T)\cong X^{\ast}(\T^{\nat}).
\]\symdef{Phi-res-G-T}{$\Phi_{\res}(\G,\T)$}
Then $\Phi_{\res}(\G,\T)$ forms a (possibly non-reduced) root system.
We call elements of $\Phi_{\res}(\G,\T)$ \textit{restricted roots}.
Following \cite[Section 1.3]{KS99}, we say that $\alpha\in\Phi(\G,\T)$ (or its associated $\alpha_{\res}$) is of
\begin{itemize}
\item
\textit{type $1$} if $2\alpha_{\res},\frac{1}{2}\alpha_{\res}\notin\Phi_{\res}(\G,\T)$,
\item
\textit{type $2$} if $2\alpha_{\res}\in\Phi_{\res}(\G,\T)$,
\item
\textit{type $3$} if $\frac{1}{2}\alpha_{\res}\in\Phi_{\res}(\G,\T)$.
\end{itemize}
We put
\[
\varrho_{\alpha}
\colonequals 
\begin{cases}
1 & \text{if $\alpha$ is of type $1$ or $3$,}\\
2 & \text{if $\alpha$ is of type $2$,}
\end{cases}
\quad
\varsigma_{\alpha}
\colonequals 
\begin{cases}
1 & \text{if $\alpha$ is of type $1$ or $2$,}\\
-1 & \text{if $\alpha$ is of type $3$.}
\end{cases}
\]
We also define a set $\Phi_{\res}^{\vee}(\G,\T)$ by
\[
\Phi_{\res}^{\vee}(\G,\T)
=\{\varrho_{\alpha}\cdot N(\alpha^{\vee}) \mid \alpha\in \Phi^{\vee}(\G,\T)\}\subset X_{\ast}(\T)^{\theta}\cong X_{\ast}(\T^{\nat}).
\]
Then we have bijections
\begin{align*}
\Phi(\G,\T)/\Theta&\xrightarrow{1:1}\Phi_{\res}(\G,\T)\colon \alpha\mapsto \alpha_{\res} \,(\colonequals p^{\ast}(\alpha)),\\
\Phi^{\vee}(\G,\T)/\Theta&\xrightarrow{1:1}\Phi_{\res}^{\vee}(\G,\T)\colon \alpha^\vee\mapsto \varrho_{\alpha}\cdot N(\alpha^\vee).
\end{align*}
(The sets $\Phi_{\res}(\G,\T)$ and $\Phi_{\res}^{\vee}(\G,\T)$ are denoted by $\Sigma^{\res}$ and $\check{\Sigma}^{\res}$ in \cite[Section 3.3]{Wal08}, respectively.)

\begin{rem}[{\cite[(1.3.3)]{KS99}}]\label{rem:A_{2n}}
There exists a restricted root of type $2$ or $3$ only when $\Phi(\G,\T)$ contains an irreducible component of Dynkin type $A_{2n}$ which is preserved and acted by some power of $\theta$ nontrivially.
\end{rem}

Now let $\eta$ be an element of $\t{S}$ and let $\nu\in\T$ be the element such that
\[
[g_{\bfS}](\eta)=\nu\theta\in\t\T=\T\theta.
\]
Then $[g_{\bfS}]\colon \G\rightarrow\G$ induces an isomorphism between $(\G_{\eta},\bfS^{\nat})$ and $(\G_{\nu\theta},\T^{\nat})$.
In particular, the sets $\Phi(\G_{\eta},\bfS^{\nat})$ and $\Phi^{\vee}(\G_{\eta},\bfS^{\nat})$ can be identified with $\Phi(\G_{\nu\theta},\T^{\nat})$ and $\Phi^{\vee}(\G_{\nu\theta},\T^{\nat})$, respectively (note that here we ignore the Galois actions).
The latter sets are described in terms of the restricted roots and coroots as follows:
\begin{align*}
\Phi(\G_{\nu\theta},\T^{\nat})
&=\{p^{\ast}(\alpha) \mid \alpha\in \Phi(\G,\T); N(\alpha)(\nu)=\varsigma_{\alpha}\}
\subset\Phi_{\res}(\G,\T),\\
\Phi^{\vee}(\G_{\nu\theta},\T^{\nat})
&=\{\varrho_{\alpha}\cdot N(\alpha^{\vee}) \mid \alpha^{\vee}\in \Phi^{\vee}(\G,\T); N(\alpha)(\nu)=\varsigma_{\alpha}\}
\subset\Phi_{\res}^{\vee}(\G,\T).
\end{align*}
Note that these sets are thought of as subsets of $X^{\ast}(\T^{\nat})$ and $X_{\ast}(\T^{\nat})$.

\begin{rem}
Since the fixed element $g_{\bfS}$ is not canonical, $\nu$ is not determined by $\eta$ canonically.
However, $g_{\bfS}$ is unique up to $\T$-translation, i.e., any other $g'_{\bfS}\in\G$ satisfying $[g'_{\bfS}](\bfS,\bfB_{\bfS})=(\T,\bfB)$ is necessarily given by $g'_{\bfS}=tg_{\bfS}$ for some $t\in\T$.
If we use $g'_{\bfS}$ instead of $g_{\bfS}$, then $[g'_{\bfS}](\eta)$ is given by $t(\nu\theta)t^{-1}$.
Thus, if we write $\nu'$ for the corresponding element of $\T$ determined by $\eta$ and $[g'_{\bfS}]$, then we have $\nu'=t\nu\theta(t)^{-1}$.
In particular, we have $N(\alpha)(\nu)=N(\alpha)(\nu')$ for any $\alpha\in\Phi(\G,\T)$.
Therefore the description of the sets $\Phi(\G_{\nu\theta},\T^{\nat})$ and $\Phi^{\vee}(\G_{\nu\theta},\T^{\nat})$ is determined only by $\eta$.
\end{rem}

\subsection{Good product expansion in twisted spaces}\label{subsec:approx}

We discuss a twisted version of the theory of good product expansion due to Adler and Spice (\cite{AS08}).

We first recall the definition of a good product expansion of elements of $p$-adic groups in the untwisted setting.
We temporarily let $\G$ be any tamely ramified connected reductive group over $F$.
Let $\bar{\G}$ be the quotient $\G/\bfZ_{\G}^{\circ}$.

\begin{defn}[{\cite[Definitions 4.11 and 6.1]{AS08}}]\label{defn:good-element}
\begin{enumerate}
\item
We say that an element $\gamma\in G$ is \textit{good of depth $0$} if $\gamma$ is semisimple and its image $\bar{\gamma}$ in $\bar{G}$ is \textit{absolutely semisimple}, i.e., every character value of $\bar{\gamma}$ (see \cite[Definition A.4]{AS08}) is of finite prime-to-$p$ order.
\item
For $d\in\R_{>0}$, an element $\gamma\in G$ is said to be \textit{good of depth $d$} if there exists an $F$-rational tame-modulo-center torus $\bfS$ in $\G$ such that
\begin{itemize}
\item
$\gamma\in S_{d}\smallsetminus S_{d+}$, and
\item
for every $\alpha\in\Phi(\G,\bfS)$, $\alpha(\gamma)=1$ or $\val_{F}(\alpha(\gamma)-1)=d$.
\end{itemize}
\end{enumerate}
\end{defn}

\begin{defn}[{\cite[Definition 6.4]{AS08}}]\label{defn:good-seq}
For $r\in\widetilde{\R}$, a sequence $\underline{\gamma}=(\gamma_{i})_{0\leq i<r}$ of elements of $G$ indexed by real numbers $0\leq i <r$ is called a \textit{good sequence} if
\begin{itemize}
\item
$\gamma_{i}$ is $1$ or a good of depth $i$ for each $0\leq i<r$, and
\item
there exists an $F$-rational tame torus $\bfS$ of $\G$ such that $\gamma_{i}\in S$ for every $0\leq i<r$.
\end{itemize}
\end{defn}

To a good sequence $\underline{\gamma}=(\gamma_{i})_{0\leq i<r}$, we associate subgroups of $\G$ to $\underline{\gamma}$ as follows:
\[
\bfC_{\G}^{(r)}(\underline{\gamma})
\colonequals 
\Bigl(\bigcap_{0\leq i<r} \bfZ_{\G}(\gamma_{i})\Bigr)^{\circ},
\quad
\bfZ_{\G}^{(r)}(\underline{\gamma})
\colonequals 
\bfZ_{\bfC_{\G}^{(r)}(\underline{\gamma})}.
\]
We write $C_{\G}^{(r)}(\underline{\gamma}):=\bfC_{\G}^{(r)}(\underline{\gamma})(F)$ and $Z_{\G}^{(r)}(\underline{\gamma})=\bfZ_{\G}^{(r)}(\underline{\gamma})(F)$.

\begin{defn}[{\cite[Definition 6.8]{AS08}}]\label{defn:approx}
For $\gamma\in G$, a good sequence $\underline{\gamma}=(\gamma_{i})_{0\leq i<r}$ $(r\in \widetilde{\R}_{>0})$ is called an \textit{$r$-approximation to $\gamma$} if there exists a point $\x\in\mcB(\bfC_{\G}^{(r)}(\underline{\gamma}),F)$ satisfying $\gamma\in(\prod_{0\leq i<r}\gamma_{i})G_{\x,r}$.
When we have $\gamma\in C_{\G}^{(r)}(\underline{\gamma})$, we say that $\underline{\gamma}$ is a \textit{normal $r$-approximation to $\gamma$}.
\end{defn}

When we have a normal $r$-approximation $\underline{\gamma}$ to $\gamma$, we put
\[
\gamma_{<r}\colonequals \prod_{0\leq i<r}\gamma_{i},\quad
\gamma_{\geq r}\colonequals \gamma\cdot\gamma_{<r}^{-1}
\]\symdef{gamma-less-r}{$\gamma_{<r}$}\symdef{gamma-greater-equal-r}{$\gamma_{\geq r}$}
and simply say that ``$\gamma=\gamma_{<r}\cdot\gamma_{\geq r}$ is a normal $r$-approximation.''
Note that $\gamma_{\geq r}$ commutes with $\gamma_{<r}$ when $\gamma=\gamma_{<r}\cdot\gamma_{\geq r}$ is a normal $r$-approximation.

Now we move on to the setting of twisted spaces.
Let $(\G,\theta)$ be as in Section \ref{subsec:twisted-spaces}.
In particular, we have the associated twisted space $\t{\G}=\G\theta$.
From now on, we assume that 
\[
\textbf{$p$ is odd and prime to $l$ (the order of $\theta$).}
\]
We put $\G^{\dagger}\colonequals \G\rtimes\langle\theta\rangle$.
Recall that $\bfA_{\t{\G}}$ is the maximal split subtorus of $\bfZ_{\G}^{\theta}$.\symdef{G-dagger}{$\G^\dagger$}
To extend the theory of Adler--Spice to $\G^{\dagger}$, we appeal to the notion of a topological Jordan decomposition.

\begin{defn}[{\cite[Definition 1.6]{Spi08}}]\label{defn:Jordan}
For $\gamma\in G^{\dagger}$, we say that a pair $(\gamma_{0},\gamma_{+})$ of elements of $G^{\dagger}$ is a \textit{topological $p$-Jordan decomposition modulo $A_{\t{\G}}$} of $\gamma$ if 
\begin{itemize}
\item
$\gamma=\gamma_{0}\gamma_{+}=\gamma_{+}\gamma_{0}$, 
\item
$\gamma_{0}$ is absolutely $p$-semisimple modulo $A_{\t{\G}}$, i.e., the image $\bar{\gamma}_{0}$ of $\gamma_{0}$ in $G^{\dagger}/A_{\t{\G}}$ is of finite prime-to-$p$ order, and
\item
$\gamma_{+}$ is topologically $p$-unipotent modulo $A_{\t{\G}}$, i.e., the image $\bar{\gamma}_{+}$ of $\gamma_{+}$ in $G^{\dagger}/A_{\t{\G}}$ satisfies $\lim_{n\rightarrow\infty}\bar{\gamma}_{+}^{p^{n}}=1$.
\end{itemize}
\end{defn}
\symdef{gamma-0}{$\gamma_{0}$}\symdef{gamma-plus}{$\gamma_{+}$}

In this paper, we refer to a topological $p$-Jordan decomposition modulo $A_{\t{\G}}$ simply as a \textit{topological Jordan decomposition}.
Similarly, when an element $\gamma$ is absolutely $p$-semisimple modulo $A_{\t\G}$ (resp.\ topologically $p$-unipotent modulo $A_{\t\G}$), we often simply say that $\gamma$ is topologically semisimple (resp.\ topologically unipotent) as long as there is no risk of confusion.

\begin{rem}\label{rem:Jordan}
Note that, for a given element $\gamma\in G^{\dagger}$, its topological Jordan decomposition is unique modulo $A_{\t{\G}}$ if it exists.
More precisely, if both $(\gamma_{0},\gamma_{+})$ and $(\gamma'_{0},\gamma'_{+})$ are topological Jordan decompositions of $\gamma$, then we have $\bar{\gamma}_{0}=\bar{\gamma}'_{0}$ and $\bar{\gamma}_{+}=\bar{\gamma}'_{+}$ in $G^{\dagger}/A_{\t{\G}}$.
\end{rem}

We say that an elliptic semisimple element $\delta$ of $\t{G}$ is \textit{tame} if there exists an $F$-rational elliptic twisted maximal torus $(\t\bfS,\bfS)$ such that $\bfS$ is tame and $\delta\in\t{S}$.
We note that if $\G$ is tamely ramified and $p$ does not divide the order of the absolute Weyl group of $\G$, any $F$-rational maximal torus of $\G$ is tame by \cite[Theorem 3.3]{Fin21-IMRN}.
In particular, any elliptic semisimple element $\delta$ of $\t{G}$ is tame.

\begin{lem}\label{lem:pi_{0}}
The order of $\pi_{0}(\bfS^{\theta_\bfS})$ is divided only by prime factors of $l$.
\end{lem}

\begin{proof}
Since $\pi_0(\bfS^{\theta_\bfS})$ is finite, it is enough to investigate the order of $X^{\ast}(\pi_0(\bfS^{\theta_\bfS}))$, which is the torsion part of $X^{\ast}(\bfS^{\theta_\bfS})$.
From the left exact sequence
\[
1 \rightarrow \bfS^{\theta_\bfS} \rightarrow \bfS \xrightarrow{1-\theta_\bfS} \bfS,
\]
where the map $1-\theta_\bfS$ is given by $s\mapsto s\cdot\theta_\bfS(s)^{-1}$, we get a right exact sequence
\[
X^{\ast}(\bfS)\xrightarrow{1-\theta_\bfS^\ast} X^{\ast}(\bfS) \rightarrow X^{\ast}(\bfS^{\theta_\bfS}) \rightarrow 0.
\]
We let $n$ be the rank of $\bfS$ and fix an identification $X^{\ast}(\bfS)\cong\Z^{\oplus n}$.
The above right exact sequence shows that $X^{\ast}(\bfS^{\theta_\bfS})$ is the cokernel of the homomorphism $1-\theta_\bfS^{\ast}\colon\Z^{\oplus n} \rightarrow \Z^{\oplus n}$.

Let us fix any prime number $\ell$ which does not divide $l$ and show that $\Cok(1-\theta_\bfS^{\ast})$ has no $\ell$-torsion.
For this, it suffices to show that $\Cok(1-\theta_\bfS^{\ast})\otimes_{\Z} R\cong\Cok(1-\theta_\bfS^{\ast}\colon R^{\oplus n} \rightarrow R^{\oplus n})$ has no $\ell$-torsion for $R=\Z_{(\ell)}[\mu_l]$, where $\Z_{(\ell)}$ denotes the localization of $\Z$ at $\ell$ and $\mu_l$ is the set of $l$-th roots of unity (note that $R$ is flat).
Since the order of $\theta_\bfS^{\ast}$ is $l$, all the eigenvalues of $\theta_\bfS^\ast$ are contained in $\mu_l$, hence $\theta_\bfS^{\ast}$ can be diagonalized in $R$.
If we let $e_1,\ldots,e_n$ denote the eigenvalues of $\theta_\bfS^{\ast}$, then we have $\Cok(1-\theta_\bfS^{\ast})\cong\bigoplus_{i=1}^{n}R/(1-e_i)R$.
As $l$ is invertible in $\Z_{(\ell)}$, $1-e_i$ is also invertible in $R$ unless $e_i=1$.
This implies that $\Cok(1-\theta_\bfS^{\ast})\otimes_{\Z} R$ is free.
\end{proof}

\begin{prop}\label{prop:Jordan}
Let $\delta$ be a tame elliptic semisimple element of $\t{G}$.
Then there exists a pair $(\delta_{0},\delta_{+})$ of elements of $G^{\dagger}$ such that
\begin{enumerate}
\item
$\delta=\delta_{0}\delta_{+}=\delta_{+}\delta_{0}$,
\item
$\delta_{0}\in\t{G}\subset G^{\dagger}$ is topologically semisimple,
\item
$\delta_{+}\in G_{\delta_{0},0+}\subset G^{\dagger}$ is topologically unipotent, and
\item
$\delta_{0}$ and $\delta_{+}$ belong to the closure $\overline{\langle\delta\rangle A_{\t{\G}}}$ of $\langle\delta\rangle A_{\t{\G}}$ in $G^{\dagger}$.
\end{enumerate}
In particular, $(\delta_{0},\delta_{+})$ is a topological Jordan decomposition of $\delta$.
\end{prop}\symdef{delta-0}{$\delta_{0}$}\symdef{delta-plus}{$\delta_{+}$}

\begin{proof}
Since $\delta$ is elliptic semisimple, there exists an $F$-rational elliptic twisted maximal torus $\t\bfS$ of $\t{\G}$ such that $\delta\in\t{S}$ by definition (Definition \ref{defn:elliptic}).
Since the associated automorphism $\theta_{\bfS}$ of $\bfS$ can be thought of as the conjugation by $\delta$ on $\bfS$, the element $\delta^{l}$ of $S$ is fixed by $\theta_{\bfS}$.
In other words, $\delta^{l}$ belongs to $S^{\theta_{\bfS}}$.
Thus, by putting $k\colonequals |\pi_{0}(\bfS^{\theta_{\bfS}})|$, we see that $\delta^{lk}\in S^{\nat}$ (recall that $\bfS^{\nat}\colonequals \bfS^{\theta_{\bfS},\circ}$).

Since $\t{\bfS}$ is elliptic, $\bfS^{\nat}$ is anisotropic modulo $\bfA_{\t{\G}}$, hence $S^{\nat}$ is compact modulo $A_{\t{\G}}$.
In particular, $\delta^{lk}$ has a topological Jordan decomposition in $S^{\nat}/A_{\t{\G}}$ by \cite[Proposition 1.8]{Spi08}.
More precisely, if we write $\bar{\delta}'$ for the image of $\delta^{lk}$ in $S^{\nat}/A_{\t{\G}}$, then we have two elements $\bar{\delta}'_{0}$ and $\bar{\delta}'_{+}$ of $S^{\nat}/A_{\t{\G}}$ such that 
\begin{itemize}
\item
$\bar{\delta}'=\bar{\delta}'_{0}\bar{\delta}'_{+}=\bar{\delta}'_{+}\bar{\delta}'_{0}$,
\item
$\bar{\delta}'_{0}$ is of finite prime-to-$p$ order,
\item
$\bar{\delta}'_{+}$ is topologically $p$-unipotent.
\end{itemize}

We put $\bar{\bfS}^{\nat}\colonequals \bfS^{\nat}/\bfA_{\t{\G}}$.
We note that the image of $\bar{\delta}'_{+}$ under the natural injection 
\[
S^{\nat}/A_{\t{\G}} \hookrightarrow \bar{S}^{\nat}=(\bfS^{\nat}/\bfA_{\t{\G}})(F)
\]
belongs to the pro-unipotent radical $\bar{S}^{\nat}_{0+}$ of the unique parahoric subgroup $\bar{S}^{\nat}_{0}$ of $\bar{S}^{\nat}$.
Indeed, by \cite[Lemma 2.21]{Spi08}, the topological $p$-unipotency of $\bar{\delta}'_{+}$ implies the topological $F$-unipotency of $\bar{\delta}'_{+}$ in the sense of \cite[Definition 2.15]{Spi08}, that is, $\bar{\delta}'_{+}$ belongs to $\bar{\bfS}^{\nat}(E)_{0+}$ for the splitting field $E$ of $\bar{\bfS}^{\nat}$.
Since the torus $\bar{\bfS}^{\nat}$ is tame, we have $\bar{\bfS}^{\nat}(E)_{0+}\cap \bar{S}^{\nat}=\bar{S}^{\nat}_{0+}$ (see \cite[Proposition 2.2]{Yu01}).

By applying \cite[Lemma 3.1.4 (2)]{Kal19} to the short exact sequence
\[
1
\rightarrow\bfA_{\t{\G}}
\rightarrow\bfS^{\nat}
\rightarrow\bar{\bfS}^{\nat}
\rightarrow1,
\]
we see that the map $S^{\nat}_{0+}\rightarrow\bar{S}^{\nat}_{0+}$ is surjective.
Thus we can take an element $\delta'_{+}\in S^{\nat}_{0+}$ mapping to $\bar{\delta}'_{+}\in \bar{S}^{\nat}_{0+}$.

Now note that $lk$ is prime to $p$ by the assumption on $\theta$ and Lemma \ref{lem:pi_{0}}.
Thus we can take $\delta_{+}\in S^{\nat}_{0+}$ satisfying $\delta_{+}^{lk}=\delta'_{+}$ (i.e., an ``$lk$-th root'' of $\delta'_{+}$) as follows.
Let $a\in\Z_{>0}$ be a positive integer such that $p^{a}\equiv1 \pmod{lk}$.
Then the topological $p$-unipotency of $\delta'_{+}$ implies that the sequence
\[
\delta'^{\frac{p^{a}-1}{lk}}_{+},\quad
\delta'^{\frac{p^{2a}-1}{lk}}_{+},\quad
\delta'^{\frac{p^{3a}-1}{lk}}_{+},\ldots
\]
is Cauchy, hence converges.
If we let $\delta_{+}^{-1}\in S^{\nat}_{0+}$ be the limit of this sequence, then we have $\delta_{+}^{lk}=\delta'_{+}$.

We put $\delta_{0}\colonequals \delta\cdot\delta_{+}^{-1}$.
Then obviously $\delta_{0}$ belongs to $\t{S}\subset\t{G}$.
Since $\delta_{+}$ belongs to $S^{\nat}$, $\delta_{+}$ commutes with $\delta_{0}$.
Moreover, by the construction, $\delta_{0}^{lk}$ belongs to $S^{\nat}$ and its image in $S^{\nat}/A_{\t{\G}}$ is of finite prime-to-$p$ order.
Hence, again by noting that $lk$ is prime to $p$, we conclude that $\delta_{0}$ is absolutely $p$-semisimple modulo $A_{\t{\G}}$.
Finally, in order to check that $\delta_{0}, \delta_{+}\in \overline{\langle\delta\rangle A_{\t{\G}}}$, it suffices to show only $\delta_{+}\in \overline{\langle\delta\rangle A_{\t{\G}}}$.
Since the similar statement holds for $\bar{\delta}'_{+}$, namely, $\bar{\delta}'_{+}\in \overline{\langle\bar{\delta}'\rangle}\subset S^{\nat}/A_{\t{\G}}$, we have $\delta'_{+}\in \overline{\langle\delta^{lk}\rangle A_{\t{\G}}}$.
By the construction of $\delta_{+}$, this implies that $\delta_{+}\in \overline{\langle\delta\rangle A_{\t{\G}}}$.
\end{proof}

In the rest of this paper, for a tame elliptic semisimple element $\delta$ of $\t{G}$, we call a decomposition as in Proposition \ref{prop:Jordan} a \textit{topological Jordan decomposition of $\delta$}.

\begin{defn}\label{defn:twisted-approx}
Let $\delta\in\t{G}$ be a tame elliptic semisimple element.
A \textit{normal $r$-approximation to $\delta$} ($r\in\widetilde{\R}_{>0}$) is a pair $(\delta=\delta_{0}\delta_{+}, \underline{\delta}_{+})$ of 
\begin{itemize}
\item
a topological Jordan decomposition $\delta=\delta_{0}\delta_{+}$ of $\delta$ and
\item
a normal $r$-approximation $\underline{\delta}_{+}=(\delta_{i})_{0<i<r}$ to $\delta_{+}$ in $G_{\delta_{0}}$.
\end{itemize}
\end{defn}

For a normal $r$-approximation $(\delta=\delta_{0}\delta_{+}, \underline{\delta}_{+})$ to a tame elliptic semisimple element $\delta\in\t{G}$, we put
\[
\delta^{+}_{<r}\colonequals \prod_{0<i<r}\delta_{i},\quad
\delta_{<r}\colonequals \prod_{0\leq i<r}\delta_{i},\quad
\delta_{\geq r}\colonequals \delta_{<r}^{-1}\delta,
\]\symdef{delta-plus-less-r}{$\delta^{+}_{<r}$}\symdef{delta-less-r}{$\delta_{<r}$}\symdef{delta-greater-equal-r}{$\delta_{\geq r}$}
\[
\bfC_{\G}^{(r)}(\delta)
\colonequals 
\bfC_{\G_{\delta_{0}}}^{(r)}(\ul{\delta}_{+}).
\]
When $(\delta=\delta_{0}\delta_{+}, \underline{\delta}_{+})$ is a normal $r$-approximation to $\delta$, we often simply write ``$\delta=\delta_{0}\delta^{+}_{<r}\delta_{\geq r}$ is a normal $r$-approximation to $\delta$''.

\begin{lem}\label{lem:rss-descent}
Let $\delta$ and $(\delta_{0},\delta_{+})$ be as in Proposition \ref{prop:Jordan}.
If $\delta$ is elliptic regular semisimple in $\t{\G}$, then so is $\delta_{+}$ in $\G_{\delta_{0}}$.
\end{lem}

\begin{proof}
Let $\t\bfS$ be an $F$-rational elliptic twisted maximal torus of $\t\G$ containing $\delta$.
Then, by Proposition \ref{prop:twisted-tori} (3), $\bfS^{\nat}$ is a maximal torus of $\G_{\delta}$.
Since $\G_{\delta}$ is a torus by the regular semisimplicity of $\delta$, this implies that $\G_{\delta}=\bfS^{\nat}$.
As we have $(\G_{\delta_{0}})_{\delta_{+}}\subset \G_{\delta}=\bfS^{\nat}$, we see that $(\G_{\delta_{0}})_{\delta_{+}}=\bfS^{\nat}$ and $\delta_{+}$ is regular semisimple in $\G_{\delta_{0}}$.

Let us check that $\bfS^{\nat}$ is elliptic in $\G_{\delta_{0}}$.
Again by Proposition \ref{prop:twisted-tori} (3), $\bfS^{\nat}$ is a maximal torus of $\G_{\delta_{0}}$.
Since $\t\bfS$ is elliptic, $\bfS^{\nat}$ is anisotropic modulo $\bfA_{\t{\G}}$.
As $\bfA_{\t{\G}}$ is contained in the center of $\G_{\delta_{0}}$, the maximal torus $\bfS^{\nat}$ of $\G_{\delta_{0}}$ is elliptic.
\end{proof}

\begin{prop}\label{prop:twisted-Jordan}
Suppose that $\G$ is tamely ramified and $p$ does not divide the order of the absolute Weyl group $\Omega_{\G}$ of $\G$.
Then any elliptic semisimple element $\delta\in\t{G}$ has a normal $r$-approximation.
\end{prop}

\begin{proof}
Since the existence of a topological Jordan decomposition of $\delta$ is guaranteed by Proposition \ref{prop:Jordan}, we only have to show that $\delta_{+}$ has a normal $r$-approximation in $G_{\delta_{0}}$.
By \cite[Lemma 8.1]{AS08}, any bounded-modulo-$Z_{\G_{\delta_{0}}}$ element of $G_{\delta_{0}}$ belonging to an $F$-rational tame maximal torus of $\G_{\delta_{0}}$ has a normal $r$-approximation as long as the assumption ``$(\mathbf{Gd}^{G_{\delta_{0}}})$'' is satisfied (see \cite[Definition 6.3]{AS08}).
Since $\delta_{+}$ is bounded-modulo-$Z_{\G_{\delta_{0}}}$, it is enough to show that $\delta_{+}$ is contained in an $F$-rational tame maximal torus of $\G_{\delta_{0}}$ and that $(\mathbf{Gd}^{G_{\delta_{0}}})$ is satisfied.

As remarked before, the assumption on $p$ implies that any $F$-rational maximal torus of $\bfG$ is tame, hence we can find an $F$-rational elliptic twisted maximal torus $(\t\bfS,\bfS)$ such that $\delta\in\t{S}$.
This implies that $\bfS^{\nat}$ is an $F$-rational tame maximal torus of $\G_{\delta_{0}}$.
By construction, $\delta_{+}$ belongs to $S^{\nat}$.
As $p$ does not divide the order of the absolute Weyl group $\Omega_{\G_{\delta_{0}}}$ of $\G_{\delta_{0}}$ (note that this is a subgroup of $\Omega_{\G}$; see \cite[Section 1.1]{KS99}), the assumption $(\mathbf{Gd}^{G_{\delta_{0}}})$ is satisfied by \cite[Theorem 3.6]{Fin21-IMRN}.
\end{proof}

\begin{lem}\label{lem:head}
Let $\delta\in\t{G}$ be a tame elliptic semisimple element having a normal $r$-approximation $\delta=\delta_{0}\delta^{+}_{<r}\delta_{\geq r}$.
Then $\delta_{0}$ belongs to $\overline{\langle\delta_{<r}\rangle A_{\t{\G}}}$.
\end{lem}

\begin{proof}
We let $\bar{\delta}_{0}$ denote the image of $\delta_{0}$ in $G^{\dagger}/A_{\t{\G}}$.
Let $p'$ be the order of $\bar{\delta}_{0}$, which is prime to $p$.
If we take $k\in\Z_{>0}$ such that $p^{k}\equiv1 \pmod{p'}$, then we have $\bar{\delta}_{0}^{p^{k}}=\bar{\delta}_{0}$.
Hence, for any $n\in\Z_{>0}$, we have $\bar{\delta}_{0}^{p^{nk}}=\bar{\delta}_{0}^{p^{(n-1)k}}=\cdots=\bar{\delta}_{0}$.
Since $\delta^{+}_{<r}$ is topologically $p$-unipotent and commutes with $\delta_{0}$, we have
\[
(\bar{\delta}_{<r})^{p^{nk}}
=(\bar{\delta}_{0})^{p^{nk}}\cdot(\bar{\delta}^{+}_{<r})^{p^{nk}}
=\bar{\delta}_{0}\cdot(\bar{\delta}^{+}_{<r})^{p^{nk}}
\xrightarrow[n\rightarrow\infty]{}\bar{\delta}_{0}.
\]
Thus $\bar{\delta}_{0}$ belongs to $\overline{\langle\bar{\delta}_{<r}\rangle}\subset G^{\dagger}/A_{\t{\G}}$.
It can be easily checked that this implies that $\delta_{0}$ belongs to $\overline{\langle\delta_{<r}\rangle A_{\t{\G}}}\subset\t{G}$.
\end{proof}

\begin{lem}\label{lem:cent-cent}
Let $\delta\in\t{G}$ be a tame elliptic semisimple element having a normal $r$-approximation $\delta=\delta_{0}\delta^{+}_{<r}\delta_{\geq r}$.
Then we have
\[
(\G_{\delta_{0}})_{\delta^{+}_{<r}}
=
\G_{\delta_{<r}}
\]
\end{lem}

\begin{proof}
The statement can be proved by a similar argument to the untwisted case (cf.\ \cite[Corollary 6.14]{AS08}) as we explain in the following.

We may take an $F$-rational tame twisted maximal torus $\t{\bfS}$ containing $\delta$, $\delta_{0}$, and $\delta_{<r}$.
Indeed, as $\delta_{+}=\delta_{<r}^{+}\delta_{\geq r}$ is a normal $r$-approximation in $\G_{\delta_{0}}$, $\delta_{+}$ belongs to $(\G_{\delta_{0}})_{\delta_{<r}^{+}}$, hence we can find an $F$-rational tame maximal torus $\bfS'$ of $(\G_{\delta_{0}})_{\delta_{<r}^{+}}$ containing $\delta_{+}$ (note that $\delta_{+}$ is semisimple).
Then $\bfS'$ contains $\delta^{+}_{<r}$ and $\delta_{+}$.
By Steinberg's result (see \cite[Theorem 1.1.A]{KS99}), $\bfS\colonequals \bfZ_{\G}(\bfS')$ gives an $F$-rational tame maximal torus of $\G$.
Moreover, $\bfS$ is $[\delta_{0}]$-stable and there exists an $[\delta_{0}]$-stable Borel subgroup containing $\bfS$.
Hence, by Lemma \ref{lem:twisted-torus}, $\t\bfS\colonequals \bfS\delta_{0}$ is an $F$-rational tame twisted maximal torus of $\t{\G}$.
(Note that then $\bfS'=\bfS^{\nat}$.)
By construction, we have $\delta,\delta_{0},\delta_{<r}\in\t{S}$.

Since $\delta_{<r}=\delta_{0}\delta^{+}_{<r}$, the inclusion $(\G_{\delta_{0}})_{\delta^{+}_{<r}}\subset\G_{\delta_{<r}}$ is obvious.
Let us show that this inclusion is in fact the equality.
As $\bfS^{\nat}$ is a maximal torus both in $(\G_{\delta_{0}})_{\delta^{+}_{<r}}$ and $\G_{\delta_{<r}}$, it suffices to check that $\Phi((\G_{\delta_{0}})_{\delta^{+}_{<r}}, \bfS^{\nat})$ equals $\Phi(\G_{\delta_{<r}}, \bfS^{\nat})$.
By taking $g_{\bfS}\in\G$ as in Section \ref{subsec:Steinberg}, we put 
\[
\nu_{<r}\theta\colonequals {}^{g_{\bfS}}\delta_{<r}\in\T\theta,\quad
\nu_{0}\theta\colonequals {}^{g_{\bfS}}\delta_{0}\in\T\theta,\quad
\nu^{+}_{<r}\colonequals {}^{g_{\bfS}}\delta^{+}_{<r} \in \T^{\nat}.
\]\symdef{nu-less-r}{$\nu_{<r}$}\symdef{nu-0}{$\nu_{0}$}\symdef{nu-plus-less-r}{$\nu^{+}_{<r}$}
Then $\Phi((\G_{\delta_{0}})_{\delta^{+}_{<r}}, \bfS^{\nat})$ and $\Phi(\G_{\delta_{<r}}, \bfS^{\nat})$ are identified with $\Phi((\G_{\nu_{0}\theta})_{\nu^{+}_{<r}}, \T^{\nat})$ and $\Phi(\G_{\nu_{<r}\theta}, \T^{\nat})$, respectively.
By the description explained in Section \ref{subsec:Steinberg}, we have 
\[
\Phi(\G_{\nu_{0}\theta},\T^{\nat})
=\{\alpha_{\res} \mid \alpha\in \Phi(\G,\T); N(\alpha)(\nu_{0})=\varsigma_{\alpha}\},
\]
hence
\[
\Phi((\G_{\nu_{0}\theta})_{\nu^{+}_{<r}},\T^{\nat})
=\{\alpha_{\res} \mid \alpha\in \Phi(\G,\T); \text{$N(\alpha)(\nu_{0})=\varsigma_{\alpha}$ and $\alpha_{\res}(\nu^{+}_{<r})=1$}\}.
\]
On the other hand, we have
\[
\Phi(\G_{\nu_{<r}\theta},\T^{\nat})
=\{\alpha_{\res} \mid \alpha\in \Phi(\G,\T); N(\alpha)(\nu_{<r})=\varsigma_{\alpha}\}.
\]

Thus our task is to check that, for any $\alpha\in\Phi(\G,\T)$ satisfying $N(\alpha)(\nu_{<r})=\varsigma_{\alpha}$, we have $N(\alpha)(\nu_{0})=\varsigma_{\alpha}$ and $\alpha_{\res}(\nu^{+}_{<r})=1$.
Let $\alpha\in\Phi(\G,\T)$ be such a root.
By the definition of $l_{\alpha}$, we have $\sum_{i=0}^{l-1}\theta^{i}(\alpha)=\frac{l}{l_{\alpha}}\sum_{i=0}^{l_{\alpha}-1}\theta^{i}(\alpha)$.
Thus we get
\begin{align*}
N(\alpha)(\nu_{<r})^{\frac{l}{l_{\alpha}}}
&=
\biggl(\sum_{i=0}^{l_{\alpha}-1}\theta^{i}(\alpha)\biggr)(\nu_{<r})^{\frac{l}{l_{\alpha}}}\\
&=
\biggl(\sum_{i=0}^{l-1}\theta^{i}(\alpha)\biggr)(\nu_{<r})
=
\alpha\biggl(\prod_{i=0}^{l-1}\theta^{i}(\nu_{<r})\biggr)
=
\alpha\bigl((\nu_{<r}\theta)^{l}\bigr).
\end{align*}
Hence, noting that $(\nu_{<r}\theta)^{l}=(\nu_{0}\theta)^{l}\cdot \nu_{<r}^{+,l}$, we have
\[
\varsigma_{\alpha}^{\frac{l}{l_{\alpha}}}
=
\alpha\bigl((\nu_{<r}\theta)^{l}\bigr)
=
\alpha\bigl((\nu_{0}\theta)^{l}\bigr)\cdot\alpha(\nu^{+}_{<r})^{l}.
\]
Since $\delta_{0}$ is of finite prime-to-$p$ order modulo $A_{\t{\G}}$ and $\delta^{+}_{<r}$ is topologically $p$-unipotent, we see that $\alpha((\nu_{0}\theta)^{l})\in\overline{F}^{\times}$ is of finite prime-to-$p$ order (note that $A_{\t{\G}}$ is killed by any root) and $\alpha(\nu^{+}_{<r})^{l}\in\overline{F}^{\times}$ is topologically $p$-unipotent.
Thus, by noting that $\varsigma_{\alpha}\in\{\pm1\}$ and $p\neq2$, we must have $\alpha(\nu^{+}_{<r})^{l}=1$, which furthermore implies that $\alpha(\nu^{+}_{<r})=1$ since $\alpha(\nu^{+}_{<r})$ is topologically unipotent.
Then we get 
\[
\varsigma_{\alpha}=N(\alpha)(\nu_{<r})=N(\alpha)(\nu_{0})\cdot N(\alpha)(\nu^{+}_{<r})=N(\alpha)(\nu_{0}).
\]
\end{proof}

\section{Regular supercuspidal representations}\label{sec:rsc}

\subsection{Regular supercuspidal representations}\label{subsec:rsc}
From now on, we furthermore assume that
\begin{itemize}
\item
$\G$ is tamely ramified over $F$, and
\item
$p$ does not divide the order of the absolute Weyl group $\Omega_{\G}$ of $\G$.
\end{itemize}
In \cite{Yu01}, Yu introduced the notion of a \textit{cuspidal $\G$-datum} and attached an irreducible supercuspidal representation of $G$ to each cuspidal $\G$-datum.
Recall that a cuspidal $\G$-datum is a quintuple $\Sigma=(\vec{\G},\vec{\vartheta},\vec{r},\x,\rho_{0})$ consisting of the following objects (here we follow the convention of \cite[Section 3.1]{HM08}):
\begin{itemize}
\item
$\vec{\G}$ is a sequence $\G^{0}\subsetneq\G^{1}\subsetneq\cdots\subsetneq\G^{d}=\G$ of tame Levi subgroups such that $\bfZ_{\G^{0}}/\bfZ_{\G}$ is anisotropic,
\item
$\x$ is a vertex of the reduced Bruhat--Tits building $\mcB^{\red}(\G^{0},F)$ of $\G^{0}$,
\item
$\vec{r}$ is a sequence $0\leq r_{0}<\cdots<r_{d-1}\leq r_{d}$ such that $0<r_{0}$ when $d>0$,
\item
$\vec{\vartheta}$ is a sequence $(\vartheta_{0},\ldots,\vartheta_{d})$ of characters $\vartheta_{i}$ of $G^{i}$ satisfying
\begin{itemize}
\item
for $0\leq i<d$, $\vartheta_{i}$ is $\G^{i+1}$-generic of depth $r_{i}$ at $\x$, and
\item
for $i=d$, $\vartheta_{d}$ is of depth of $r_d$ at $\x$ if $r_{d-1}<r_{d}$ and $\vartheta_{d}=\mathbbm{1}$ if $r_{d-1}=r_{d}$.
\end{itemize}
\item
$\rho_{0}$ is an irreducible representation of $G^{0}_{\x}$ whose restriction to $G^{0}_{\x,0}$ contains the inflation of a cuspidal representation of the quotient $G^{0}_{\x,0:0+}$.
\end{itemize}
We call the representations obtained from cuspidal $\G$-data by Yu's construction \textit{tame supercuspidal representations}.

The ``fibers'' of Yu's construction were investigated by Hakim--Murnaghan; in \cite{HM08}, they introduced an equivalence relation called \textit{$\G$-equivalence} and proved that two cuspidal $\G$-data give rise to isomorphic supercuspidal representations if and only if two data are $\G$-equivalent.

In \cite{Kal19}, Kaletha introduced the notion of the \textit{(extra) regularity} for cuspidal $\G$-data (see \cite[Section 3]{Kal19}).
Tame supercuspidal representations arising from (extra) regular cuspidal $\G$-data are called \textit{(extra) regular supercuspidal representations}.
Kaletha discovered that (extra) regular cuspidal $\G$-data can be parametrized by much simpler data called \textit{tame elliptic (extra) regular pairs} (\cite[Proposition 3.7.8]{Kal19}).
If we write $\pi_{(\bfS,\vartheta)}$ for the (extra) regular supercuspidal representation which corresponds to a tame elliptic (extra) regular pair $(\bfS,\vartheta)$, then the situation is summarized as follow:\symdef{pi-S-vartheta}{$\pi_{(\bfS,\vartheta)}$}
\[
\xymatrix{
\{\text{cusp.\ $\G$-data}\}/\text{$\G$-eq.} \ar@{->}[r]^-{1:1}_-{\text{Yu's constr.}} & \{\text{tame s.c.\ rep'ns of $G$}\}/{\sim}\\
\{\text{(ex.) reg.\ cusp.\ $\G$-data}\}/\text{$\G$-eq.}\ar@{}[u]|{\bigcup} \ar@{->}[r]^-{1:1} & \{\text{(ex.) reg.\ s.c.\ rep'ns of $G$}\}/{\sim}\ar@{}[u]|{\bigcup}\\
\{\text{tame ell.\ (ex.) reg.\ pairs}\}/\text{$G$-conj.}\ar@{<->}[u]^-{1:1} \ar@{->}_-{\quad\qquad(\bfS,\vartheta)\mapsto\pi_{(\bfS,\vartheta)}}[ur]&
}
\]

Here, let us recall the definition of a tame elliptic (extra) regular pair:
\begin{defn}[{\cite[Definition 3.7.5]{Kal19}}]\label{defn:ter-pair}
A \textit{tame elliptic regular} (resp.\ \textit{extra regular}) \textit{pair} is a pair $(\bfS,\vartheta)$ consisting of
\begin{itemize}
\item
a tame elliptic $F$-rational maximal torus $\bfS$ of $\G$ and
\item
a character $\vartheta\colon S\rightarrow\C^{\times}$
\end{itemize}
satisfying the following conditions:
\begin{enumerate}
\item
By choosing a finite tamely ramified extension $E$ of $F$ splitting $\bfS$, we put
\[
\Phi_{0+}
\colonequals 
\{\alpha\in \Phi(\G,\bfS) \mid (\vartheta\circ\Nr_{E/F}\circ\alpha^{\vee})|_{E_{0+}^{\times}}\equiv\mathbbm{1} \}.
\]
Then the action of $I_{F}$ on $\Phi_{0+}$ preserves a set of positive roots.
\item
We put $\G^{0}$ to be the tame Levi subgroup of $\G$ with maximal torus $\bfS$ and root system $\Phi_{0+}$.
Then $\vartheta|_{S_{0}}$ has trivial stabilizer for the action of $N_{\G^{0}}(\bfS)/S$ (resp.\ $\Omega_{\G^{0}}(\bfS)(F)$).
\end{enumerate}
\end{defn}

\subsection{Toral supercuspidal representations}\label{subsec:toral-sc}
We next give more detailed explanation of Yu's construction in the case of \textit{toral supercuspidal representations}, which will be of our main interest.

\begin{defn}\label{defn:toral}
We say that a cuspidal $\G$-datum $\Sigma=(\vec{\G},\vec{\vartheta},\vec{r},\x,\rho_{0})$ is \textit{toral} if it satisfies the following:
\begin{itemize}
\item
$\vec{\G}=(\bfG^{0}=\bfS\subset\G=\bfG^{1})$, where $\bfS$ is a tame elliptic maximal torus of $\G$,
\item
$0<r_{0}=r_{1}\,(=:r)$,
\item
$\vec{\vartheta}=(\vartheta,\mathbbm{1})$, where $\vartheta$ is a $\G$-generic character of $S$ of depth $r$,
\item
$\rho_{0}$ is the trivial representation $\mathbbm{1}$.
\end{itemize}
We call a tame supercuspidal representation associated to a toral cuspidal $\G$-datum a \textit{toral supercuspidal representation}.
\end{defn}

\begin{rem}
We caution that, in some literature, the terminology ``toral'' only means that $\G^{0}$ is a torus (i.e., $d$ can be greater than $1$).
For example, in \cite{FS21}, they distinguish these two versions of torality by calling the one of Definition \ref{defn:toral} the ``$0$-torality''.
We use ``toral'' rather than ``$0$-toral'' in this paper following \cite{DS18} and \cite{Kal19}.
\end{rem}

Under the bijection of \cite[Proposition 3.7.8]{Kal19} mentioned above, a tame elliptic regular pair corresponding to a toral cuspidal $\G$-datum $((\bfS\subset\G), (r=r), (\vartheta,\mathbbm{1}),\x,\mathbbm{1})$ is simply given by $(\bfS,\vartheta)$.
Let us call a tame elliptic regular pair obtained in this way a \textit{tame elliptic toral pair}.
We note that the torality implies the extra regularity.

In the following, we fix a tame elliptic toral pair $(\bfS,\vartheta)$.
Let $\x\in\mcB^{\red}(\G,F)$ be the point associated to $\bfS$ and $r\in\R_{>0}$ be the depth of $\vartheta$.
We put $s\colonequals r/2$ and define the subgroups $K$, $J$, and $J_{+}$ of $G$ by
\[
K\colonequals  SG_{\x,s},\quad
J\colonequals  (S,G)_{\x,(r,s)},\quad
J_{+}\colonequals (S,G)_{\x,(r,s+)}, 
\]\symdef{K}{$K$}\symdef{J}{$J$}\symdef{J-plus}{$J_{+}$}
where $(S,G)_{\x,(r,s)}$ and $(S,G)_{\x,(r,s+)}$ are the groups defined according to the manner of \cite[Sections 1 and 2]{Yu01}.
Note that we have $K=SJ$.

Since the depth of $\vartheta$ is $r$, we can extend $\vartheta$ to a character $\hat{\vartheta}$ of $J_{+}$ satisfying $\hat{\vartheta}|_{(S,G)_{\x,(r+,s+)}}\equiv\mathbbm{1}$.
Then, by the definition of the $\G$-genericity, there exists an element $X^{\ast}$\symdef{X-ast}{$X^\ast$} of $\mfs^{\ast}_{-r}$ which is $\G$-generic of depth $r$ in the sense of \cite[Section 8]{Yu01} and satisfies
\[
\hat{\vartheta}(\exp(Y))
=
\psi_{F}(\langle Y,X^{\ast}\rangle)
\]\symdef{hat-vartheta}{$\hat{\vartheta}$}
for any $Y\in\mfg_{\x,s+:r+}$ (or, equivalently, for any $Y\in\mfs_{s+:r+}$).
Here, as explained in \cite[Section 8]{Yu01}, we may regard $\bmfs^{\ast}$ as a subspace of $\bmfg^{\ast}$ by considering the coadjoint action of $\bfS$ on $\bmfg^{\ast}$.
We recall that the definition of $\G$-genericity consists of two conditions \textbf{GE1} and \textbf{GE2}.
The condition \textbf{GE1} requires that $\val_{F}(\langle H_{\alpha},X^{\ast}\rangle)=-r$ for any $\alpha\in\Phi(\G,\bfS)$, where $H_{\alpha}\colonequals d\alpha^{\vee}(1)$.
We do not review the condition \textbf{GE2} because \textbf{GE1} implies \textbf{GE2} by \cite[Lemma 8.1]{Yu01} when $p$ is not a torsion prime for the dual based root datum of $\G$.
(Recall that we have assumed the $p\nmid|\Omega_{\G}|$, which is equivalent to $p\nmid|\Omega_{\h\G}|$. In fact, this implies that $p$ is not a torsion prime for the dual based root datum of $\G$; see \cite[Lemma 3.2]{Fin21-IMRN}.)

The point of the construction is that, by putting $N\colonequals \Ker\hat{\vartheta}\subset J_{+}$, the quotient $J/N$ has the structure of a finite Heisenberg group:
\begin{itemize}
\item
The center of $J/N$ is given by $J_{+}/N$, which is isomorphic to $\mu_{p}\cong\F_{p}$ via $\hat{\vartheta}$ (here we fix an isomorphism between $\mu_{p}\subset\C^\times$ and $\F_{p}$).
\item
The quotient $J/J_{+}$ has a symplectic structure with respect to the pairing
\[
(J/J_{+})\times(J/J_{+}) \rightarrow \mu_{p}\cong\F_{p}\colon (g,g')\mapsto\h\vartheta([g,g'])
\]
(see \cite[Section 11]{Yu01}; we will review the structure of the symplectic space $J/J_{+}$ in more detail in Section \ref{subsec:Heisen-structure}).
\end{itemize}
Therefore, by the Stone--von Neumann theorem, there exists a unique irreducible representation of $J/N$ whose central character on $J_{+}/N$ is given by $\hat{\vartheta}$.
Furthermore, as the conjugate action of $S$ on $J$ preserves $J_{+}$ and $N$ and induces a symplectic action on $J/J_{+}$, we can extend the representation of $J/N$ to the semi-direct group $S\ltimes J$, for which we write $\omega_{(\bfS,\vartheta)}$ (so-called the \textit{Heisenberg--Weil representation}).\symdef{omega-S-vartheta}{$\omega_{(\bfS,\vartheta)}$}
Then the tensor representation $\omega_{(\bfS,\vartheta)}\otimes(\vartheta\ltimes\mathbbm{1})$ of $S\ltimes J$ descends to $SJ=K$, that is, factors through the canonical map $S\ltimes J\twoheadrightarrow K$.
We let $\rho_{(\bfS,\vartheta)}$ be the descended representation of $K$.\symdef{rho-S-vartheta}{$\rho_{(\bfS,\vartheta)}$}
The toral supercuspidal representation $\pi_{(\bfS,\vartheta)}$ is given by 
\[
\pi_{(\bfS,\vartheta)}\colonequals \cInd_{K}^{G}\rho_{(\bfS,\vartheta)}.
\]

We also recall the definitions of a few more groups and representations which will be needed later (for the proof of the character formula, in Sections \ref{sec:TCF1} and \ref{sec:HW-twisted}):
\[
K_{\sigma}
\colonequals 
SG_{\x,0+},\quad
\sigma_{(\bfS,\vartheta)}
\colonequals 
\Ind_{K}^{K_{\sigma}}\rho_{(\bfS,\vartheta)},
\]\symdef{K-sigma}{$K_{\sigma}$}\symdef{sigma-S-vartheta}{$\sigma_{(\bfS,\vartheta)}$}
\[
\tau_{(\bfS,\vartheta)}
\colonequals 
\Ind_{K}^{G_{\x}} \rho_{(\bfS,\vartheta)}\,
(\cong \Ind_{K_{\sigma}}^{G_{\x}} \sigma_{(\bfS,\vartheta)}).
\]\symdef{tau-S-vartheta}{$\tau_{(\bfS,\vartheta)}$}
The situation is summarized in the following diagram:
\[
\xymatrix@=10pt{
& G_{\x}\\ 
& K_\sigma=SG_{\x,0+}\ar@{}[u]|{\bigcup} \\
S\rtimes J \ar@{->>}[r]& K=SG_{\x,s}\ar@{}[u]|{\bigcup} \\
}
\quad\quad
\xymatrix@=10pt{
&& \tau_{(\bfS,\vartheta)}\\ 
&& \sigma_{(\bfS,\vartheta)}\ar_-{\Ind}[u] \\
\omega_{(\bfS,\vartheta)}\otimes(\vartheta\rtimes\mathbbm{1}) \ar^-{\text{descent}}[rr]&& \rho_{(\bfS,\vartheta)}\ar_-{\Ind}[u] \\
}
\]

\section{Twisted Adler--DeBacker--Spice formula: preliminary form}\label{sec:TCF1}

In this and the next sections, we discuss a twisted version of the character formula of Adler--DeBacker--Spice for toral supercuspidal representations (\cite{AS09,DS18}).
Our arguments heavily depend on the work \cite{AS08, AS09,DS18}.
We note that several technical assumptions on $p$ are required so that the theory of Adler--DeBacker--Spice works, but it is enough to assume only the oddness and the non-badness of $p$ (for the root system of $\G$ in the sense of \cite[I.4.1]{SS70}, see also \cite[Section A]{AS08}) whenever $\G$ is tamely ramified by \cite[Section 4.1]{Kal19}. 
Recall that we have assumed that $p$ is odd and does not divide the order of the absolute Weyl group $\Omega_{\G}$ of $\G$; this implies the non-badness of $p$ (\cite[Lemma 3.2]{Fin21-IMRN}).

Let $\t{\bfG}$ be a twisted space as before.
We continue to assume that $p$ does not divide the order $l$ of $\theta$.

\subsection{Twisted character of a $\theta$-stable representation}\label{subsec:TC}
Let us first review the basics of twisted characters of irreducible admissible representations.
See \cite[Section 2.6]{LH17} for more details.

Let $\eta\in\t{G}$.
Then $[\eta]$ is an $F$-rational automorphism of $\G$.
Recall that, for an irreducible admissible representation $\pi$ of $G$ realized on a $\C$-vector space $V$, its \textit{$\eta$-twist} $\pi^{\eta}$ is defined by the action
\[
\pi^{\eta}(g)\colonequals  \pi\circ[\eta](g)=\pi(\eta g\eta^{-1})
\]
on the same representation space $V$.
We say that $\pi$ is \textit{$\eta$-stable} if $\pi^{\eta}$ is isomorphic to $\pi$ as a representation of $G$.

\begin{rem}\label{rem:twist}
Note that, if we write $\eta=\eta^{\circ}\theta$ with an element $\eta^{\circ}\in G$, then we have $[\eta]=[\eta^{\circ}]\circ\theta$.
As $[\eta^{\circ}]$ does not change the isomorphism class of any representation, $\pi$ is $\eta$-stable if and only if $\pi$ is $\theta$-stable.
More explicitly, $\pi(\eta^{\circ})$ gives an intertwiner between $\pi^{\theta}$ and $\pi^{\eta}$, i.e., $\pi^{\eta}(g)\circ\pi(\eta^{\circ})=\pi(\eta^{\circ})\circ\pi^{\theta}(g)$ for any $g\in G$.
\end{rem}

Suppose that $\pi$ is an $\eta$-stable irreducible admissible representation of $G$.
We fix an intertwiner 
\[
I_{\pi}^{\eta}\colon\pi\xrightarrow{\sim}\pi^{\eta}
\]\symdef{I-pi-eta}{$I_{\pi}^{\eta}$}
(note that such an $I_{\pi}^{\eta}$ is unique up to $\C^{\times}$-multiple, as $\pi$ is irreducible)
and put
\[
\t{\pi}(g\eta)\colonequals \pi(g)\circ I_{\pi}^{\eta}
\]
for any $g\eta\in\t{G}$ with $g\in G$.
Then we get a representation $\t{\pi}$ of $\t{G}$ on the representation space $V$ of $\pi$, i.e., a map $\t{\pi}\colon\t{G}\rightarrow\Aut_{\C}(V)$ satisfying the following relation for any $g_{1},g_{2}\in G$ and $\delta\in\t{G}$:
\[
\t{\pi}(g_{1}\cdot\delta\cdot g_{2})=\pi(g_{1})\circ\t{\pi}(\delta)\circ\pi(g_{2}).
\]

For any $f\in C_{c}^{\infty}(\t{G})$, an operator $\t{\pi}(f)$ on $V$ is defined by
\[
\t{\pi}(f)\colonequals  \int_{\t{G}} f(\delta) \t{\pi}(\delta) \, d\delta,
\]
where $d\delta$ is a measure on $\t{G}$ obtained by transferring a Haar measure $dg$ on $G$ by the bijection $G\rightarrow\t{G}\colon g\mapsto g\eta$.
Then, as in the untwisted case, the operator $\t{\pi}(f)$ is of finite rank and hence we can define its trace.
Moreover, the resulting distribution $f\mapsto \tr \t{\pi}(f)$ on $C_{c}^{\infty}(\t{G})$ is known to be locally $L^1$.
In this setting, the \textit{($\eta$-)twisted character} $\Theta_{\t{\pi}}$ of $\pi$ is defined to be the unique locally constant function on $\t{G}_{\rs}$ such that 
\[
\tr \t{\pi}(f)
=
\int_{\t{G}_{\rs}} \Theta_{\t{\pi}}(\delta) f(\delta)\,d\delta
\]\symdef{Theta-pi-tilde}{$\Theta_{\tilde{\pi}}$}
for every $f \in C_{c}^{\infty}(\t{G})$ satisfying $\supp(f) \subseteq \t{G}_{\rs}$, where $\t{G}_{\rs}$ denotes the set of regular semisimple elements of $\t{G}$.

\begin{rem}\label{rem:twist-char}
\begin{enumerate}
\item
We emphasize that the twisted representation $\t{\pi}$ and the twisted character $\Theta_{\t{\pi}}$ depend on the choice of an intertwiner $I_{\pi}^{\eta}$ between $\pi$ and $\pi^{\eta}$ although this dependence is not reflected to the symbol $\t{\pi}$.
For any $c\in\C^{\times}$, $cI_{\pi}^{\eta}\colonequals (c\cdot\id_{V})\circ I_{\pi}^{\eta}$ is again an intertwiner between $\pi$ and $\pi^{\eta}$.
If we define a twisted representation by using $cI_{\pi}^{\eta}$ (let us write $c\t{\pi}$ for it), then its twisted character is simply given by $\Theta_{c\t{\pi}}=c\cdot\Theta_{\t{\pi}}$.
\item
As mentioned in Remark \ref{rem:twist}, for any $\eta=\eta^{\circ}\theta\in\t{G}$, $\pi$ is $\eta$-stable if and only if it is $\theta$-stable.
When $I_{\pi}^{\theta}$ is an intertwiner between $\pi$ and $\pi^{\theta}$, then $I_{\pi}^{\eta}\colonequals \pi(\eta^{\circ})\circ I_{\pi}^{\theta}$ gives an intertwiner between $\pi$ and $\pi^{\eta}$.
If we write $\t{\pi}[I_{\pi}^{\eta}]$ and $\t{\pi}[I_{\pi}^{\theta}]$ for the twisted representations of $\t{G}$ obtained from $\pi$ by using these two intertwiners $I_{\pi}^{\eta}$ and $I_{\pi}^{\theta}$, then we can easily check that $\t{\pi}[I_{\pi}^{\eta}]=\t{\pi}[I_{\pi}^{\theta}]$.
In particular, we have $\Theta_{\t{\pi}[I_{\pi}^{\eta}]}=\Theta_{\t{\pi}[I_{\pi}^{\theta}]}$.
\end{enumerate}
\end{rem}

\subsection{Twist and intertwiner}\label{subsec:twist-int}
Let $\eta\in\t{G}$.
Only in this subsection, for any subgroup $H$ of $G$, we let $H^{\eta}$ denote its conjugate $[\eta]^{-1}(H)=\eta^{-1}H\eta$.
We caution that this usage of notation is temporary; in other places of this paper, the superscript symbol $(-)^{\eta}$ denotes the stabilizer of $\eta$.

For any tame elliptic toral pair $(\bfS,\vartheta)$ of $\G$, its $\eta$-twist
\[
(\bfS,\vartheta)^{\eta}
\colonequals 
([\eta]^{-1}(\bfS),\vartheta\circ[\eta])
\] 
is again a tame elliptic toral pair of $\G$.
Thus we have the toral regular supercuspidal representation $\pi_{(\bfS,\vartheta)^{\eta}}$ associated to $(\bfS,\vartheta)^{\eta}$.
On the other hand, we also have the $\eta$-twist $\pi_{(\bfS,\vartheta)}^{\eta}$ of the toral regular supercuspidal representation $\pi_{(\bfS,\vartheta)}$ associated to $(\bfS,\vartheta)$.
In fact, these representations are isomorphic.
Let us investigate how we can construct an intertwiner $\pi_{(\bfS,\vartheta)^{\eta}}\cong\pi_{(\bfS,\vartheta)}^{\eta}$.

Recall from Section \ref{subsec:rsc} that $\pi_{(\bfS,\vartheta)}$ is defined to be the compact induction $\cInd_{K}^{G}\rho_{(\bfS,\vartheta)}$ of a representation $\rho_{(\bfS,\vartheta)}$ of an open compact-mod-center subgroup $K=SJ$ of $G$.
Hence the $\eta$-twisted representation $\pi_{(\bfS,\vartheta)}^{\eta}$ is isomorphic to the compact induction of $\rho_{(\bfS,\vartheta)}^{\eta}$ from $K^{\eta}$ to $G$ by the following explicit intertwiner:
\[
\cInd_{K^{\eta}}^{G}\rho_{(\bfS,\vartheta)}^{\eta}\xrightarrow{\sim}(\cInd_{K}^{G}\rho_{(\bfS,\vartheta)})^{\eta}=\pi_{(\bfS,\vartheta)}^{\eta}\colon f\mapsto f\circ[\eta]^{-1}.
\tag{1}
\]
On the other hand, we can easily see that the open compact-mod-center subgroup associated to $\eta$-twisted pair $(\bfS,\vartheta)^{\eta}$ is given by $K^{\eta}=S^{\eta}J^{\eta}$.
Thus $\pi_{(\bfS,\vartheta)^{\eta}}$ is given by the compact induction $\cInd_{K^{\eta}}^{G}\rho_{(\bfS,\vartheta)^{\eta}}$ of a representation $\rho_{(\bfS,\vartheta)^{\eta}}$ of $K^{\eta}$.

Let us consider the relationship between the representations $\rho_{(\bfS,\vartheta)}^{\eta}$ and $\rho_{(\bfS,\vartheta)^{\eta}}$ of $K^{\eta}$.
The representation $\rho_{(\bfS,\vartheta)}$ of $K$ is defined to be the push-out of the representation $\omega_{(\bfS,\vartheta)}\otimes(\vartheta\ltimes\mathbbm{1})$ of $S\ltimes J$ along the natural multiplication map $S\ltimes J\twoheadrightarrow SJ$.
Hence $\rho_{(\bfS,\vartheta)}^{\eta}$ is the push-out of $\omega_{(\bfS,\vartheta)}^{\eta}\otimes(\vartheta^{\eta}\ltimes\mathbbm{1})$ along $S^{\eta}\ltimes J^{\eta}\twoheadrightarrow S^{\eta}J^{\eta}$.
On the other hand, $\rho_{(\bfS,\vartheta)^{\eta}}$ is defined to be the push-out of $\omega_{(\bfS,\vartheta)^{\eta}}\otimes(\vartheta^{\eta}\ltimes\mathbbm{1})$ along $S^{\eta}\ltimes J^{\eta}\twoheadrightarrow S^{\eta}J^{\eta}$.
We note that both of $\omega_{(\bfS,\vartheta)}^{\eta}$ and $\omega_{(\bfS,\vartheta)^{\eta}}$ are Heisenberg--Weil representations with central character $\h\vartheta^{\eta}$.
Hence, by the Stone--von Neumann theorem, $\omega_{(\bfS,\vartheta)}^{\eta}$ and $\omega_{(\bfS,\vartheta)^{\eta}}$ are guaranteed to be isomorphic.
Let us fix an intertwiner
\[
I_{\omega_{(\bfS,\vartheta)}}^{\eta}\colon\omega_{(\bfS,\vartheta)^{\eta}}\xrightarrow{\sim}\omega_{(\bfS,\vartheta)}^{\eta},
\]\symdef{I-omega-eta}{$I_{\omega_(\bfS,\vartheta)}^{\eta}$}
which naturally induces an intertwiner
\[
I_{\rho_{(\bfS,\vartheta)}}^{\eta}\colon\rho_{(\bfS,\vartheta)^{\eta}}\xrightarrow{\sim}\rho_{(\bfS,\vartheta)}^{\eta}.
\]\symdef{I-rho-eta}{$I_{\rho_(\bfS,\vartheta)}^{\eta}$}
Then we get an intertwiner between $\cInd_{K^{\eta}}^{G}\rho_{(\bfS,\vartheta)^{\eta}}$ and $\cInd_{K^{\eta}}^{G}\rho_{(\bfS,\vartheta)}^{\eta}$ given by 
\[
\cInd_{K^{\eta}}^{G}\rho_{(\bfS,\vartheta)^{\eta}}\xrightarrow{\sim}\cInd_{K^{\eta}}^{G}\rho_{(\bfS,\vartheta)}^{\eta}\colon f\mapsto I_{\rho_{(\bfS,\vartheta)}}^{\eta}\circ f.
\tag{2}
\]

Therefore, combining (1) with (2), we obtain an intertwiner $I_{\pi_{(\bfS,\vartheta)}}^{\eta}$ between $\pi_{(\bfS,\vartheta)^{\eta}}$ and $\pi_{(\bfS,\vartheta)}^{\eta}$ given by $f\mapsto I_{\rho_{(\bfS,\vartheta)}}^{\eta}\circ f\circ [\eta]^{-1}$:
\[
\pi_{(\bfS,\vartheta)^{\eta}}
=\cInd_{K^{\eta}}^{G}\rho_{(\bfS,\vartheta)^{\eta}}
\xrightarrow{(2)}
\cInd_{K^{\eta}}^{G}\rho_{(\bfS,\vartheta)}^{\eta}
\xrightarrow{(1)}
(\cInd_{K}^{G}\rho_{(\bfS,\vartheta)})^{\eta}
=\pi_{(\bfS,\vartheta)}^{\eta}.
\]

From now on (until the end of Section \ref{sec:HW-twisted}), suppose that we have the following:
\begin{itemize}
\item
an $F$-rational tame elliptic twisted maximal torus $(\t\bfS,\bfS)$ of $\t{\G}$, and
\item 
a tame elliptic toral pair $(\bfS,\vartheta)$ of depth $r\in\R_{>0}$ which is $\theta_\bfS$-invariant, i.e., $\vartheta=\vartheta\circ\theta_\bfS$ (note that this is equivalent to that $(\bfS,\vartheta)=(\bfS,\vartheta)^{\eta}$ for any $\eta\in\t{S}$).
\end{itemize}
Let us fix a base point $\ul{\eta}\in\t{S}$ which is topologically semisimple in the following.
We remark that we can always find such an element since $\t{S}$ is nonempty by the definition of a twisted maximal torus (Definition \ref{defn:twisted-tori}); apply Proposition \ref{prop:Jordan} to any element of $\t{S}$ and take its topologically semisimple part.

Then, by using the intertwiner $I_{\pi_{(\bfS,\vartheta)}}^{\ul{\eta}}$ constructed as above, we obtain a representation $\t{\pi}_{(\bfS,\vartheta)}$ of $\t{G}$ and its twisted character $\Theta_{\t{\pi}_{(\bfS,\vartheta)}}$ (see Section \ref{subsec:TC}).

We note that, since $[\ul{\eta}]$ preserves $\bfS$, the point $\x\in\mcB_{\red}(\G,F)$ associated to $\bfS$ is stabilized by the action on $\mcB_{\red}(\G,F)$ induced by $[\ul{\eta}]$.
Accordingly, every group used in the construction of $\rho_{(\bfS,\vartheta)}$ (such as $G_{\x,s}$, $K$, $J$, $J_{+}$, and so on) is stabilized by $[\ul{\eta}]$.
Hence we also have a twisted representation $\t{\rho}_{(\bfS,\vartheta)}$ of $\t{K}\colonequals K\ul{\eta}$ and its twisted character $\Theta_{\t{\rho}_{(\bfS,\vartheta)}}$, which is a function on $\t{K}$ defined by 
\[
\Theta_{\t{\rho}_{(\bfS,\vartheta)}}(k\ul{\eta})
\colonequals \tr\bigl(\rho_{(\bfS,\vartheta)}(k)\circ I_{\rho_{(\bfS,\vartheta)}}^{\ul{\eta}}\bigr).
\]\symdef{Theta-rho-tilde}{$\Theta_{\tilde{\rho}_{(\bfS,\vartheta)}}$}
Also note that $I_{\rho_{(\bfS,\vartheta)}}^{\ul{\eta}}$ naturally induces an intertwiner between $\sigma_{(\bfS,\vartheta)}$ and its $\ul{\eta}$-twist $\sigma_{(\bfS,\vartheta)}^{\ul{\eta}}$ (recall that $\sigma_{(\bfS,\vartheta)}$ is a representation of $K_{\sigma}$ defined by $\Ind_{K}^{K_{\sigma}}\rho_{(\bfS,\vartheta)}$).
Thus we get a twisted representation $\t{\sigma}_{(\bfS,\vartheta)}$ of $\t{K}_{\sigma}\colonequals K_{\sigma}\ul{\eta}$ and its twisted character $\Theta_{\t{\sigma}_{(\bfS,\vartheta)}}$, which is a function on $\t{K}_{\sigma}$ defined by 
\[
\Theta_{\t{\sigma}_{(\bfS,\vartheta)}}(k\ul{\eta})
\colonequals \tr\bigl(\sigma_{(\bfS,\vartheta)}(k)\circ I_{\sigma_{(\bfS,\vartheta)}}^{\ul{\eta}}\bigr).
\]\symdef{Theta-sigma-tilde}{$\Theta_{\tilde{\sigma}_{(\bfS,\vartheta)}}$}

\begin{rem}
We emphasize that the construction of the intertwiner $I_{\pi_{(\bfS,\vartheta)}}^{\ul{\eta}}$ explained above involves the unspecified choice of an intertwiner $I_{\omega_{(\bfS,\vartheta)}}^{\ul{\eta}}\colon\omega_{(\bfS,\vartheta)^{\ul{\eta}}}\cong\omega_{(\bfS,\vartheta)}^{\ul{\eta}}$ of Heisenberg--Weil representations.
Also, note that $I_{\pi_{(\bfS,\vartheta)}}^{\ul{\eta}}$ depends only on $I_{\omega_{(\bfS,\vartheta)}}^{\ul{\eta}}$.
In Section \ref{subsec:Heisen-int}, we will explain more about how to choose $I_{\omega_{(\bfS,\vartheta)}}^{\ul{\eta}}$.
\end{rem}

We finally recall that, by the torality of $\vartheta$, there exists a $\G$-generic element $X^{\ast}\in\mfs^{\ast}_{-r}$ of depth $r$ which lifts a unique element of $\mfs^{\ast}_{-r:-r+}$ satisfying $\vartheta(\exp(Y))=\psi_{F}(\langle Y,X^{\ast}\rangle)$ for any $Y\in\mfs_{s+:r+}$.
We note that
\[
\vartheta\circ \theta_\bfS(\exp(Y))
=\vartheta(\exp([\eta](Y)))
=\psi_{F}(\langle [\eta](Y),X^{\ast}\rangle)
=\psi_{F}(\langle Y,[\eta](X^{\ast})\rangle),
\]
where we again write $[\eta]$ for the action on $\mfs$ induced by $[\eta]$ and used that $\exp\colon\mfs_{s+:r+}\cong S_{s+:r+}$ is $[\eta]$-equivariant (the action of $[\eta]$ on $X^{\ast}$ is, by definition, given by the identity $\langle [\eta](Y),X^{\ast}\rangle=\langle Y,[\eta](X^{\ast})\rangle$).
Thus, as we have $\vartheta\circ\theta_\bfS=\vartheta$ by the assumption, we see that $[\eta](X^{\ast})$ equals $X^{\ast}$ in $\mfs_{-r:-r+}^{\ast}$.

\begin{lem}\label{lem:X-inv}
We may take $X^{\ast}\in\mfs_{-r}^{\ast}$ to be $[\eta]$-invariant.
\end{lem}

\begin{proof}
Since $p\nmid l$, we have $\frac{1}{l}\sum_{i=0}^{l}[\eta]^i(X^{\ast}) \in \mfs^{\ast}_{-r}$.
Note that this element is $[\eta]$-invariant.
Moreover, as the image of $X^{\ast}$ in $\mfs^{\ast}_{-r:-r+}$ is $[\eta]$-invariant, the image of $\frac{1}{l}\sum_{i=0}^{l}[\eta]^i(X^{\ast})$ in $\mfs^{\ast}_{-r:-r+}$ is equal to the image of $X^{\ast}$.
Thus, by replacing $X^{\ast}$ with $\frac{1}{l}\sum_{i=0}^{l}[\eta]^i(X^{\ast})$, we get a desired element.
\end{proof}

In the following, by this lemma, we assume that $X^{\ast}\in\mfs^{\ast}_{-r}$ is invariant under $[\eta]|_{\bfS}=\theta_{\bfS}$.

We finish this subsection by showing one more lemma.
Recall that $\bfS^\natural:=\bfS^{\theta_\bfS,\circ}$.
We put $\vartheta^{\nat}\colonequals \vartheta|_{S^{\nat}}$.

\begin{lem}\label{lem:toral-pair-descent}
For any $\eta\in\t{S}$, $(\bfS^{\nat},\vartheta^{\nat})$ is a tame elliptic toral pair of $\G_{\eta}$ of depth $r$.
\end{lem}

\begin{proof}
Note that, as we have $\bfS^{\nat}\subset\bfS$, we have $\mfs^{\nat}\subset\mfs$, hence $\mfs^{\ast}\twoheadrightarrow\mfs^{\nat\ast}$.
We can take an element of $\mfs^{\nat\ast}_{-r}$ representing the character $\vartheta^{\nat}$ to be the image of $X^{\ast}$ under the natural map $\mfs^{\ast}\twoheadrightarrow\mfs^{\nat\ast}$.
Our task is to show that $X^{\ast}$ is an $\G_{\eta}$-generic element of depth $r$.
We note that our assumption that $p\nmid|\Omega_{\G}|$ implies that $p\nmid|\Omega_{\G_{\eta}}|$ (recall that $\Omega_{\G_{\eta}}$ is regarded as a subgroup of $\Omega_{\G}$).
Thus it is enough to only check that \textbf{GE1} is satisfied, which requires that $\val_{F}(\langle H_{\alpha_{\res}},X^{\ast}\rangle)=-r$ for any $\alpha_{\res}\in\Phi(\G_{\eta},\bfS^{\nat})$, where $H_{\alpha_{\res}}=d\alpha_{\res}^{\vee}(1)$ (see Section \ref{subsec:toral-sc}).

By the description of $\Phi(\G_{\eta},\bfS^{\nat})$ and $\Phi^{\vee}(\G_{\eta},\bfS^{\nat})$ as in Section \ref{subsec:Steinberg}, we have $H_{\alpha_{\res}}=\varrho_{\alpha}\cdot \sum_{i=0}^{l_{\alpha}-1}H_{\theta_{\bfS}^{i}(\alpha)}$.
As $X^{\ast}$ is $\theta_{\bfS}$-invariant, we have $\langle H_{\theta_{\bfS}^{i}(\alpha)},X^{\ast}\rangle=\langle H_{\alpha},\theta_{\bfS}^{i}(X^{\ast})\rangle=\langle H_{\alpha},X^{\ast}\rangle$.
Hence $\langle H_{\alpha_{\res}},X^{\ast}\rangle=\varrho_{\alpha}\cdot l_{\alpha}\cdot\langle H_{\alpha},X^{\ast}\rangle$.
Since $X^{\ast}$ is $\G$-generic of depth $r$ and $p\nmid\varrho_{\alpha}\cdot l_{\alpha}$, we get $\val_{F}(\langle H_{\alpha_{\res}},X^{\ast}\rangle)=-r$.
\end{proof}

\subsection{Separation lemma}\label{subsec:separation}
In this subsection, we prove some technical lemma and propositions which will be needed later.

The following follows from \cite[Theorem 12.7.1]{KP23} by using the tamely ramified descent for the Bruhat--Tits buildings (\cite[Section 12.9]{KP23}).

\begin{prop}\label{prop:BT-descent}
Let $\delta_{0}\in\t{S}$ be absolutely $p$-semisimple modulo $A_{\t\G}$.
There exists an identification between the building $\mcB(\G_{\delta_{0}},F)$ and the fixed points of $\mcB(\G,F)$ under the action induced by $[\delta_{0}]$ such that $\mcA(\bfS^{\nat},F)$ is mapped to $\mcA(\bfS,F)^{\delta_{0}}$:
\[
\xymatrix@R=10pt{
\mcB(\G_{\delta_{0}},F)\ar^-{\cong}[r]&\mcB(\G,F)^{\delta_{0}}\ar@{}[r]|*{\subset}&\mcB(\G,F)\\
\mcA(\bfS^{\nat},F)\ar^-{\cong}[r]\ar@{}[u]|{\bigcup}&\mcA(\bfS,F)^{\delta_{0}}\ar@{}[r]|*{\subset}\ar@{}[u]|{\bigcup}&\mcA(\bfS,F)\ar@{}[u]|{\bigcup}
}
\]
\end{prop}

\begin{prop}\label{prop:MP-descent}
Let $\delta_{0}\in\t{S}$ be absolutely $p$-semisimple modulo $A_{\t\G}$.
Suppose that a point $\x$ associated to $\bfS$ belongs to $\mcA(\bfS^{\nat},F)$ under the identification as in Proposition \ref{prop:BT-descent}.
Then we have the following for any $r,s\in\widetilde{\R}_{>0}$ satisfying $r<s$:
\begin{enumerate}
\item
$S^{\nat}_{r}=(S_{r})^{\delta_{0}}\,(=(S_{r})^{\theta_{\bfS}})$ and $S^{\nat}_{0+:r}\cong(S_{0+:r})^{\delta_{0}}$,
\item
$G_{\delta_{0},\x,r}=(G_{\x,r})^{\delta_{0}}$ and $(S^{\nat},G_{\delta_{0}})_{\x,(r,s(+))}=(S,G)_{\x,(r,s(+))}^{\delta_{0}}$,
\item
$S^{\nat}_{0+}G_{\delta_{0},\x,r}=(S_{0+}G_{\x,r})^{\delta_{0}}$,
\item
$(S^{\nat},G_{\delta_{0}})_{\x,(r,s):(r,s+)}\cong(S,G)_{\x,(r,s):(r,s+)}^{\delta_{0}}$.
\end{enumerate}
\end{prop}

\begin{proof}
Let us first show (1).
Recall that $\delta_{0}$ acts on $\bfS$ via $\theta_{\bfS}$ and we put $\bfS^{\nat}\colonequals \bfS^{\theta_{\bfS},\circ}$.
The $r$-th filtration of $S$ is defined by 
\[
S_{r}\colonequals \{t\in S_{0} \mid \text{$\val_{F}(\chi(t)-1)\geq r$ for any $\chi\in X^{\ast}(\bfS)$} \},
\]
where $S_{0}$ denotes the unique parahoric subgroup of $S$ (see \cite[Definitions 2.5.13 and B.5.1]{KP23}).
Similarly, the $r$-th filtration of $S^{\nat}$ is defined by 
\[
S^{\nat}_{r}\colonequals \{t\in S^{\nat}_0 \mid \text{$\val_{F}(\chi(t)-1)\geq r$ for any $\chi\in X^{\ast}(\bfS^{\nat})$} \}.
\]
Thus, noting that $S^{\nat}_0$ is contained in $S_{0}$, we have $S^{\nat}_{r}\subset(S_{r})^{\theta_{\bfS}}$.
To show the converse inclusion, we take any element $t\in (S_{r})^{\theta_{\bfS}}$.
By Lemma \ref{lem:pi_{0}}, there exists a prime-to-$p$ number $k\in\Z_{>0}$ satisfying $t^{k}\in S^{\nat}_{r}$.
Then, as discussed in the proof of Proposition \ref{prop:Jordan}, we can remove $k$ to get $t\in S^{\nat}_{r}$ since $p\nmid k$ and $t$ is topologically $p$-unipotent.
We consider the latter part of (1).
By the former part which we just showed, we have $S^{\nat}_{0+:r}\cong(S_{0+})^{\delta_{0}}/(S_{r})^{\delta_{0}}$.
Note that we have $(S_{0+})^{\delta_{0}}/(S_{r})^{\delta_{0}}\hookrightarrow (S_{0+:r})^{\delta_{0}}$.
Let $\bar{s}$ be an element of $(S_{0+:r})^{\delta_{0}}$ represented by $s\in S_{0+}$.
Then $\prod_{i=0}^{l-1}\theta_{\bfS}^{i}(s)$ is an element of $(S_{0+})^{\delta_{0}}$.
Again by noting that $\prod_{i=0}^{l-1}\theta_{\bfS}^{i}(s)$ is topologically $p$-unipotent and $p\nmid l$, we can find an element $t\in (S_{0+})^{\delta_{0}}$ satisfying $t^{l}=\prod_{i=0}^{l-1}\theta_{\bfS}^{i}(s)$.
Then we have $\bar{t}^{l}=\bar{s}^{l}$, hence $\bar{t}=\bar{s}$ since the order of $(S_{0+:r})^{\delta_{0}}$ is prime to $l$.
Hence we obtained the surjectivity of the map $(S_{0+})^{\delta_{0}}/(S_{r})^{\delta_{0}}\hookrightarrow (S_{0+:r})^{\delta_{0}}$.

The assertion (2) follows from \cite[Proposition 12.8.5]{KP23} (together with the tamely ramified descent of Bruhat--Tits theory).
Note that the assumption of \cite[Proposition 12.8.5]{KP23} is satisfied by (1).

Let us show (3).
The inclusion $S^{\nat}_{0+}G_{\delta_{0},\x,r}\subset(S_{0+}G_{\x,r})^{\delta_{0}}$ is obvious.
To check the converse inclusion, let us take an element $g$ of $(S_{0+}G_{\x,r})^{\delta_{0}}$.
Since we have $S_{0+}\cap G_{\x,r}=S_{r}$ (see \cite[Proposition 4.6]{AS08}), we have a bijection
\[
S_{0+:r}=S_{0+}/S_{r}\xrightarrow{1:1} S_{0+}G_{\x,r}/G_{\x,r}.
\]
This implies that the coset $gG_{\x,r}$ is represented by an element $s$ of $S_{0+}$.
As $g$ is $[\delta_{0}]$-invariant and the above bijection is $[\delta_{0}]$-equivariant, the coset $sS_{r}$ is also $[\delta_{0}]$-invariant.
Since we have $(S_{0+:r})^{\delta_{0}}=S^{\nat}_{0+:r}$ by (1), we know that $s$ can be taken to be an element of $S^{\nat}_{0+}$.
Now let us write $g=sg'$ with $s\in S^{\nat}_{0+}$ and $g'\in G_{\x,r}$.
Since $g$ and $s$ are $\delta_{0}$-invariant, so is $g'$.
By (2), this implies that $g'\in G_{\delta_{0},\x,r}$.

The assertion (4) follows from the same argument as in the proof of assertion (1) by using (2).
\end{proof}

\begin{lem}\label{lem:sep-depth0}
Let $\delta$ be an elliptic regular semisimple element of $\t{G}$ with a topological Jordan decomposition $\delta=\delta_{0}\delta_{+}$.
If $\delta$ belongs to $\t{S}G_{\x,r}$, then $\delta_{0}$ belongs to ${}^{G_{\x,r}}\t{S}\colonequals \{{}^{g}s \mid g\in G_{\x,r}, s\in\t{S}\}$.
\end{lem}

\begin{proof}
For any element $g\in G^{\dagger}$, we write $\bar{g}$ for its image in $G^{\dagger}/A_{\t{\G}}$.
Similarly, we write $\overline{G_{\x,r}}$ and $\overline{S}$ for the images of $G_{\x,r}$ and $S$ in $G^{\dagger}/A_{\t\G}$, respectively.

Since $\overline{\delta_{0}}$ belongs to the closure of $\langle\overline{\delta}\rangle$ in $G^{\dagger}/A_{\t{\G}}$ (see Proposition \ref{prop:Jordan}), the assumption $\delta\in\t{S}G_{\x,r}$ implies that $\overline{\delta_{0}}\in\t{S}G_{\x,r}/A_{\t{\G}}$.
Let us take elements $s_{0}\in\t{S}$ and $g_{+}\in G_{\x,r}$ satisfying $\delta_{0}=g_{+}s_{0}$.
If we let $p'$ be the order of $\ol{\delta_{0}}$, which is prime to $p$, then we have
\[
1=\ol{\delta_{0}}^{p'}
=\prod_{i=0}^{p'-1}[s_{0}]^{i}(\ol{g_{+}})\cdot \ol{s_{0}}^{p'}.
\]
Since $\prod_{i=0}^{p'-1}[s_{0}]^{i}(\ol{g_{+}})\in \ol{G_{\x,r}}$, this implies that $s_{0}^{p'}$ lies in $A_{\t\G}(S\cap G_{\x,r})=A_{\t\G}S_{r}$ (see \cite[Proposition 4.6]{AS08} for the equality).
Furthermore, by noting that $s_{0}^{p'}$ is fixed by $[s_{0}]$ and $[s_{0}]$ acts on $S$ as $\theta_{\bfS}$ and on $A_{\t{\G}}$ trivially, we have $s_{0}^{p'}\in A_{\t\G}(S_{r})^{\theta_{\bfS}}$.
Thus, by Proposition \ref{prop:MP-descent} (1), we get $s_{0}^{p'}\in A_{\t\G}S^{\nat}_{r}$.
As $p'$ is prime to $p$, we can find an element $s_{r}\in S^{\nat}_{r}$ such that $s_{0}^{p'}\in A_{\t\G}\cdot s_{r}^{p'}$ (see the proof of Proposition \ref{prop:Jordan}, the same argument as in the construction of $\delta_{+}$ works).
Then, by replacing $s_{0}\in\t{S}$ with $s_{0}s_{r}^{-1}\in\t{S}$ and $g_{+}\in G_{\x,r}$ with $g_{+}s_{r}$, respectively, we may assume that 
\[
\prod_{i=0}^{p'-1}[s_{0}]^{i}(\ol{g_{+}})=1
\quad
\text{and}
\quad
\ol{s_{0}}^{p'}=1.
\]

In other words, we have an action of a finite cyclic group $\Z/p'\Z$ on $\ol{G_{\x,r}}$ given by $\bar{i}\cdot g=[s_{0}]^{i}(g)$ and a $1$-cocycle $\Z/p'\Z\rightarrow \ol{G_{\x,r}}$ given by $\bar{1}\mapsto \ol{g_{+}}$.
Since $p'$ is prime to $p$ and $\ol{G_{\x,r}}$ is a pro-$p$ group, the first group cohomology $H^{1}(\Z/p'\Z,\ol{G_{\x,r}})$ is trivial.
(This follows from a standard argument by using that the action of $\Z/p'\Z$ is filtration-preserving; see the proof of \cite[Theorem 13.8.5]{KP23} for the details).
Hence the cohomology class of the $1$-cocycle $[\bar{1}\mapsto \ol{g_{+}}]$ is trivial.
Namely, there exists an element $k\in G_{\x,r}$ such that $kg_{+}[s_{0}](k)^{-1}\in A_{\t{\G}}$.
This means that 
\[
{}^{k}\delta_{0}
=kg_{+}s_{0}k^{-1}
=kg_{+}[s_{0}](k)^{-1}\cdot s_{0}
\in A_{\t\G}\cdot\t{S}=\t{S}.
\]
\end{proof}

\begin{prop}\label{prop:sep}
Suppose that a point $\x$ associated to $\bfS$ belongs to $\mcA(\bfS^{\nat},F)$ under the identification as in Proposition \ref{prop:BT-descent}.
Let $\delta$ be an elliptic regular semisimple element of $\t{G}$ with a normal $r$-approximation $\delta=\delta_{0}\delta^{+}_{<r}\delta_{\geq r}$.
If $\delta$ belongs to ${}^{G_{\x,0+}}(\t{S}G_{\x,r})$, then $\delta_{\geq r} \in G_{\delta_{<r},\x,r}$ and there exists $k\in G_{\x,0+}$ such that
\[
{}^{k}\delta_{0}\in\t{S},\quad
{}^{k}\delta^{+}_{<r}\in S^{\nat}.
\]
Here, the point $\x\in\mcB(\G_{\delta_{0}},F)$ is regarded as a point of $\mcB(\G_{\delta_{<r}},F)$ by an embedding $\mcB(\G_{\delta_{<r}},F)\hookrightarrow \mcB(\G_{\delta_{0}},F)$.
\end{prop}

\begin{proof}
By replacing $\delta$ with its $G_{\x,0+}$-conjugate, we may assume that $\delta$ itself belongs to $\t{S}G_{\x,r}$.
By Lemma \ref{lem:sep-depth0}, we have ${}^{k}\delta_{0}\in \t{S}$ for some element $k\in G_{\x,r}$.
By furthermore replacing $\delta$ with ${}^{k}\delta$ (this is again an element of $\t{S}G_{\x,r}$ since $G_{\x,r}$ normalizes $\t{S}G_{\x,r}$), we may also assume that $\delta_0\in \t{S}$.
Then we have $\delta_{+}\in SG_{\x,r}$.

Recall that, by the construction of a topological Jordan decomposition (Proposition \ref{prop:Jordan}), $\delta_{+}$ belongs to $G_{\delta_{0},0+}$.
Hence we get $\delta_{+}\in (SG_{\x,r})\cap G_{\delta_{0},0+}=S^{\nat}_{0+}G_{\delta_{0},\x,r}$ by Proposition \ref{prop:MP-descent} (3).

Now the problem is reduced to the untwisted setting.
By applying \cite[Corollary 9.16]{AS08} to $\delta_{+}\in S^{\nat}_{0+}G_{\delta_{0},\x,r}$, we get $\delta^{+}_{<r}\in {}^{G_{\delta_{0},\x,0+}}S^{\nat}$.
In other words, we can find an element $k\in G_{\delta_{0},\x,0+}$ such that ${}^{k}\delta^{+}_{<r}\in S^{\nat}$.
Thus the remaining task is to show that $\delta_{\geq r} \in G_{\delta_{<r},\x,r}$.

By \cite[Lemma 9.13]{AS08}, the point $\x$ belongs to the set ``$\mcB_{r}(\delta_{+})$'' (which is considered in the group $\G_{\delta_{0}}$; see \cite[Definition 9.5]{AS08} for the definition).
By the description of the set $\mcB_{r}(\delta_{+})$ in \cite[Lemma 9.6]{AS08}, we have
\[
\mcB_{r}(\delta_{+})
=
\{\y\in\mcB(\bfC_{\G_{\delta_{0}}}^{(r)}(\delta_{+}),F) \mid \delta_{\geq r}\in G_{\delta_{0},\y,r}\}.
\]
Hence $\x$ belongs to the building of $\bfC_{\G_{\delta_{0}}}^{(r)}(\delta_{+})=(\G_{\delta_{0}})_{\delta^{+}_{<r}}$ (see \cite[Corollary 6.14]{AS08}), which furthermore equals $\G_{\delta_{<r}}$ by Lemma \ref{lem:cent-cent}, and we have $\delta_{\geq r}\in G_{\delta_{0},\x,r}$.
By the definition of a normal approximation, $\delta_{\geq r}$ belongs to $(G_{\delta_{0}})_{\delta_{<r}^{+}}=G_{\delta_{<r}}$.
Thus $\delta_{\geq r}$ lies in $G_{\delta_{0},\x,r}\cap G_{\delta_{<r}}$, which equals $G_{\delta_{<r},\x,r}$ by \cite[Proposition 4.6]{AS08}.
\end{proof}

\subsection{Twisted character formula of 1st form}\label{subsec:TCF-1st}

Recall that we have fixed a pair $(\bfS,\vartheta)$ in Section \ref{subsec:twist-int}.
Thus, from now on, we simply write $\omega$, $\rho$, $\sigma$, $\tau$, and $\pi$ for the representations $\omega_{(\bfS,\vartheta)}$, $\rho_{(\bfS,\vartheta)}$, $\sigma_{(\bfS,\vartheta)}$, $\tau_{(\bfS,\vartheta)}$, and $\pi_{(\bfS,\vartheta)}$ (see Section \ref{subsec:toral-sc}), respectively.
Similarly, we simply write $\t\rho$, $\t\sigma$, and $\t\pi$ for the twisted representations as introduced in Section \ref{subsec:twist-int}.
We use the identification of Bruhat--Tits buildings and apartments as in Proposition \ref{prop:BT-descent} in the following.
We may suppose that a point $\x$ associated to $\bfS$ comes from $\mcA(\bfS^{\nat},F)$.

In the following, we fix an elliptic regular semisimple element $\delta\in\t{G}$ and a normal $r$-approximation $\delta=\delta_{0}\delta^{+}_{<r}\delta_{\geq r}$ to $\delta$, which exists by Proposition \ref{prop:twisted-Jordan}.
We write $\eta\colonequals \delta_{<r}$ in short.
Our final goal is to establish an explicit formula of the twisted character $\Theta_{\tilde{\pi}}(\delta).$

As a first step, we show the following lemma, which is a twisted version of \cite[Lemma 6.1]{AS09}:

\begin{lem}\label{lem:finiteness}
The set $S\backslash\{g\in G\mid {}^{g}\eta\in\t{S}\}/G_{\eta}$ is finite.
\end{lem}

\begin{proof}
First we note that $\bfZ_{\G}(\bfS^{\nat})=\bfS$ (Proposition \ref{prop:twisted-tori} (1)).
Thus, if an element $n\in\G$ belongs to $\bfN_{\G}(\bfS^{\nat})$, i.e., satisfies $n\bfS^{\nat}n^{-1}\subset\bfS^{\nat}$, then we get $n\bfS n^{-1}\supset\bfS$ by taking the centralizer groups in $\G$.
Hence $\bfN_{\G}(\bfS^{\nat}) \subset \bfN_{\G}(\bfS)$.
Since $\bfS$ is contained in $\bfN_{\G}(\bfS^{\nat})$ and of finite index in $\bfN_{\G}(\bfS)$, $S$ is of finite index in $N_{G}(\bfS^{\nat})$.
Thus it is enough to show that $N_{G}(\bfS^{\nat})\backslash\{g\in G\mid {}^{g}\eta\in\t{S}\}/G_{\eta}$ is finite.

For any element $g\in G$ satisfying ${}^{g}\eta\in\t{S}$, by Proposition \ref{prop:twisted-tori} (2), we have 
\[
{}^{g}\G_{\eta}
=\G_{{}^{g}\eta}
=\bfZ_{\G}({}^{g}\eta)^{\circ}
\supset\bfZ_{\G}(\t{\bfS})^{\circ}
=\bfS^{\nat}.
\]
In other words, we have ${}^{g^{-1}}\bfS^{\nat}\subset\G_{\eta}$.
Since $\bfS^{\nat}$ is an $F$-rational maximal torus of $\G_{\eta}$ (see Proposition \ref{prop:twisted-tori} (3)), so is ${}^{g^{-1}}\bfS^{\nat}$.
Therefore we get an injection
\[
N_{G}(\bfS^{\nat})\backslash\{g\in G\mid {}^{g}\eta\in\t{S}\}
\hookrightarrow
\{\text{$F$-rational maximal tori of $\G_{\eta}$}\}
\colon g\mapsto {}^{g^{-1}}\bfS^{\nat}.
\]
By taking the quotients with respect to the action of $G_{\eta}$, we furthermore get
\[
N_{G}(\bfS^{\nat})\backslash\{g\in G\mid {}^{g}\eta\in\t{S}\}/G_{\eta}
\hookrightarrow
\{\text{$F$-rational maximal tori of $\G_{\eta}$}\}/{\sim_{G_{\eta}}},
\]
where the symbol $\sim_{G_{\eta}}$ denotes the equivalence class given by $G_{\eta}$-conjugation.
As the right-hand side is finite, $N_{G}(\bfS^{\nat})\backslash\{g\in G\mid {}^{g}\eta\in\t{S}\}/G_{\eta}$ is also finite.
\end{proof}

Recall that $\t{\sigma}$ is a representation of $\t{K}_{\sigma}=\t{S}G_{\x,0+}$  and that $\Theta_{\t{\sigma}}$ is its twisted character with respect to the intertwiner chosen in Section \ref{subsec:twist-int}.
Let $\dot{\Theta}_{\t{\sigma}}$ be the zero extension of $\Theta_{\t{\sigma}}$ from $\t{K}_{\sigma}=\t{S}G_{\x,0+}$ to $\t{G}$.

The following lemma is a twisted version of \cite[Proposition 4.3]{AS09}.

\begin{lem}\label{lem:sigma}
For any $g\in G$, if $\dot{\Theta}_{\t{\sigma}}({}^{g}\delta)\neq0$, then we have ${}^{g}\delta\in{}^{G_{\x,0+}}(\t{S}G_{\x,r})$.
\end{lem}

\begin{proof}
We put $\delta'\colonequals {}^{g}\delta$.
By letting $\delta'_{0}:={}^{g}\delta_{0}$, $\delta'^{+}_{<r}:={}^{g}\delta^{+}_{<r}$, $\delta'_{\geq r}:={}^{g}\delta_{\geq r}$, we get a normal $r$-approximation $\delta'=\delta'_{0}\delta'^{+}_{<r}\delta'_{\geq r}$ to $\delta'$.
Suppose that $\dot{\Theta}_{\t{\sigma}}(\delta')\neq0$, in particular, $\delta'$ belongs to $\t{K}_{\sigma}=\t{S}G_{\x,0+}$.
Then, by Proposition \ref{prop:sep} (take $r$ in Proposition \ref{prop:sep} to be $0+$), we know that $\delta'_{0}\in {}^{G_{\x,0+}}\t{S}$ and $\delta'_{+} \in G_{\delta'_{0},\x,0+}$.

Let $t\in\R_{>0}$ be the largest number such that $\delta'_{+}\in G_{\delta'_{0},\x,t}\smallsetminus G_{\delta'_{0},\x,t+}$.
To complete the proof, it suffices to show that $t\geq r$.
Let us suppose that $t<r$ for the sake of a contradiction.

We take $k\in G_{\x,0+}$ satisfying $\delta'_{0}\in{}^{k}\t{S}$ and put $(\t{\bfS}',\bfS')\colonequals ({}^{k}\t{\bfS},{}^{k}\bfS)$.
By \cite[Lemma 9.13]{AS09} (we take $(\G',\G)$ to be $(\bfS^{\prime\nat},\G_{\delta'_{0}})$), we know that $\x\in\mcB_{t}(\delta'_{+})$.
In other words, $\x$ belongs to the building of $\bfC_{\G_{\delta'_{0}}}^{(t)}(\delta'_{+})=(\G_{\delta'_{0}})_{\delta^{\prime+}_{<t}}=\G_{\delta'_{<t}}$ and we have $\delta'_{\geq t}\in G_{\delta'_{<t},\x,t}$ (cf.\ the proof of Proposition \ref{prop:sep}).
For any $h\in C_{\bfG_{\delta'_{0}}}^{(t)}(\delta'_{+})_{\x,r-t}=G_{\delta'_{<t},\x,r-t}$, we have $[\delta'^{-1},h]=[\delta'^{-1}_{\geq t},h]\in G_{\delta'_{0},\x,r}$.
Thus, by noting that $\sigma$ is $\hat{\vartheta}$-isotypic on $G_{\x,r}$ (\cite[Lemma 2.5]{AS08}) and that $\Theta_{\t{\sigma}}$ is invariant under $K_{\sigma}$-conjugation, we get
\[
\Theta_{\t{\sigma}}(\delta')
=\Theta_{\t{\sigma}}({}^{h}\delta')
=\Theta_{\t{\sigma}}(\delta'\cdot[\delta'^{-1},h])
=\Theta_{\t{\sigma}}(\delta')\cdot\hat{\vartheta}([\delta'^{-1}_{\geq t},h])
\]
for any $h\in G_{\delta'_{<r},\x,r-t}$.
Since $\hat{\vartheta}([\delta'^{-1}_{\geq t},-])$ is nontrivial on $G_{\delta'_{<r},\x,r-t}$ as proved in the final paragraph of the proof of \cite[Proposition 4.3]{AS09}, we conclude that $\Theta_{\t{\sigma}}(\delta')$ equals zero.
This is a contradiction.
\end{proof}

We next establish a twisted version of \cite[Lemma 6.3]{AS08}.

\begin{lem}\label{lem:ave}
Let $\mcK_{\eta}$ be an open compact subgroup of $G_{\eta}$.
Then the function
\[
\mcF\colon G \rightarrow \C
;\quad
g \mapsto \int_{\mcK_{\eta}} \dot{\Theta}_{\t{\sigma}}\bigl({}^{gk}\delta\bigr)\,dk
\]
is compactly supported modulo $Z_{\G}$, where $dk$ is any Haar measure on $\mcK_\eta$.
\end{lem}

\begin{proof}
We first note that the support of $\mcF$ is contained the following set:
\[
\bigl\{g\in G \mid {}^{g}\eta\in{}^{G_{\x,0+}}\t{S}\bigr\}.
\]
Indeed, if ${}^{gk}{\delta}$ belongs to the support of $\dot{\Theta}_{\t{\sigma}}$, then ${}^{gk}{\delta}$ lies in ${}^{G_{\x,0+}}(\t{S}G_{\x,r})$ by Lemma \ref{lem:sigma}.
On the other hand, as $\mcK_{\eta}$ is a subset of $G_{\eta}$, every $k\in \mcK_{\eta}$ commutes with $\eta=\delta_{<r}$.
Hence we have ${}^{gk}\delta = {}^{g}\eta \cdot{}^{gk}\delta_{\geq r}$.
Thus, by Proposition \ref{prop:sep}, ${}^{g}\eta$ necessarily belongs to ${}^{G_{\x,0+}}\t{S}$.

We consider the following double quotient:
\[
K_{\sigma}
\backslash
\bigl\{g\in G \mid {}^{g}\eta\in{}^{G_{\x,0+}}\t{S}\bigr\}
/
G_{\eta}.
\]
Since $K_{\sigma}=SG_{\x,0+}$, we have a natural surjection
\[
S
\backslash
\bigl\{g\in G \mid {}^{g}\eta\in\t{S}\bigr\}
/
G_{\eta}
\twoheadrightarrow
K_{\sigma}
\backslash
\bigl\{g\in G \mid {}^{g}\eta\in{}^{G_{\x,0+}}\t{S}\bigr\}
/
G_{\eta}.
\]
As the former set is finite by Lemma \ref{lem:finiteness}, so is the latter set.
Therefore, in order to show that $\mcF$ is compactly supported modulo $Z_{\G}$, it is enough to show that $\mcF$ is compactly supported modulo $Z_{\G}$ on each double coset $K_{\sigma}gG_{\eta}$.
From now on, we fix an element $g\in G$ satisfying ${}^{g}\eta\in{}^{G_{\x,0+}}\t{S}$.
By replacing $g$ with some other representative in the double coset $K_{\sigma}gG_{\eta}$ if necessary, we may suppose that $g$ satisfies ${}^{g}\eta\in\t{S}$.

We define a function $\mcF_{g}\colon G_{\eta} \rightarrow \C$ by
\[
\mcF_{g}(h)
\colonequals 
\int_{\mcK_{\eta}} \dot{\Theta}_{\t{\sigma}}\bigl({}^{ghk}\delta\bigr)\,dk.
\]
Note that the function $\dot{\Theta}_{\t{\sigma}}$ is invariant under $K_{\sigma}$-conjugation.
Thus the function $\mcF$ is left-$K_{\sigma}$-invariant, and the restriction of $\mcF$ to the double coset $K_{\sigma}gG_{\eta}$ is given by $\mcF|_{K_{\sigma}gG_{\eta}}(lgh)=\mcF_{g}(h)$.
As $K_{\sigma}$ is compact modulo $Z_{\G}$, it is enough to show that $\mcF_{g}$ is compactly supported modulo $G_{\eta}\cap Z_{\G}$.
Since $\bfA_{\t\G}$ is defined to be the maximal split subtorus of $\bfZ_{\G}^{\theta}$, we have $\bfA_{\t\G}\subset \G_{\eta}\cap \bfZ_{\G} \subset \bfZ_{\G}^{\theta}$.
Hence it suffices to show that $\mcF_{g}$ is compactly supported modulo $A_{\t{\G}}$.

We compute $\dot{\Theta}_{\t{\sigma}}({}^{ghk}\delta)$ in the integrand of $\mcF_{g}$.
Since $g$ is chosen to satisfy ${}^{g}\eta\in\t{S}$, by also noting that $h,k\in G_{\eta}$, we have ${}^{ghk}\eta={}^{g}\eta\in\t{S}$.
On the other hand, if $\dot{\Theta}_{\t{\sigma}}({}^{ghk}\delta)$ is not zero, then we have ${}^{ghk}\delta_{\geq r} \in G_{{}^{g}\eta,\x,r}$ by Lemma \ref{lem:sigma} and Proposition \ref{prop:sep}.
Therefore, by noting that the restriction of $\sigma$ on $G_{\x,r}$ is $\hat{\vartheta}$-isotypic (\cite[Lemma 2.5]{AS09}), we get
\[
\dot{\Theta}_{\t{\sigma}}({}^{ghk}\delta)
=
\dot{\Theta}_{\t{\sigma}}({}^{g}\eta)
\mathbbm{1}_{G_{{}^{g}\eta,\x,r}}({}^{ghk}\delta_{\geq r})\hat{\vartheta}({}^{ghk}\delta_{\geq r}).
\]
Since the term $\dot{\Theta}_{\t{\sigma}}({}^{g}\eta)$ does not depend on $h$ or $k$, it suffices to show that the function
\[
\tilde{\mcF}_{g}\colon G_{\eta} \rightarrow \C;\quad
h \mapsto \int_{\mcK_{\eta}} 
\mathbbm{1}_{G_{{}^{g}\eta,\x,r}}({}^{ghk}\delta_{\geq r})\hat{\vartheta}({}^{ghk}\delta_{\geq r})
\,dk
\]
is compactly supported modulo $A_{\t{\G}}$.

In the following, we put $\eta'\colonequals {}^{g}\eta$.
Now recall that our toral cuspidal $\G$-datum is given by $((\bfS\subset\G), \x, (r,r), (\vartheta,\mathbbm{1}), \mathbbm{1})$.
We consider a toral cuspidal $\G_{\eta'}$-datum $((\bfS^{\nat}\subset\G_{\eta'}), \x, (r,r), (\vartheta^{\nat},\mathbbm{1}), \mathbbm{1})$, where we put $\vartheta^{\nat}\colonequals \vartheta|_{S^{\nat}}$.
(Note that the torality is guaranteed by Lemma \ref{lem:toral-pair-descent}.)
We express various objects appearing in Yu's construction for this cuspidal $\G_{\eta'}$-datum by adding a subscript $\eta'$ to the notation used in Section \ref{subsec:rsc}.
Then, again by using Lemma \ref{lem:sigma}, Proposition \ref{prop:sep}, and \cite[Lemma 2.5]{AS09}, we have
\[
\dot{\Theta}_{\sigma_{\eta'}}({}^{ghk}\delta_{\geq r})
=
\mathbbm{1}_{G_{\eta',\x,r}}({}^{ghk}\delta_{\geq r})\hat{\vartheta}_{\eta'}({}^{ghk}\delta_{\geq r}).
\]
Namely, we get
\[
\tilde{\mcF}_{g}(h)
=
\int_{\mcK_{\eta}}
\dot{\Theta}_{\sigma_{\eta'}}({}^{ghk}\delta_{\geq r})\,dk.
\]
Since the representation $\cInd_{K_{\sigma_{\eta'}}}^{G_{\eta'}}\sigma_{\eta'}$ is supercuspidal by Yu's theory, this function is compactly supported modulo $Z_{\G_{\eta'}}$, by Harish-Chandra's well-known result (\cite[Lemma 23]{HC70}).
Therefore, now our assertion is reduced to the compactness of the quotient $Z_{\G_{\eta'}}/A_{\t{\G}}$.

Since we have $\G_{\eta'}\supset\bfS^{\nat}$, we have $\bfZ_{\G_{\eta'}} \subset \bfS^{\nat}$.
As $\t{\bfS}$ is an $F$-rational elliptic twisted maximal torus of $\t{\G}$, $\bfS^{\nat}$ is anisotropic modulo $\bfA_{\t{\G}}$ (see Definition \ref{defn:elliptic}), hence $S^\natural$ is compact modulo $A_{\t{\G}}$.
Thus $Z_{\G_{\eta'}}$ is compact modulo $A_{\t{\G}}$.
\end{proof}

Before we state the ``first form'' of a twisted version of Adler--DeBacker--Spice character formula, we introduce some notation.
Recall that, for any connected reductive group $\J$ and a regular semisimple element $X_{J}^{\ast}\in\mfj^{\ast}$, the Fourier transform of the orbital integral $\hat{\mu}^{\J}_{X_{J}^{\ast}}$ is defined as follows (see \cite[Section 4.2]{Kal19} for the details).
We consider the distribution $O_{X_{J}^{\ast}}(-)$ on $\mfj^{\ast}$ given by
\[
O_{X_{J}^{\ast}}(f)\colonequals \int_{J/Z_{\J}(X^{\ast})^{\circ}} f^{\ast}(hX_{J}^{\ast}h^{-1})\,dh
\]
for $f^{\ast}\in C_{c}^{\infty}(\mfj^{\ast})$, where we fix a Haar measure $dh$ on $J$.
For any element $f\in C_{c}^{\infty}(\mfj)$, we let $\hat{f}$ denote its Fourier transform with respect to the fixed additive character $\psi_{F}$, that is, $\hat{f}$ is an element of $C_{c}^{\infty}(\mfj^{\ast})$ given by
\[
\hat{f}(Y^{\ast})\colonequals \int_{\mfj} f(Y)\cdot\psi_{F}(\langle Y,Y^{\ast}\rangle)\,dY,
\]
where $dY$ is a Haar measure on $\mfj$.
Then the distribution $f\mapsto O_{X_{J}^{\ast}}(\hat{f})$ on $\mfj$ is represented by a function $\hat{\mu}^{\J}_{X_{J}^{\ast}}$ on $\mfj$, i.e., we have
\[
O_{X_{J}^{\ast}}(\hat{f})
=
\int_{\mfj} \hat{\mu}^{\J}_{X_{J}^{\ast}}(Y)\cdot f(Y)\,dY
\]
for any $f\in C_{c}^{\infty}(\mfj)$.
We emphasize that the function $\hat{\mu}^{\J}_{X_{J}^{\ast}}$ does not depend on the choice of $dY$, but depends on the choice of $dh$.

Recall that, as discussed in the proof of Lemma \ref{lem:ave}, we have a tame elliptic toral pair $(\bfS^{\nat},\vartheta^{\nat})$ of $\G_{\eta'}$ (here, $\eta'\colonequals {}^{g}\eta$ for an element $g\in G$ satisfying ${}^{g}\eta\in\t{S}$) represented by $X^{\ast}\in\mfs_{-r}^{\nat\ast}$, which is the image of the fixed element $X^{\ast}\in(\mfs_{-r}^{\ast})^{\theta_{\bfS}}$ representing the character $\vartheta|_{S_{r}}$ (see Section \ref{subsec:twist-int}).
Since $X^{\ast}$ is $\G_{\eta'}$-generic of depth $r$ (regarded as an element of $\mfg_{\eta'}^{\ast}$), in particular, $X^\ast$ is regular semisimple in $\G_{\eta'}$.
Hence we have $\bfZ_{\G_{\eta'}}(X^{\ast})^{\circ}=\bfS^{\nat}$.
By noting that $G_{\eta'}/S^{\nat}$ is the quotient of $G_{\eta'}/(G_{\eta'}\cap Z_{\G})$ by $S^{\nat}/(G_{\eta'}\cap Z_{\G})$, we choose a measure on  $G_{\eta'}/S^{\nat}$ which is the quotient of the following two measures:
\begin{itemize}
\item
the Haar measure $dh$ on $G_{\eta'}/(G_{\eta'}\cap Z_{\G})$ satisfying $dh\bigl((G_{\eta'}\cap K_{\sigma})/(G_{\eta'}\cap Z_{\G})\bigr)=1$;
\item
the Haar measure on $S^{\nat}/(G_{\eta'}\cap Z_{\G})$ whose total volume is $1$ (note that $G_{\eta'}\cap Z_{\G}$ is co-compact in $S^{\nat}$, which follows from that $S^{\nat}$ is compact modulo $Z_{\G_{\eta'}}$; cf.\ the final step of the proof of Lemma \ref{lem:ave}).
\end{itemize}

The following is the twisted version of \cite[Theorem 6.4]{AS09}:

\begin{thm}\label{thm:TCF-0}
We have
\begin{align}\label{eq:thm:TCF-0}
\Theta_{\t{\pi}}(\delta)
=
\sum_{\begin{subarray}{c} g\in S\backslash G/G_{\eta} \\ {}^{g}\eta\in\t{S}\end{subarray}}
\Theta_{\t{\sigma}}({}^{g}\eta)
\cdot
\hat{\mu}^{\G_{{}^{g}\eta}}_{X^{\ast}} \bigl(\log({}^{g}\delta_{\geq r})\bigr).
\end{align}
Here, note that the condition ${}^{g}\eta\in\t{S}$ implies that $\bfS^{\nat}\subset\G_{{}^{g}\eta}$, hence the function $\hat{\mu}^{\G_{{}^{g}\eta}}_{X^{\ast}}(-)$ makes sense as explained above.
In the definition of $\hat{\mu}^{\G_{{}^{g}\eta}}_{X^{\ast}}$, we use the Haar measure on $G_{{}^{g}\eta}/S^{\nat}$ explained above.
\end{thm}

\begin{proof}
The starting point of the proof is the twisted version of Harish-Chandra's integration formula (see \cite[Partie I, Th\'eor\`eme 6.2.1 (2)]{LH17}):
\[
\Theta_{\t{\pi}}(\delta)
=
\frac{\deg\pi}{\dim\sigma}
\int_{G/Z_{\G}}
\int_{\mcK} \dot{\Theta}_{\t{\sigma}} ({}^{\dot{g}k}\delta)\, dk\,d\dot{g},
\]
where $\mcK$ is an open compact subgroup of $G$, $dk$ is the Haar measure on $\mcK$ satisfying $dk(\mcK)=1$, and $d\dot{g}$ is a Haar measure on $G/Z_{\G}$ and $\deg\pi$ denotes the formal degree of $\pi$ with respect to the measure $d\dot{g}$.

We take an open compact subgroup $\mcK_{\eta}$ of $G_{\eta}$ to be $\mcK_{\eta}=\mcK\cap G_{\eta}$.
We let $dc$ be the Haar measure of $\mcK_{\eta}$ satisfying $dc(\mcK_{\eta})=1$.
Then we can replace the integral over $\mcK$ in Harish-Chandra's integration formula with an integral over $\mcK_{\eta}$ by the following standard argument.
First, since $\mcK_{\eta}\subset\mcK$ and $dc(\mcK_{\eta})=1$, we have
\[
\int_{G/Z_{\G}}
\int_{\mcK} \dot{\Theta}_{\t{\sigma}} ({}^{\dot{g}k}\delta)\, dk\,d\dot{g}
=
\int_{G/Z_{\G}}
\int_{\mcK_{\eta}}
\int_{\mcK} \dot{\Theta}_{\t{\sigma}} ({}^{\dot{g}kc}\delta)\, dk\,dc\,d\dot{g}.
\]
By applying Fubini's theorem to the inner double integral (note that both of $\mcK_{\eta}$ and $\mcK$ are compact), we get
\[
\int_{G/Z_{\G}}
\int_{\mcK_{\eta}}
\int_{\mcK} \dot{\Theta}_{\t{\sigma}} ({}^{\dot{g}kc}\delta)\, dk\,dc\,d\dot{g}
=
\int_{G/Z_{\G}}
\int_{\mcK}
\int_{\mcK_{\eta}} \dot{\Theta}_{\t{\sigma}} ({}^{\dot{g}kc}\delta)\, dc\,dk\,d\dot{g}.
\]
Then, since the inner integral over $\mcK_{\eta}$ is compactly supported as a function on $\dot{g}\in G/Z_{\G}$ (Lemma \ref{lem:ave}), we can apply Fubini's theorem to the outer double integral:
\[
\int_{G/Z_{\G}}
\int_{\mcK}
\int_{\mcK_{\eta}} \dot{\Theta}_{\t{\sigma}} ({}^{\dot{g}kc}\delta)\, dc\,dk\,d\dot{g}
=
\int_{\mcK}
\int_{G/Z_{\G}}
\int_{\mcK_{\eta}} \dot{\Theta}_{\t{\sigma}} ({}^{\dot{g}kc}\delta)\, dc\,d\dot{g}\,dk.
\]
Finally, by using that $d\dot{g}$ is right $G$-invariant and that $dk(\mcK)=1$, we get 
\[
\int_{\mcK}
\int_{G/Z_{\G}}
\int_{\mcK_{\eta}} \dot{\Theta}_{\t{\sigma}} ({}^{\dot{g}kc}\delta)\, dc\,d\dot{g}\,dk
=
\int_{G/Z_{\G}}
\int_{\mcK_{\eta}} \dot{\Theta}_{\t{\sigma}} ({}^{\dot{g}c}\delta)\, dc\,d\dot{g}.
\]

Now we consider the following partition of $G/Z_{\G}$ into double cosets:
\[
\int_{G/Z_{\G}}
\int_{\mcK_{\eta}} \dot{\Theta}_{\t{\sigma}} ({}^{\dot{g}c}\delta)\, dc\,d\dot{g}
=
\sum_{g\in K_{\sigma}\backslash G/G_{\eta}}
\int_{K_{\sigma}gG_{\eta}/Z_{\G}}
\int_{\mcK_{\eta}} \dot{\Theta}_{\t{\sigma}} ({}^{\dot{g}c}\delta)\, dc\,d\dot{g}.
\]
Note that, by Proposition \ref{prop:sep} and Lemma \ref{lem:sigma}, if the contribution of the summand with respect to $g\in K_{\sigma}\backslash G/G_{\eta}$ is nonzero, then there exists an element $g'$ in the double coset satisfying ${}^{g'}\eta\in{}^{G_{\x,0+}}\t{S}$.
By Lemma \ref{lem:index}, which will be proved later, the natural surjective map
\[
S\backslash\bigl\{g\in G \mid {}^{g}\eta\in\t{S}\bigr\}/G_{\eta}
\twoheadrightarrow
K_{\sigma}\backslash\bigl\{g\in G \mid {}^{g}\eta\in{}^{G_{\x,0+}}\t{S}\bigr\}/G_{\eta}
\]
is in fact bijective.
Hence we see that the above sum of double integrals equals
\begin{align}\label{eq:thm:TCF-1}
\sum_{\begin{subarray}{c} g\in S\backslash G/G_{\eta} \\ {}^{g}\eta\in\t{S} \end{subarray}}
\int_{K_{\sigma}gG_{\eta}/Z_{\G}}
\int_{\mcK_{\eta}} \dot{\Theta}_{\t{\sigma}} ({}^{\dot{g}c}\delta)\, dc\,d\dot{g}.
\end{align}

Let us compute each summand by fixing $g\in S\backslash G/G_{\eta}$ satisfying ${}^{g}\eta\in\t{S}$.
We put
\[
y\colonequals \dot{g}g^{-1}
\quad\text{and}\quad
c'\colonequals {}^{g}c=gcg^{-1}.
\]
Then, letting $dy$ and $dc'$ be the Haar measures on $K_{\sigma}{}^{g}G_{\eta}/Z_{\G}$ and ${}^{g}\mcK_{\eta}$ naturally induced from $d\dot{g}$ and $dc$, respectively, we get
\[
\int_{K_{\sigma}gG_{\eta}/Z_{\G}}
\int_{\mcK_{\eta}} \dot{\Theta}_{\t{\sigma}} ({}^{\dot{g}c}\delta)\, dc\,d\dot{g}
=
\int_{K_{\sigma}{}^{g}G_{\eta}/Z_{\G}}
\int_{{}^{g}\mcK_{\eta}} \dot{\Theta}_{\t{\sigma}} ({}^{yc'g}\delta)\, dc'\,dy.
\]
By putting $\delta'\colonequals {}^{g}\delta$, $\eta'\colonequals {}^{g}\eta$, and $\delta'_{\geq r}\colonequals {}^{g}\delta_{\geq r}$, we get
\[
\int_{K_{\sigma}{}^{g}G_{\eta}/Z_{\G}}
\int_{{}^{g}\mcK_{\eta}} \dot{\Theta}_{\t{\sigma}} ({}^{yc'g}\delta)\, dc'\,dy
=
\int_{K_{\sigma}G_{\eta'}/Z_{\G}}
\int_{\mcK_{\eta'}} \dot{\Theta}_{\t{\sigma}} ({}^{yc'}\delta')\, dc'\,dy.
\]

We let $dh$ be the Haar measure on $G_{\eta'}/(G_{\eta'}\cap Z_{\G})\cong G_{\eta'}Z_{\G}/Z_{\G}$ 
normalized so that $dh\bigl((G_{\eta'}\cap K_{\sigma})/(G_{\eta'}\cap Z_{\G})\bigr)=1$.
Let $d\dot{y}$ be the quotient measure on $K_{\sigma}G_{\eta'}/G_{\eta'}Z_{\G}$ of $dy$ by $dh$.
Then we have
\begin{multline*}
\int_{K_{\sigma}G_{\eta'}/Z_{\G}}
\int_{\mcK_{\eta'}} \dot{\Theta}_{\t{\sigma}} ({}^{yc'}\delta')\, dc'\,dy\\
=
\int_{K_{\sigma}G_{\eta'}/G_{\eta'}Z_{\G}}
\int_{G_{\eta'}/G_{\eta'}\cap Z_{\G}}
\int_{\mcK_{\eta'}} \dot{\Theta}_{\t{\sigma}} ({}^{\dot{y}hc'}\delta')\, dc'\,dh\,d\dot{y}.
\end{multline*}
Since $\dot{\Theta}_{\t{\sigma}}$ is left-$K_{\sigma}$-invariant, this triple integral equals
\begin{align}\label{eq:thm:TCF-2}
d\dot{y}(K_{\sigma}G_{\eta'}/G_{\eta'}Z_{\G})
\int_{G_{\eta'}/G_{\eta'}\cap Z_{\G}}
\int_{\mcK_{\eta'}} \dot{\Theta}_{\t{\sigma}} ({}^{hc'}\delta')\, dc'\,dh.
\end{align}

Let us compute the volume $d\dot{y}(K_{\sigma}G_{\eta'}/G_{\eta'}Z_{\G})$.
Since $K_{\sigma}G_{\eta'}/G_{\eta'}Z_{\G}\cong K_{\sigma}/(G_{\eta'}\cap K_{\sigma})Z_{\G}$ is equal to the quotient of $K_{\sigma}/Z_{\G}$ by $(G_{\eta'}\cap K_{\sigma})Z_{\G}/Z_{\G}\cong(G_{\eta'}\cap K_{\sigma})/(G_{\eta'}\cap Z_{\G})$, the volume $d\dot{y}(K_{\sigma}G_{\eta'}/G_{\eta'}Z_{\G})$ is given by
\[
dy(K_{\sigma}/Z_{\G})
\cdot
dh\bigl((G_{\eta'}\cap K_{\sigma})/(G_{\eta'}\cap Z_{\G})\bigr)^{-1}.
\]
By our choice of $dh$, we have $dh((G_{\eta'}\cap K_{\sigma})/(G_{\eta'}\cap Z_{\G}))=1$.
On the other hand, we have 
\[
dy(K_{\sigma}/Z_{\G})=d\dot{g}(K_{\sigma}g/Z_{\G})=d\dot{g}(K_{\sigma}/Z_{\G})
=
\frac{\dim\sigma}{\deg\pi}.
\]
(The final equality is a well-known formula for the formal degree of a compactly induced supercuspidal representation; see, for example, \cite[Partie I, Th\'eor\`eme 6.2.1 (1)]{LH17}.)
Hence we obtain $d\dot{y}(K_{\sigma}G_{\eta'}/G_{\eta'}Z_{\G})=\dim\sigma/\deg\pi$.

Let us next compute the double integral in \eqref{eq:thm:TCF-2}.
Recall that in the proof of Lemma \ref{lem:ave} we showed that
\[
\dot{\Theta}_{\t{\sigma}}({}^{hc'}\delta')
=
\Theta_{\t{\sigma}}(\eta')
\mathbbm{1}_{G_{\eta',\x,r}}({}^{hc'}\delta'_{\geq r})\hat{\vartheta}({}^{hc'}\delta'_{\geq r}).
\]
Note that $\delta'_{+}$ is regular semisimple in $\G_{\delta'_{0}}$ by Lemma \ref{lem:rss-descent}.
As $\delta'_{+}=\delta^{\prime+}_{<r}\delta'_{\geq r}$ is a normal $r$-approximation in $\G_{\delta'_{0}}$, the regular semisimplicity of $\delta'_{+}$ in $\G_{\delta'_{0}}$ implies that of $\delta'_{\geq r}$ in $(\G_{\delta'_{0}})_{\delta^{\prime+}_{<r}}=\G_{\eta'}$ (Lemma \ref{lem:cent-cent} and \cite[Corollary 6.14]{AS08}).
Thus $\log(\delta'_{\geq r})\in\mfg_{\eta',\x,r}$ is also regular semisimple.
By the orbital integral formula of Adler--Spice \cite[Lemma B.4]{AS09}, we get
\[
\int_{G_{\eta'}/G_{\eta'}\cap Z_{\G}}
\int_{\mcK_{\eta'}}
\mathbbm{1}_{G_{\eta',\x,r}}({}^{hc'}\delta'_{\geq r})\hat{\vartheta}({}^{hc'}\delta'_{\geq r}) \,dc'\,dh
=
\hat{\mu}^{\G_{\eta'}}_{X^{\ast}}\bigl(\log(\delta'_{\geq r})\bigr).
\]
\end{proof}

\begin{lem}\label{lem:index}
The natural surjective map
\[
S\backslash\bigl\{g\in G \mid {}^{g}\eta\in\t{S}\bigr\}/G_{\eta}
\twoheadrightarrow
K_{\sigma}\backslash\bigl\{g\in G \mid {}^{g}\eta\in{}^{G_{\x,0+}}\t{S}\bigr\}/G_{\eta}
\]
is bijective.
\end{lem}

\begin{proof}
Suppose that two double cosets $SgG_{\eta}$ and $Sg'G_{\eta}$ map to the same double coset $K_{\sigma}gG_{\eta}$.
Then, as $K_{\sigma}=SG_{\x,0+}$ and $S$ normalizes $G_{\x,0+}$, we may assume that $g'$ is given by $kg$ with some $k\in G_{\x,0+}$.
We write $\eta'\colonequals {}^{g}\eta$.
As we have $SgG_{\eta}=SG_{\eta'}g$ and $SkgG_{\eta}=SkG_{\eta'}g$, it suffices to show that $SG_{\eta'}=SkG_{\eta'}$ (for $\eta'\in\t{S}$ and $k\in G_{\x,0+}$ satisfying ${}^{k}\eta'\in\t{S}$).

Let $\eta'=\eta'_{0}\eta'_{+}$ and ${}^{k}\eta'={}^{k}\eta'_{0}{}^{k}\eta'_{+}$ be the topological Jordan decompositions induced from $\eta=\eta_{0}\eta_{+}$.
Let $p'$ be the order of $\eta_{0}$ modulo $A_{\t{\G}}$, which is prime-to-$p$.
By Lemma \ref{lem:head}, the conditions $\eta',{}^{k}\eta'\in\t{S}$ implies that $\eta'_{0},{}^{k}\eta'_{0}\in\t{S}$.
Thus there exists an element $s_{+}\in S$ such that ${}^{k}\eta'_{0}=s_{+}\eta'_{0}$.
By noting that 
\[
s_{+}
={}^{k}\eta'_{0}\cdot{\eta'_{0}}^{-1}
=k\cdot ({\eta'_{0}}k{\eta'_{0}}^{-1})^{-1} \in G_{\x,0+},
\]
$s_{+}$ belongs to $S\cap G_{\x,0+}=S_{0+}$.
Since the order of $\eta'_{0}$ and ${}^{k}\eta'_{0}$ modulo $A_{\t{\G}}$ is given by $p'$, we get
\[
1={}^{k}{\eta'_{0}}^{p'}
=\prod_{i=0}^{p'-1} [\eta'_{0}]^{i}(s_{+})\cdot {\eta'_{0}}^{p'}
=\prod_{i=0}^{p'-1} [\eta'_{0}]^{i}(s_{+})
\]
in $S/A_{\t{\G}}$.
Thus, by the same argument as in the proof of Lemma \ref{lem:sep-depth0} (i.e., using the vanishing of $H^{1}(\Z/p'\Z, \overline{S_{0+}})$), we can find an element $t_{+}\in S_{0+}$ satisfying $t_{+}s_{+}[\eta'_{0}](t_{+})^{-1}\in A_{\t{\G}}$, hence $t_{+}s_{+}[\eta'_{0}](t_{+})^{-1}\in A_{\t{\G},0+}$.
Then we have
\[
{}^{t_{+}k}\eta'_{0}
={}^{t_{+}}(s_{+}\eta'_{0})
=t_{+}s_{+}[\eta'_{0}](t_{+})^{-1}\cdot\eta'_{0}.
\]
Hence, by replacing $k$ with $t_{+}k$, we may assume that ${}^{k}\eta'_{0}=a\eta'_{0}$ for some $a\in A_{\t{\G},0+}$.
In other words, the image $\ol{k}$ of $k$ in $G_{\x,0+}/A_{\t{\G},0+}$ is fixed by $[\eta'_{0}]$.
We note that the short exact sequence
\[
1
\rightarrow
A_{\t{\G},0+}
\rightarrow
G_{\x,0+}
\rightarrow
G_{\x,0+}/A_{\t{\G},0+}
\rightarrow
1
\]
induces an exact sequence
\[
1
\rightarrow
A_{\t{\G},0+}^{\eta'_{0}}
\rightarrow
G_{\x,0+}^{\eta'_{0}}
\rightarrow
(G_{\x,0+}/A_{\t{\G},0+})^{\eta'_{0}}
\rightarrow
H^{1}(\langle[\eta'_{0}]\rangle, A_{\t{\G},0+})
\]
and that $H^{1}(\langle[\eta'_{0}]\rangle, A_{\t{\G},0+})$ vanishes since $\langle[\eta'_{0}]\rangle$ is of finite prime-to-$p$ order and $A_{\t{\G},0+}$ is a pro-$p$ group.
Thus we can find an element $k'\in G_{\x,0+}^{\eta'_{0}}$ whose image $\ol{k'}$ in $G_{\x,0+}/A_{\t{\G},0+}$ equals $\ol{k}$.
As we have $G_{\x,0+}^{\eta'_{0}}=G_{\eta'_{0},\x,0+}$ by Proposition \ref{prop:MP-descent} (2), this implies that $k$ belongs to $G_{\eta'_{0},\x,0+}A_{\t{\G},0+}=G_{\eta'_{0},\x,0+}$.

Now we utilize \cite[Lemma 9.10]{AS08}, which asserts that if
\begin{itemize}
\item
$(\G',\G)$ is a tame reductive $F$-sequence,
\item
$\gamma\in G$ is an element having a normal $r$-approximation,
\item
$\x\in\mcB(\bfC_{\G}^{(r)}(\gamma),F)\cap\mcB(\G',F)$,
\item
$k\in G_{\x,0+}$, and
\item
$Z_{\bfG}^{(r)}(\gamma)\subset G'$ and ${}^{k}Z_{\bfG}^{(r)}(\gamma)\subset G'$,
\end{itemize}
the element $k$ belongs to $G'_{\x,0+}C_{\bfG}^{(r)}(\gamma)_{\x,0+}$.
If we take 
\begin{itemize}
\item
$(\G,\G')\colonequals (\G_{\eta'_{0}},\bfS^{\nat})$,
\item
$\gamma\colonequals \eta'_{+}$ (then $\bfC_{\G_{\eta'_{0}}}^{(r)}(\gamma)=(\G_{\eta'_{0}})_{\eta'_{+}}=\G_{\eta'}$ by Lemma \ref{lem:cent-cent}),
\item
$\x$ to be the point $\x$ belonging to $\mcA(\bfS^{\nat},F)$, and
\item
$k\in G_{\eta'_{0},\x,0+}$,
\end{itemize}
then the assumptions of \cite[Lemma 9.10]{AS08} are satisfied.
Indeed, we have $Z_{\bfG}^{(r)}(\gamma)=Z_{\bfC_{\G}^{(r)}(\gamma)}=Z_{\G_{\eta'}}$.
As we have $\eta'\in\t{S}$, we have $G_{\eta'}\supset S^{\nat}$, hence $Z_{\bfG}^{(r)}(\gamma)\subset S^{\nat}$.
Similarly, we have ${}^{k}Z_{\bfG}^{(r)}(\gamma)\subset S^{\nat}$.
Thus we conclude that $k$ belongs to $S^{\nat}_{0+}G_{\eta',\x,0+} =G_{\eta',\x,0+}$.
In particular, we get $SG_{\eta'}=SkG_{\eta'}$.
\end{proof}

Now, by Theorem \ref{thm:TCF-0}, our task is to describe each summand of the right-hand side of \eqref{eq:thm:TCF-0}.
We next show the following proposition, which is a twisted version of \cite[Proposition 5.3.2]{AS09}:

\begin{prop}\label{prop:5.3.2}
For any element $\eta'\in\t{S}$ with a topological Jordan decomposition $\eta'=\eta'_{0}\eta'_{+}$, we have
\[
\Theta_{\t{\sigma}}(\eta')
=
\sum_{g\in S^{\nat}_{0+}G_{\eta'_{0},\x,s}\backslash[\![ \eta'_{+};\x,r]\!]_{G_{\eta'_{0}}}^{(s)}} \Theta_{\t{\rho}}({}^{g}\eta'),
\]
where $[\![ \eta'_{+};\x,r]\!]_{G_{\eta'_{0}}}^{(s)}$ is the subgroup as in \cite[Definition 6.6]{AS08} (taken in $G_{\eta'_{0}}$).
\end{prop}

\begin{proof}
Recall that the representation $\sigma$ of $SG_{\x,0+}$ is defined by inducing the representation $\rho$ of $SG_{\x,s}$.
Thus, by the Frobenius character formula for induced representations, we have 
\[
\Theta_{\t{\sigma}}(\eta')
=\sum_{\begin{subarray}{c}g\in SG_{\x,s}\backslash SG_{\x,0+} \\ {}^{g}\eta'\in \t{S}G_{\x,s} \end{subarray}} \Theta_{\t{\rho}}({}^{g}\eta')
=\sum_{\begin{subarray}{c}g\in S_{0+}G_{\x,s}\backslash G_{\x,0+} \\ {}^{g}\eta'\in \t{S}G_{\x,s} \end{subarray}} \Theta_{\t{\rho}}({}^{g}\eta').
\]

Let us show that the index set on the most-right-hand side can be replaced with the set $\{g\in S^{\nat}_{0+}G_{\eta'_{0},\x,s}\backslash G_{\eta'_{0},\x,0+} \mid {}^{g}\eta'_{+}\in S^{\nat}_{0+}G_{\eta'_{0},\x,s} \}$.
Suppose that a coset in $S_{0+}G_{\x,s}\backslash G_{\x,0+}$ contains an element $g$ satisfying ${}^{g}\eta'\in \t{S}G_{\x,s}$.
Then, since ${}^{g}\eta'={}^{g}\eta'_{0}{}^{g}\eta'_{+}$ gives a topological Jordan decomposition of ${}^{g}\eta'$, Proposition \ref{prop:Jordan} (4) implies that ${}^{g}\eta'_{0}\in\t{S}G_{\x,s}$.
On the other hand, as $\eta'$ belongs to $\t{S}$, we have $\eta'_{0}\in\t{S}$ (this also follows from Proposition \ref{prop:Jordan} (4)).
Thus there exists an element $g_{+}\in SG_{\x,s}$ satisfying ${}^{g}\eta'_{0}=g_{+}\eta'_{0}$.
As we have
\[
g_{+}={}^{g}\eta_{0}'\eta_{0}^{\prime-1}=g\cdot \eta'_{0}g^{-1}\eta_{0}^{\prime-1}\in G_{\x,0+},
\]
$g_{+}$ belongs to $SG_{\x,s}\cap G_{\x,0+}=S_{0+}G_{\x,s}$.
Then, by the same argument as in the proofs of Lemma \ref{lem:sep-depth0} or Lemma \ref{lem:index} using the vanishing of $H^{1}(\Z/p'\Z, S_{0+}G_{\x,s})$, we can find an element $k\in S_{0+}G_{\x,s}$ such that ${}^{kg}\eta'_{0}=\eta'_{0}$.
Therefore, by replacing $g$ with $kg$, we may assume that $g$ belongs to $G_{\x,0+}\cap G^{\eta'_{0}}=G_{\eta'_{0},\x,0+}$ (Proposition \ref{prop:MP-descent} (2)).
In other words, we may assume that our coset in $S_{0+}G_{\x,s}\backslash G_{\x,0+}$ comes from a coset in $S^{\nat}_{0+}G_{\eta'_{0},\x,s}\backslash G_{\eta'_{0},\x,0+}$ (note that $S_{0+}G_{\x,s}\cap G_{\eta'_{0},\x,0+}=S^{\nat}_{0+}G_{\eta'_{0},\x,s}$ by Proposition \ref{prop:MP-descent} (3)).
Furthermore, as ${}^{g}\eta'_{+}$ belongs to $S_{0+}G_{\x,s}$ and commutes with ${}^{g}\eta'_{0}=\eta'_{0}$, we get ${}^{g}\eta'_{+}\in S_{0+}G_{\x,s}\cap G^{\eta'_{0}}=S^{\nat}_{0+}G_{\eta'_{0},\x,s}$.

Now the same argument as in the proof of \cite[Proposition 5.3.2]{AS09} can be applied to the descended group $G_{\eta'_{0}}$.
(One of the most important inputs in the proof of \cite[Proposition 5.3.2]{AS09} is the $\hat{\vartheta}$-isotypicity of the representation. In the current situation, $\rho$ is $\hat{\vartheta}$-isotypic on $G_{\x,r}$, hence also on $G_{\eta'_{0},\x,r}$.)
Then we get
\[
\Theta_{\t\sigma}(\eta')
=
\sum_{\begin{subarray}{c}g\in S^{\nat}_{0+}G_{\eta'_{0},\x,s}\backslash G_{\eta'_{0},\x,0+} \\ {}^{g}\eta'_{+}\in S^{\nat}_{0+}G_{\eta'_{0},\x,s} \end{subarray}} \Theta_{\t{\rho}}({}^{g}\eta')
=
\sum_{g\in S^{\nat}_{0+}G_{\eta'_{0},\x,s}\backslash[\![ \eta'_{+};\x,r]\!]_{G_{\eta'_{0}}}^{(s)}} \Theta_{\t{\rho}}({}^{g}\eta').
\]
(We caution that our notation are different from those of Adler--Spice. Especially, the representation $\t{\rho}$ in \cite[Proposition 5.3.2]{DS18} is nothing but our $\rho_{(\bfS,\vartheta)}$. In this paper, the symbol $\sim$ basically means that a representation is twisted.)
\end{proof}

\begin{cor}\label{cor:5.3.2}
Let $\eta\in\t{S}$ be an element with a topological Jordan decomposition $\eta'=\eta'_{0}\eta'_{+}$.
Then we have
\[
\Theta_{\t{\sigma}}(\eta')
=
\Theta_{\t{\rho}}(\eta')
\cdot|\t{\mfG}_{\G_{\eta'_{0}}}(\vartheta,\eta'_{+})|
\cdot\mfG_{\G_{\eta'_{0}}}(\vartheta,\eta'_{+}),
\]
where $\t{\mfG}_{\G_{\eta'_{0}}}(\vartheta,\eta'_{+})$ and $\mfG_{\G_{\eta'_{0}}}(\vartheta,\eta'_{+})$ are the quantities defined in \cite[Definition 5.2.4]{AS09} (in the group $\G_{\eta'_{0}}$).
\end{cor}

\begin{proof}
By Proposition \ref{prop:5.3.2}, $\Theta_{\t{\sigma}}(\eta')$ is given by the sum of $\Theta_{\t{\rho}}({}^{g}\eta')$ over the set $g\in S^{\nat}_{0+}G_{\eta'_{0},\x,s}\backslash[\![ \eta'_{+};\x,r]\!]_{G_{\eta'_{0}}}^{(s)}$.
For any element $g\in[\![ \eta'_{+};\x,r]\!]_{G_{\eta'_{0}}}^{(s)}$, we have
\[
{}^{g}\eta'
=g\eta' g^{-1}
=\eta'\cdot\eta^{\prime-1}g\eta' g^{-1}
=\eta'\cdot[\eta^{\prime-1},g]
=\eta'\cdot[\eta_{+}^{\prime-1},g].
\]
As $g$ belongs to $[\![ \eta_{+};\x,r]\!]_{G_{\eta_{0}}}^{(s)}$, we know that $[\eta_{+}^{-1},g]$ belongs to $J_{+}$.
(Note that this fact is necessary so that \cite[Definition 5.2.4]{AS09} makes sense and essentially proved in \cite[Section 5]{AS09}; one can verify this property by using \cite[Lemmas 5.30 and 5.32]{AS08}).
By $\hat{\vartheta}$-isotypicity of $\rho$ on $J_{+}$, we have
\[
\sum_{g\in S^{\nat}_{0+}G_{\eta'_{0},\x,s}\backslash[\![ \eta'_{+};\x,r]\!]_{G_{\eta'_{0}}}^{(s)}} \Theta_{\t{\rho}}({}^{g}\eta')
=
\Theta_{\t{\rho}}(\eta')
\cdot\sum_{g\in S^{\nat}_{0+}G_{\eta'_{0},\x,s}\backslash[\![ \eta'_{+};\x,r]\!]_{G_{\eta'_{0}}}^{(s)}} \hat{\vartheta}([\eta_{+}^{\prime-1},g]).
\]
The sum on the right-hand side is nothing but $\t{\mfG}_{\G_{\eta'_{0}}}(\vartheta,\eta'_{+})$ by definition.
Since we have $\t{\mfG}_{\G_{\eta'_{0}}}(\vartheta,\eta'_{+})=|\t{\mfG}_{\G_{\eta'_{0}}}(\vartheta,\eta'_{+})|\cdot\mfG_{\G_{\eta'_{0}}}(\vartheta,\eta'_{+})$, we get the assertion.
\end{proof}

In summary, by combining Theorem \ref{thm:TCF-0} with Corollary \ref{cor:5.3.2}, we obtain the following:
\begin{thm}\label{thm:TCF-pre}
We have
\begin{align*}
\Theta_{\t{\pi}}(\delta)
=
\sum_{\begin{subarray}{c} g\in S\backslash G/G_{\eta} \\ {}^{g}\eta\in\t{S}\end{subarray}}
\Theta_{\t{\rho}}({}^{g}\eta)
\cdot|\t{\mfG}_{\G_{{}^{g}\eta_{0}}}(\vartheta,{}^{g}\eta_{+})|
\cdot\mfG_{\G_{{}^{g}\eta_{0}}}(\vartheta,{}^{g}\eta_{+})
\cdot
\hat{\mu}^{\G_{{}^{g}\eta}}_{X^{\ast}} \bigl(\log({}^{g}\delta_{\geq r})\bigr).
\end{align*}
\end{thm}

\section{Twisted Heisenberg--Weil representations}\label{sec:HW-twisted}
Let us keep the notation as in the previous section.
Our aim in this section is to compute $\Theta_{\t{\rho}}({}^{g}\eta)$ in each summand of the right-hand side of Theorem \ref{thm:TCF-pre}.
Recall that the representation $\rho=\rho_{(\bfS,\vartheta)}$ is defined by descending $\omega_{(\bfS,\vartheta)}\otimes(\vartheta\ltimes\mathbbm{1})$ from $S\ltimes J$ to $K=SJ$.
Hence, noting that ${}^{g}\eta\in\t{S}$, the computation of $\Theta_{\t{\rho}}({}^{g}\eta)$ is reduced to the computation of the twisted character of the Heisenberg--Weil representation $\omega_{(\bfS,\vartheta)}$.
For this, we repeat the computations in the proofs of \cite[Proposition 3.8]{AS09} and \cite[Proposition 4.21]{DS18} by taking the effect of the ``twist'' into consideration.

In the following, we additionally assume that
\[
\textbf{no restricted root of type $2$ or $3$ appears in $\Phi_{\res}(\G,\T)$.}
\]

\begin{rem}\label{rem:Prasad}
We believe that this assumption is not very harmless for our purpose, namely, study of the $\theta$-stable toral supercuspidal representations.
As explained in Remark \ref{rem:A_{2n}}, restricted roots of type $2$ or $3$ appear only when $\G$ contains a factor of type $A_{2n}$ on which some power of $\theta$ acts nontrivially.
However, for example, it is known that $\GL_{2n+1}$ does not have a $\theta$-stable irreducible supercuspidal representation for such $\theta$ whenever $p\neq2$ (see, e.g., \cite[Proposition 4]{Pra99}).
\end{rem}

\subsection{Structure of the Heisenberg quotient}\label{subsec:Heisen-structure}
We first recall the description of the group $J/J_{+}$ according to Adler--Spice (\cite[Proof of Proposition 3.8]{AS09}).
By fixing a finite tamely ramified extension $E$ of $F$ splitting $\bfS$, we put
\[
\bfV\colonequals \Lie(\bfS,\G)(E)_{\x,(r,s):(r,s+)}
\quad
\text{and}
\quad
V\colonequals \bfV^{\Gamma}.
\]\symdef{V}{$V$}
Recall that we have
\[
J/J_{+}
\cong
(S,G)_{\x,(r,s):(r,s+)}
=
(\bfS,\G)(E)_{\x,(r,s):(r,s+)}^{\Gamma}.
\]
Thus the exponential map $\Lie(\bfS,\G)(E)_{\x,(r,s):(r,s+)}\xrightarrow{\sim}(\bfS,\G)(E)_{\x,(r,s):(r,s+)}$ induces an identification
\[
V\xrightarrow{\sim}J/J_{+}.
\]

Let us investigate the space $V$ by using the root space decomposition of $\bmfg$ with respect to $\bfS$.
For $\alpha\in\Phi(\G,\bfS)$, we put $\bfV_{\alpha}$ to be the image of $\bmfg_{\alpha}(E)\cap\Lie(\bfS,\G)(E)_{\x,(r,s)}$ in $\Lie(\bfS,\G)(E)_{\x,(r,s):(r,s+)}$.
Then the root space decomposition $\bmfg\cong\bmfs\oplus\bigoplus_{\alpha\in\Phi(\G,\bfS)}\bmfg_{\alpha}$ naturally induces a decomposition
\[
\bfV
=
\bigoplus_{\alpha\in\Phi(\G,\bfS)}\bfV_{\alpha}.
\]
For each $\alpha\in\Phi(\G,\bfS)$, we put $V_{\alpha}\colonequals \bfV_{\alpha}^{\Gamma_{\alpha}}$ (recall that $\Gamma_{\alpha}$ is the stabilizer of $\alpha$ in $\Gamma$).
Here, note that $\bfV_{\alpha}$ and $V_{\alpha}$ could be zero depending on $\alpha\in\Phi(\G,\bfS)$.
We define a subset $\Xi(\G,\bfS)$ of $\Phi(\G,\bfS)$ by
\[
\Xi(\G,\bfS)\colonequals \{\alpha\in\Phi(\G,\bfS)\mid V_{\alpha}\neq0\}.
\]\symdef{Xi-G-S}{$\Xi(\G,\bfS)$}
Note that, for any $\alpha\in\Xi(\G,\bfS)$, the space $V_{\alpha}$ is (noncanonically) isomorphic to the residue field $k_{\alpha}$ of $F_{\alpha}$.
Also note that $\Xi(\G,\bfS)$ is preserved by the action of $\Sigma=\Gamma\times\{\pm1\}$ on $\Phi(\G,\bfS)$.
In the following, we simply write $\Xi$ for $\Xi(\G,\bfS)$.

For $\dot{\alpha}\in\dot{\Phi}(\G,\bfS)$, we put
\[
\bfV_{\dot{\alpha}}\colonequals \bigoplus_{\beta\in\dot{\alpha}}\bfV_{\beta}
\quad\text{and}\quad
V_{\dot{\alpha}}\colonequals \bfV_{\dot{\alpha}}^{\Gamma}.
\]
Then, for any $\dot{\alpha}\in\dot{\Phi}(\G,\bfS)$, we have
\[
V_{\alpha}
\xrightarrow{\sim}
V_{\dot{\alpha}}=\Bigl(\bigoplus_{\beta\in\dot{\alpha}}\bfV_{\beta}\Bigr)^{\Gamma}
\colon X_{\alpha} \mapsto \sum_{\sigma\in\Gamma/\Gamma_{\alpha}}\sigma(X_{\alpha}).
\]\symdef{V-alpha}{$V_{\alpha}$}
Therefore we get
\begin{align}\label{eq:decomp}
V
\cong
\bigoplus_{\dot{\alpha}\in\dot{\Xi}}V_{\dot{\alpha}}
=
\bigoplus_{\ddot{\alpha}\in\ddot{\Xi}}V_{\ddot{\alpha}},
\end{align}
where we put $V_{\ddot{\alpha}}\colonequals V_{\dot{\alpha}}\oplus V_{-\dot{\alpha}}$ for $\alpha\in\Xi_{\asym}$ and $V_{\ddot{\alpha}}\colonequals V_{\dot{\alpha}}$ for $\alpha\in\Xi_{\sym}$.\symdef{V-alpha-dot}{$V_{\dot{\alpha}}$}\symdef{V-alpha-ddot}{$V_{\ddot{\alpha}}$}

Recall that $J/J_{+}$ has a structure of a symplectic $\F_{p}$-vector space given by 
\[
(J/J_{+})\times (J/J_{+}) \rightarrow \mu_{p}\cong\F_{p}\colon (g,g')\mapsto\h\vartheta([g,g'])
\]
(see Section \ref{subsec:toral-sc}).
Under the identification $V\cong J/J_{+}$, this is transformed into the symplectic form on $V$ given by
\[
V\times V \rightarrow \F_{p}\colon
(X_{1},X_{2})\mapsto c\cdot\Tr_{k/\F_{p}}(\langle X^{\ast},[X_{1},X_{2}] \rangle),
\]
where $c\in \F_{p}^{\times}$ is a constant determined by the fixed identification $\mu_{p}\cong\F_{p}$.
Here, $\langle X^{\ast},[X_{1},X_{2}] \rangle\in k$ denotes the pairing of $X^{\ast}\in\mfs^{\ast}_{r:r+}$ with the $\mfs$-part of $[X_{1},X_{2}]\in\mfg_{\x,r:r+}$ (i.e., the trivial isotypic component with respect to the $\bfS$-action).
In fact, the above decomposition (\ref{eq:decomp}) gives an orthogonal decomposition of $V$ into symplectic subspaces.
Each symplectic subspace $V_{\ddot{\alpha}}$ is described as follows.

\begin{description}
\item[Asymmetric case]
Suppose that $\alpha\in\Xi_{\asym}$.
We put $V_{\pm\alpha}\colonequals V_{\alpha}\oplus V_{-\alpha}$.
Recall that the identification $V_{\alpha}\cong V_{\dot{\alpha}}\subset V$ is given by $X_{\alpha}\mapsto \sum_{\sigma\in\Gamma/\Gamma_{\alpha}}\sigma(X_{\alpha})$.
By noting that, for any $\alpha_{1},\alpha_{2}\in\Xi$, we have $\langle X^{\ast},[X_{\alpha_{1}},X_{\alpha_{2}}]\rangle\neq0$ only if $\alpha_{2}=-\alpha_{1}$, the resulting symplectic form $V_{\pm\alpha}\times V_{\pm\alpha} \rightarrow \F_{p}$ maps $(X_{\alpha}+X_{-\alpha},Y_{\alpha}+Y_{-\alpha})$ to 
\begin{align*}
&c\cdot\Tr_{k/\F_{p}}\Bigl(\sum_{\sigma\in\Gamma/\Gamma_{\alpha}}\sum_{\sigma'\in\Gamma/\Gamma_{\alpha}}\langle X^{\ast},[\sigma(X_{\alpha}+X_{-\alpha}), \sigma'(Y_{\alpha}+Y_{-\alpha})] \rangle\Bigr)\\
&=c\cdot \sum_{\tau\in\Gal(k/\F_p)}\tau\Bigl(\sum_{\sigma\in\Gamma/\Gamma_{\alpha}}\langle X^{\ast},\sigma([X_{\alpha},Y_{-\alpha}]+[X_{-\alpha},Y_{\alpha}]\rangle)\Bigr).
\end{align*}
Since $X^{\ast}$ is $F$-rational, this equals $c\cdot e_{\alpha}\cdot\Tr_{k_{\alpha}/\F_{p}}(\langle X^{\ast},[X_{\alpha},Y_{-\alpha}]+[X_{-\alpha},Y_{\alpha}] \rangle)$, where $e_{\alpha}$ denotes the ramification index of $F_\alpha/F$.

We recall that $V_{\alpha}\cong k_{\alpha}$.
Hence, by fixing nonzero elements $X_{\alpha}\in V_{\alpha}$ and $X_{-\alpha}\in V_{-\alpha}$ so that $X_{\alpha}\in V_{\alpha}$ and $X_{-\alpha}\in V_{-\alpha}$ are identified with $1\in k_{\alpha}$, we may think of the above symplectic form as the symplectic form on $k_{\alpha}\oplus k_{\alpha}$ which maps $(x_{+}+x_{-},y_{+}+y_{-})$ to
\[
\Tr_{k_{\alpha}/\F_{p}}(C\cdot (x_{+}y_{-}-x_{-}y_{+})),
\]
where we put $C\colonequals c\cdot e_{\alpha}\cdot \langle X^{\ast},[X_{\alpha},X_{-\alpha}]\rangle\in k_{\alpha}^{\times}$.\symdef{C}{$C$}\symdef{c}{$c$}\symdef{e-alpha}{$e_{\alpha}$}

\item[Symmetric case]
Suppose that $\alpha\in\Xi_{\sym}$.
Let $\tau_{\alpha}\in\Gamma/\Gamma_{\alpha}$ be the unique element satisfying $\tau_{\alpha}(\alpha)=-\alpha$.
By the same discussion as in the asymmetric case, we see that the symplectic form on $V_{\alpha}$ induced from that on $J/J_{+}$ is given by
\begin{align*}
(X_{\alpha},Y_{\alpha})
&\mapsto c\cdot\Tr_{k/\F_{p}}\Bigl(\sum_{\sigma\in\Gamma/\Gamma_{\alpha}}\sum_{\sigma'\in\Gamma/\Gamma_{\alpha}}\langle X^{\ast},[\sigma(X_{\alpha}), \sigma'(Y_{\alpha})] \rangle \Bigr)\\
&=c\cdot \sum_{\tau\in\Gal(k/\F_p)} \tau\Bigl(\sum_{\sigma\in\Gamma/\Gamma_{\alpha}}\langle X^{\ast},\sigma([X_{\alpha},\tau_{\alpha}(Y_{\alpha})]) \rangle\Bigr)\\
&=c\cdot e_{\alpha}\cdot\Tr_{k_{\alpha}/\F_{p}}(\langle X^{\ast},[X_{\alpha},\tau_{\alpha}(Y_{\alpha})] \rangle).
\end{align*}
By recalling that $V_{\alpha}\cong k_{\alpha}$ and fixing a nonzero element $X_{\alpha}\in V_{\alpha}$, we may think of the above symplectic form as the symplectic form on $k_{\alpha}$ which maps $(x,y)$ to
\[
\Tr_{k_{\alpha}/\F_{p}}(C\cdot x\tau_{\alpha}(y)),
\]
where we put $C\colonequals c\cdot e_{\alpha}\cdot [X_{\alpha},\tau_{\alpha}(X_{\alpha})]\in k_{\alpha}^{\times}$.
Note that $\tau_{\alpha}(C)=-C$.

\end{description}

Let us introduce one particular property of the set $\Xi$ deduced from the above description of the symplectic form:

\begin{lem}\label{lem:ram-empty}
The set $\Xi$ does not contain any symmetric ramified root.
\end{lem}

\begin{proof}
This fact is explained in the proof of \cite[Proposition 4.21]{DS18}.
For the sake of completeness, we explain it here.
Let $\alpha\in \Xi_{\sym}$.
Then, as explained above, we have $V_{\alpha}\cong k_{\alpha}$ and the symplectic form on $V_{\alpha}\times V_{\alpha}$ is given by 
\[
k_{\alpha}\times k_{\alpha}\rightarrow \F_{p}\colon 
(x,y)\mapsto \Tr_{k_{\alpha}/\F_{p}}(C\cdot x\tau_{\alpha}(y))
\]
with an element $C\in k_{\alpha}^{\times}$ satisfying $\tau_{\alpha}(C)=-C$.
If $\alpha$ is ramified, $\tau_{\alpha}$ acts trivially on $k_{\alpha}$, hence there cannot exist such an element $C$.
Thus $\alpha$ must be unramified.
\end{proof}

\subsection{Intertwiner of Heisenberg--Weil representations}\label{subsec:Heisen-int}

Recall that we fixed a topologically semisimple element $\ul{\eta}\in\t{S}$ (Section \ref{subsec:twist-int}).
Hence any element $\eta'\in\t{S}$ is written as $\eta'=s\cdot \ul{\eta}$ with a unique element $s\in S$.
Note that the action of $[\ul{\eta}]$ on $\bmfg$ induces an action on the set $\Phi(\G,\bfS)$ of order $l$, which does not depend on the choice of $\ul{\eta}\in\t{S}$.
By abuse of notation, let us write $\theta_{\bfS}$ for this action and $\Theta_{\bfS}$ for the group $\langle\theta_{\bfS}\rangle$ generated by $\theta_{\bfS}$.
To be more precise, for any $\alpha\in\Phi(\G,\bfS)$, $\theta_{\bfS}(\alpha)$ is the root given by $\theta_{\bfS}(\alpha)=\alpha\circ[\ul{\eta}]^{-1}$.
Whenever there is no risk of confusion, we abbreviate $\theta_{\bfS}(\alpha)$ and $\Theta_{\bfS}(\alpha)$ even as $\theta(\alpha)$ and $\Theta(\alpha)$, respectively.
We note that, for $X_{\alpha}\in\bmfg_{\alpha}$, $[\ul{\eta}](X_{\alpha})$ belongs to $\mfg_{\theta(\alpha)}$.
We also note that, as $\ul{\eta}$ is $F$-rational, the actions of $\Theta_{\bfS}$ and $\Sigma=\Gamma\times\{\pm1\}$ on $\Phi(\G,\bfS)$ commute.
Especially, the symmetry of $\Phi(\G,\bfS)$ is preserved by $\Theta_{\bfS}$.

Let us investigate the action $[\ul{\eta}]$ on $J/J_{+}$ through the isomorphism $V\cong J/J_{+}$ and the above decomposition (\ref{eq:decomp}) of $V$.
Note that $[\ul{\eta}]$ preserves the symplectic structure of $V$.
Indeed, for any $g,g'\in J/J_{+}$, we have
\[
\h\vartheta([[\ul{\eta}](g),[\ul{\eta}](g')])
=\h\vartheta([\ul{\eta} g\ul{\eta}^{-1},\ul{\eta} g'\ul{\eta}^{-1}])
=\h\vartheta([\ul{\eta}]([g,g']))
=\h\vartheta([g,g']).
\]
Moreover, each $V_{\ddot{\alpha}}$ is mapped onto $V_{\theta(\ddot{\alpha})}$, respectively.
In particular, the action of $\Theta_{\bfS}$ on $\Phi(\G,\bfS)$ preserves $\Xi$.

As in Section \ref{subsec:Steinberg}, for any $\alpha\in\Phi(\G,\bfS)$, we let $l_{\alpha}$ be the cardinality of the $\Theta_{\bfS}$-orbit of $\alpha$ (note that then $[\ul{\eta}]^{l_\alpha}(X_\alpha)=X_\alpha$ for any $X_\alpha\in \bmfg_\alpha$; see \cite[(1.3.5)]{KS99}).

\begin{defn}\label{defn:l'}
For $\alpha\in\Phi(\G,\bfS)$, let $m_{\alpha}$ be the order of $\Sigma\backslash(\Sigma\times\Theta_{\bfS})\alpha$.
In other words, $m_{\alpha}$ is the smallest positive integer such that $\theta^{m_{\alpha}}(\ddot{\alpha})=\ddot{\alpha}$ (hence $m_{\alpha}\mid l_{\alpha}$).
\end{defn}\symdef{m-alpha}{$m_{\alpha}$}

For $\ddot{\alpha}\in\ddot{\Xi}$, let us write $(\omega_{\ddot{\alpha}},W_{\ddot{\alpha}})$ for the Heisenberg--Weil representation of $\Sp(V_{\ddot{\alpha}})\ltimes\bbH(V_{\ddot{\alpha}})$ with central character given by $\hat{\vartheta}$, which is unique up to isomorphism.
Since the action $[\ul{\eta}]$ on $V$ induces an symplectic isomorphism from $V_{\ddot{\alpha}}$ to $V_{\theta(\ddot{\alpha})}$, an isomorphism from $\Sp(V_{\ddot{\alpha}})\ltimes\bbH(V_{\ddot{\alpha}})$ to $\Sp(V_{\theta(\ddot{\alpha})})\ltimes\bbH(V_{\theta(\ddot{\alpha})})$ is induced (for which we write $[\ul{\eta}]_{\ast}$).
Then the pull back $(\omega_{\theta(\ddot{\alpha})}^{\ul{\eta}},W_{\theta(\ddot{\alpha})})$ of the Heisenberg--Weil representation $W_{\theta(\ddot{\alpha})}$ of $\Sp(V_{\theta(\ddot{\alpha})})\ltimes\bbH(V_{\theta(\ddot{\alpha})})$ to $\Sp(V_{\ddot{\alpha}})\ltimes\bbH(V_{\ddot{\alpha}})$ via $[\ul{\eta}]_{\ast}$ is isomorphic to $(\omega_{\ddot{\alpha}},W_{\ddot{\alpha}})$.
For each $\ddot{\alpha}\in \ddot{\Xi}$, we fix an intertwiner
\[
I_{\ddot{\alpha}}^{\ul{\eta}}\colon \omega_{\ddot{\alpha}}\xrightarrow{\sim} \omega_{\theta(\ddot{\alpha})}^{\ul{\eta}}.
\]\symdef{I-alpha-eta}{$I_{\ddot{\alpha}}^{\ul{\eta}}$}

Recall that the representation $\omega=\omega_{(\bfS,\vartheta)}$ is the Heisenberg--Weil representation of $\Sp(V)\ltimes\bbH(V)$ with central character $\hat{\vartheta}$.
Since we have the decomposition \eqref{eq:decomp}, $\omega$ can be realized by tensoring Heisenberg--Weil representations $\omega_{\ddot{\alpha}}$ for $\ddot{\alpha}\in\ddot{\Xi}$ (see Section \ref{subsec:twisted-HW}).
Furthermore, by tensoring the fixed intertwiners $I_{\ddot{\alpha}}^{\ul{\eta}}$, we get an intertwiner between $\omega$ and its $[\ul{\eta}]_{\ast}$-twist $\omega^{\ul{\eta}}$.
Let us write $I_{\omega}^{\ul{\eta}}$ for the intertwiner obtained in this way:
\[
I_{\omega}^{\ul{\eta}}=I_{\omega_{(\bfS,\vartheta)}}^{\ul{\eta}}\colon\omega_{(\bfS,\vartheta)}\xrightarrow{\sim}\omega_{(\bfS,\vartheta)}^{\ul{\eta}}.
\]

Then we have the following:

\begin{prop}\label{prop:twisted-HW-Yu}
For each $\ddot{\alpha}\in\ddot{\Xi}$, we put
\[
I^{\ul{\eta}}_{\Theta(\ddot{\alpha})}
\colonequals 
I_{\theta^{m_{\alpha}-1}(\ddot{\alpha})}^{\ul{\eta}} \circ\cdots\circ I_{\ddot{\alpha}}^{\ul{\eta}}.
\]\symdef{I-Theta-alpha-eta}{$I_{\Theta(\ddot{\alpha})}^{\ul{\eta}}$}
Then, for any $\eta'=s\ul{\eta}\in\t{S}$, we have
\[
\tr\bigl(\omega([s])\circ I_{\omega}^{\ul{\eta}}\bigr)
=
\prod_{\ddot{\alpha}\in\Theta_{\bfS}\backslash\ddot{\Xi}}
\tr\Bigl(\omega_{\ddot{\alpha}}([\eta']^{m_{\alpha}}\circ[\ul{\eta}]^{-m_{\alpha}})\circ I^{\ul{\eta}}_{\Theta(\ddot{\alpha})}\Bigr).
\]
\end{prop}

\begin{proof}
Let us fix a set $\{\alpha_{0},\ldots,\alpha_{r}\}$ of representatives of $\Theta_{\bfS}\backslash\ddot{\Xi}$.
Then we can utilize the results of Section \ref{subsec:twisted-HW}, by taking $\iota\colonequals [\ul{\eta}]$, $l_{i}\colonequals m_{\alpha_{i}}-1$ for each $0\leq i\leq r$ and putting $V^{i}_{j}\colonequals V_{\theta^{j}(\ddot{\alpha}_{i})}$.
Since the symplectic automorphism $[s]$ preserves each $V^{i}_{j}$, Proposition \ref{prop:twisted-HW} implies that the trace of $\omega([s])\circ I_{\omega}^{\ul{\eta}}$ is given by
\[
\prod_{\ddot{\alpha}\in\Theta_{\bfS}\backslash\ddot{\Xi}}
\tr\Bigl(\omega_{\ddot{\alpha}}\bigl([s]\circ[\ul{\eta}]_{\ast}([s])\circ\cdots\circ[\ul{\eta}]_{\ast}^{m_{\alpha}-1}([s])\bigr)\circ I^{\ul{\eta}}_{\Theta(\ddot{\alpha})}\Bigr).
\]
By noting that $[\ul{\eta}]_{\ast}^{i}([s])=[\ul{\eta}]^{i}\circ[s]\circ[\ul{\eta}]^{-i}$, we have
\[
[s]\circ[\ul{\eta}]_{\ast}([s])\circ\cdots\circ[\ul{\eta}]_{\ast}^{m_{\alpha}-1}([s])
=([s]\circ[\ul{\eta}])^{m_{\alpha}}\circ[\ul{\eta}]^{-m_{\alpha}}
=[s\ul{\eta}]^{m_{\alpha}}\circ[\ul{\eta}]^{-m_{\alpha}}.
\]
Thus we get the desired result.
\end{proof}

We still have not specified the choice of each $I^{\ul{\eta}}_{\ddot{\alpha}}$ so far.
This means that also $I_{\omega}^{\ul{\eta}}$ still has an ambiguity of a scalar multiple.
Now we explain our choice of $I^{\ul{\eta}}_{\ddot{\alpha}}$.
For any $\ddot{\alpha}\in\ddot{\Xi}$, note that $I^{\ul{\eta}}_{\Theta(\ddot{\alpha})}$ is an automorphism of $W_{\ddot{\alpha}}$ such that 
\[
\xymatrix{
W_{\ddot{\alpha}}\ar^-{I^{\ul{\eta}}_{\Theta(\ddot{\alpha})}}[rr]\ar_-{\omega_{\ddot{\alpha}}(g,h)}[d]&&W_{\ddot{\alpha}}\ar^-{\omega_{\ddot{\alpha}}([\ul{\eta}]^{m_{\alpha}}_{\ast}(g,h))}[d]\\
W_{\ddot{\alpha}}\ar_-{I^{\ul{\eta}}_{\Theta(\ddot{\alpha})}}[rr]&&W_{\ddot{\alpha}}
}
\]
is commutative for any $(g,h)\in \Sp(V_{\ddot{\alpha}})\ltimes\bbH(V_{\ddot{\alpha}})$.
Note that $[\ul{\eta}]^{m_{\alpha}}$ is a symplectic automorphism of $V$ preserving $V_{\ddot{\alpha}}$.
Hence $I^{\ul{\eta}}_{\Theta(\ddot{\alpha})}$ must be a scalar multiple of the Heisenberg--Weil action $\omega_{\ddot{\alpha}}([\ul{\eta}]^{m_{\alpha}})$.
We choose $I^{\ul{\eta}}_{\ddot{\alpha}}$ for $\ddot{\alpha}\in\ddot{\Xi}$ so that we have
\[
I^{\ul{\eta}}_{\Theta(\ddot{\alpha})}
=
\omega_{\ddot{\alpha}}([\ul{\eta}]^{m_{\alpha}}).
\]

\begin{rem}
We eventually do not specify the choice of $I^{\ul{\eta}}_{\ddot{\alpha}}$ in a unique way.
Our normalization just requires that their composition $I^{\ul{\eta}}_{\Theta(\ddot{\alpha})}=I_{\theta^{m_{\alpha}-1}(\ddot{\alpha})}^{\ul{\eta}} \circ\cdots\circ I_{\ddot{\alpha}}^{\ul{\eta}}$ exactly coincides with $\omega_{\ddot{\alpha}}([\ul{\eta}]^{m_{\alpha}})$.
In fact, this is enough for our purpose, i.e., what affects the value of the twisted character is only $I^{\ul{\eta}}_{\Theta(\ddot{\alpha})}$, not each individual $I^{\ul{\eta}}_{\ddot{\alpha}}$, by Proposition \ref{prop:twisted-HW-Yu}.
\end{rem}

\begin{cor}\label{cor:twisted-HW-Yu}
With the above choice of an intertwiner $I_{\omega}^{\ul{\eta}}$, for any $\eta'=s\ul{\eta}\in\t{S}$ with a topological Jordan decomposition $\eta'_{0}\eta'_{+}$,
\[
\tr\bigl(\omega([s])\circ I_{\omega}^{\ul{\eta}}\bigr)
=
\prod_{\ddot{\alpha}\in\Theta_{\bfS}\backslash\ddot{\Xi}}
\Theta_{\omega_{\ddot{\alpha}}}([\eta'_{0}]^{m_{\alpha}}).
\]
\end{cor}

\begin{proof}
With the choice of an intertwiner $I_{\omega}^{\ul{\eta}}$ explained as above, we get
\[
\tr\bigl(\omega([s])\circ I_{\omega}^{\ul{\eta}}\bigr)
=
\prod_{\ddot{\alpha}\in\Theta_{\bfS}\backslash\ddot{\Xi}}
\Theta_{\omega_{\ddot{\alpha}}}([\eta']^{m_{\alpha}})
\]
by Proposition \ref{prop:twisted-HW-Yu}.
Noting that the topologically unipotent part $\eta'_{+}$ acts on $V_{\dot{\alpha}}$ and $V_{\ddot{\alpha}}$ trivially via conjugation, we get the assertion.
\end{proof}

\subsection{Descent of the Heisenberg quotient}\label{subsec:descent-Yu}
Recall that we have fixed an elliptic regular semisimple element $\delta\in\t{S}$ and written $\eta$ for $\delta_{<r}$ so far.
However, to discuss more generally, we temporarily (until the end of Section \ref{subsec:twisted-Weil-summary}) let $\eta\in\t{S}$ denote any topologically semisimple element.
(Our intent here is to use the symbol $\eta$ instead of $\eta'_0$ in Corollary \ref{cor:twisted-HW-Yu} to make the notation lighter.)

In Section \ref{subsec:Steinberg}, we introduced the notion of a restricted root.
Although we discussed it for $\Phi(\G,\T)$ in Section \ref{subsec:Steinberg}, the same also holds for $\Phi(\G,\bfS)$, i.e., we have the set of restricted roots $\Phi_{\res}(\G,\bfS)$ equipped with a natural map
\[
\Phi(\G,\bfS)
\twoheadrightarrow\Phi(\G,\bfS)/\Theta_{\bfS}
\xrightarrow{1:1}\Phi_{\res}(\G,\bfS)
\colon \alpha\mapsto\alpha_{\res}.
\]
Note that, since $\Phi_{\res}(\G,\bfS)$ carries a Galois action induced from that of $\Phi(\G,\bfS)$, we can also discuss the symmetry of a restricted root.
For any $\alpha\in\Phi(\G,\bfS)$, we put $\Gamma_{\alpha_{\res}}$ to be the stabilizer in $\Gamma$ of the restricted root $\alpha_{\res}$ (or, equivalently, the $\Theta_{\bfS}$-orbit $\Theta_{\bfS}\alpha$ of $\alpha$):
\[
\Gamma_{\alpha_{\res}}
\colonequals \{\sigma\in\Gamma\mid \sigma(\alpha_{\res})=\alpha_{\res}\}
=\{\sigma\in\Gamma\mid \sigma(\Theta_{\bfS}\alpha)=\Theta_{\bfS}\alpha\}.
\]
Similarly, we put $\Gamma_{\pm\alpha_{\res}}$ to be the stabilizer in $\Gamma$ of the set $\{\pm\alpha_{\res}\}$:
\[
\Gamma_{\pm\alpha_{\res}}
\colonequals \{\sigma\in\Gamma\mid \sigma(\{\pm\alpha_{\res}\})=\{\pm\alpha_{\res}\}\}
=\{\sigma\in\Gamma\mid \sigma(\{\pm\Theta_{\bfS}\alpha\})=\{\pm\Theta_{\bfS}\alpha\}\}.
\]
Let $F_{\alpha_{\res}}$ and $F_{\pm\alpha_{\res}}$ denote the subfields of $\ol{F}$ fixed by $\Gamma_{\alpha_{\res}}$ and $\Gamma_{\pm\alpha_{\res}}$, respectively.
\[
\xymatrix@R=10pt{
F_{\alpha_{\res}}\ar@{}[r]|*{\subset}&F_{\alpha}&&\Gamma_{\alpha_{\res}}\ar@{}[d]|{\bigcap}&\Gamma_{\alpha}\ar@{}[l]|*{\supset}\ar@{}[d]|{\bigcap}\\
F_{\pm\alpha_{\res}}\ar@{}[r]|*{\subset}\ar@{}[u]|{\bigcup}&F_{\pm\alpha}\ar@{}[u]|{\bigcup}&&\Gamma_{\pm\alpha_{\res}}&\Gamma_{\pm\alpha}\ar@{}[l]|*{\supset}
}
\]
Note that it can be easily check that $\Gamma_\alpha$ is normal in $\Gamma_{\pm\alpha_\res}$ (and the same is true for any two subgroups in this diagram).
Also note that, when $\alpha$ is symmetric, we have $\Gamma_{\alpha_\res}\cap\Gamma_{\pm\alpha}=\Gamma_\alpha$; this follows from that the action $\theta_\bfS$ on $\Phi(\bfG,\bfS)$ preserves a set of positive roots, especially, $\alpha$ cannot be sent to $-\alpha$.

Recall that $m_\alpha$ is the smallest positive integer such that $\theta^{m_\alpha}(\alpha)\in\ddot{\alpha}$.
We fix an element $\sigma_\alpha\in\Gamma_{\pm\alpha_\res}/\Gamma_\alpha$ as follows:
\begin{enumerate}
    \item When $\alpha$ is asymmetric, we choose $\sigma_\alpha$ such that $\theta^{m_\alpha}(\alpha)=\pm\sigma_\alpha(\alpha)$. Note that, as $\alpha$ is asymmetric, exactly one of $\theta^{m_\alpha}(\alpha)$ or $-\theta^{m_\alpha}(\alpha)$ belongs to $\dot{\alpha}$. In other words, the sign ``$\pm$'' in $\theta^{m_\alpha}(\alpha)=\pm\sigma_\alpha(\alpha)$ is determined by $\alpha$.
    \item When $\alpha$ is symmetric, we have $\theta^{m_\alpha}(\alpha)\in\ddot{\alpha}=\dot{\alpha}$. Noting this, we choose $\sigma_\alpha$ such that $\theta^{m_\alpha}(\alpha)=\sigma_\alpha(\alpha)$.
\end{enumerate}
When $\alpha$ is symmetric, we also fix a nontrivial element $\tau_\alpha\in\Gamma_{\pm\alpha}/\Gamma_{\alpha}$.

\begin{lem}
\begin{enumerate}
    \item When $\alpha$ is asymmetric, the extension $F_{\alpha}/F_{\pm\alpha_{\res}}$ is cyclic of order $l_\alpha/m_\alpha$ whose Galois group is given by $\langle\sigma_{\alpha}\rangle$.
    \item When $\alpha$ is symmetric, the extension $F_{\alpha}/F_{\alpha_{\res}}$ is cyclic of order $l_\alpha/m_\alpha$ whose Galois group is given by $\langle\sigma_{\alpha}\rangle$. Moreover, $\tau_\alpha$ and $\sigma_\alpha$ commute and generate the Galois group of $F_{\alpha}/F_{\pm\alpha_\res}$.
\end{enumerate}
\end{lem}

\begin{proof}
%
Let us show (1); suppose that $\alpha$ is asymmetric.
We take any element $\sigma\in\Gamma_{\pm\alpha_{\res}}$ and suppose that $\sigma(\alpha)=\pm\theta_{\bfS}^{i}(\alpha)$ for some $i$.
Then, as $m_{\alpha}$ is the smallest positive integer satisfying $\theta_{\bfS}^{m_{\alpha}}(\alpha)\in\ddot{\alpha}$, we have $i\equiv jm_{\alpha} \pmod{l_{\alpha}}$ for some $j\in\Z$.
Hence we have $\sigma(\alpha)=\sigma_{\alpha}^{j}(\alpha)$.
This implies that $\sigma^{-1}\sigma_{\alpha}^{j}$ belongs to $\Gamma_{\alpha}$.
In other words, $\sigma_{\alpha}$ generates $\Gamma_{\alpha_{\res}}/\Gamma_{\alpha}$.

We can also show (2) in a similar argument.
\end{proof}

We put $f_{\alpha}$ to be the residue degree of the extension $F_{\alpha}/F_{\alpha_{\res}}$.\symdef{f-alpha}{$f_{\alpha}$}

\begin{lem}\label{lem:symm-rest}
Let $\alpha\in\Phi(\bfG,\bfS)$.
    \begin{enumerate}
    \item If $\alpha$ is asymmetric and $\alpha_\res$ is symmetric, then $\alpha_{\res}$ is
    \[
    \begin{cases}
    \text{unramified} & \text{if $[k_{\alpha}:k_{\pm\alpha_\res}]$ is even,} \\
    \text{ramified} & \text{if $[k_{\alpha}:k_{\pm\alpha_\res}]$ is odd.}
    \end{cases}
    \]
    \item If $\alpha$ is symmetric unramified, then $\alpha_\res$ is symmetric and 
    \[
    \begin{cases}
    \text{unramified} & \text{if $f_\alpha$ is odd,} \\
    \text{ramified} & \text{if $f_\alpha$ is even.}
    \end{cases}
    \]
    \item If $\alpha$ is symmetric ramified, then $\alpha_{\res}$ is symmetric ramified.
\end{enumerate}
\end{lem}

\begin{proof}

We first show (1).
Let $\alpha$ be an asymmetric root whose $\alpha_\res$ is symmetric.
As shown above, $\sigma_\alpha$ generates the Galois group of $F_{\alpha}/F_{\pm\alpha_\res}$, hence also that of $k_{\alpha}/k_{\pm\alpha_\res}$.
Similarly, $\sigma_\alpha^2$ generates the Galois group of $F_{\alpha}/F_{\alpha_\res}$, hence also that of $k_{\alpha}/k_{\alpha_\res}$.
Hence $k_{\alpha_\res}=k_{\pm\alpha_\res}$ if and only if $k_{\alpha}/k_{\pm\alpha_\res}$ is of odd degree.

We next show (2) and (3).
Let $\alpha$ be a symmetric root.
Since $\tau_{\alpha}(\alpha_{\res})=-\alpha_{\res}$, $\alpha_{\res}$ is also symmetric.
Note that we have $k_{\alpha_{\res}}\cap k_{\pm\alpha}=k_{\pm\alpha_{\res}}$ and $k_{\alpha}=k_{\alpha_{\res}}k_{\pm\alpha}$.

First suppose that $\alpha$ is unramified.
If $f_{\alpha}=[k_{\alpha}:k_{\alpha_{\res}}]$ is odd, then $k_{\pm\alpha}$ does not contain $k_{\alpha_{\res}}$.
Hence $k_{\alpha_{\res}}\cap k_{\pm\alpha}=k_{\pm\alpha_{\res}}$ is strictly smaller than $k_{\alpha_{\res}}$, which implies that $\alpha_{\res}$ is unramified.
On the other hand, if $f_{\alpha}$ is even, then $k_{\pm\alpha}$ contains $k_{\alpha_{\res}}$.
Hence $k_{\alpha_{\res}}\cap k_{\pm\alpha}=k_{\pm\alpha_{\res}}$ is equal to $k_{\alpha_{\res}}$, which implies that $\alpha_{\res}$ is ramified.

Let us next suppose that $\alpha$ is ramified.
In this case, as we have $k_{\alpha_{\res}}\cap k_{\pm\alpha}=k_{\alpha_{\res}}\cap k_{\alpha}=k_{\alpha_{\res}}$, we get $k_{\alpha_{\res}}=k_{\pm\alpha_{\res}}$.
Thus $\alpha_{\res}$ is ramified.
\end{proof}

As reviewed in Section \ref{subsec:Steinberg}, the group $\G_{\eta}$ is a connected reductive group with a maximal torus $\bfS^{\nat}$.
Furthermore, $\Phi(\G_{\eta},\bfS^{\nat})$ is regarded as a subset of $\Phi_{\res}(\G,\bfS)$.
By Proposition \ref{prop:BT-descent}, we may suppose that the point $\x$ is contained in $\mcA(\bfS^{\nat},F)\subset\mcB(\G_{\eta},F)$.
We introduce the subgroups $J_{\eta}$ (resp.\ $J_{\eta,+}$) in the same way as $J$ (resp.\ $J_{+}$) by using $(\G_{\eta},\bfS^{\nat},\x,r,s(+))$ instead of $(\G,\bfS,\x,r,s(+))$, i.e.,
\[
J_{\eta}\colonequals (S^{\nat},G_{\eta})_{\x,(r,s)}
\quad\text{and}\quad
J_{\eta,+}\colonequals (S^{\nat},G_{\eta})_{\x,(r,s+)}.
\]
By the same discussion as in Section \ref{subsec:Heisen-structure}, if we put
\[
\bfV_{\eta}\colonequals \Lie(\bfS_{\eta},\G_{\eta})(E)_{\x,(r,s):(r,s+)}
\quad
\text{and}
\quad
V_{\eta}\colonequals \bfV_{\eta}^{\Gamma}
\]
then we have $J_{\eta}/J_{\eta,+}\cong V_{\eta}$ and a root space decomposition similar to (\ref{eq:decomp}):
\[
V_{\eta}
\cong
\bigoplus_{\ddot{\alpha}_{\res}\in\ddot{\Xi}_{\eta}}V_{\eta,\ddot{\alpha}_{\res}},
\]
where we use the notation defined in the same way as in Section \ref{subsec:Heisen-structure}, e.g., 
\[
\Xi_{\eta}
\colonequals \Xi(\bfG_{\eta},\bfS^{\nat})
\colonequals \{\alpha_{\res}\in\Phi(\G_{\eta},\bfS^{\nat}) \mid V_{\eta,\alpha_{\res}}\neq0\}.
\]

By Proposition \ref{prop:MP-descent} (4), we have a natural identification
\[
V_{\eta}
\cong(S^{\nat},G_{\eta})_{\x,(r,s):(r,s+)}
\cong(S,G)_{\x,(r,s):(r,s+)}^{\eta}
\cong V^{\eta},
\]
where $(S,G)_{\x,(r,s):(r,s+)}^{\eta}$ and $V^{\eta}$ denote the set of $[\eta]$-fixed points in $(S,G)_{\x,(r,s):(r,s+)}$ and $V$, respectively.
Let us investigate this identification more precisely.
The Lie algebra $\bmfg_{\eta}$ of $\G_{\eta}$ is naturally identified with the $[\eta]$-fixed points $\bmfg^{\eta}$ of the Lie algebra $\bmfg$ of $\G$.
If $\alpha_{\res}\in\Phi(\G_{\eta},\bfS^{\nat})$ is a restricted root obtained from $\alpha\in\Phi(\G,\bfS)$, then the root subspace $\bmfg_{\eta,\alpha_{\res}}$ of $\bmfg_{\eta}$ is identified with the $[\eta]$-fixed points in the sum $\bigoplus_{\alpha'\in\Theta\alpha}\bmfg_{\alpha'}$ of root subspaces of $\bmfg$:
\[
\bmfg_{\eta,\alpha_{\res}}
\cong
\Bigl(\bigoplus_{\alpha'\in\Theta\alpha}\bmfg_{\alpha}\Bigr)^{\eta}.
\]
This induces an identification
\[
V_{\eta,\ddot{\alpha}_{\res}}
\cong
\Bigl(\bigoplus_{\ddot{\alpha}'\in\Sigma\backslash(\Sigma\times\Theta)\alpha} V_{\ddot{\alpha}'}\Bigr)^{\eta}
\]
for any $\alpha_{\res}\in\Xi_{\eta}$.
Let us put $V_{\Theta(\ddot{\alpha})}\colonequals \bigoplus_{\ddot{\alpha}'\in\Sigma\backslash(\Sigma\times\Theta)\alpha} V_{\ddot{\alpha}'}$.
In particular, by letting $\Xi_{\res}$ be the set of restricted roots associated to $\Xi$, the set $\Xi_{\eta}$ can be thought of as the set of restricted roots $\alpha_{\res}\in\Xi_{\res}$ such that the $[\eta]$-action has a nonzero fixed point in $V_{\Theta(\ddot{\alpha})}$.
\[
\xymatrix@R=10pt{
\Phi(\G,\bfS)\ar@{->>}[r]&\Phi(\G,\bfS)/\Theta_{\bfS} \ar^-{1:1}[r]&\Phi(\G,\bfS)_{\res}&\Phi(\G_{\eta},\bfS^{\nat})\ar@{}[l]|*{\supset}\\
\Xi\ar@{->>}[r]\ar@{}[u]|{\bigcup}&\Xi/\Theta_{\bfS} \ar^-{1:1}[r]\ar@{}[u]|{\bigcup}&\Xi_{\res}\ar@{}[u]|{\bigcup}&\Xi_{\eta}\ar@{}[l]|*{\supset}\ar@{}[u]|{\bigcup}
}
\]

\subsection{Twisted characters of Weil representations: asymmetric roots}\label{subsec:twisted-Weil-asym}
Let $\alpha\in\Xi_{\asym}$.
We compute $\Theta_{\omega_{\ddot{\alpha}}}([\eta]^{m_{\alpha}})$, which constitutes the right-hand side of Corollary \ref{cor:twisted-HW-Yu}.
Recall that $V_{\ddot{\alpha}}=V_{\dot{\alpha}}\oplus V_{-\dot{\alpha}}\cong V_{\alpha}\oplus V_{-\alpha}$, where $V_{\alpha}$ and $V_{-\alpha}$ are $1$-dimensional $k_{\alpha}$-vector spaces, which are identified with $k_{\alpha}$ by fixing nonzero elements $X_{\alpha}\in V_{\alpha}$ and $X_{-\alpha}\in V_{-\alpha}$ (see Section \ref{subsec:Heisen-structure}).

We carry out a case-by-case computation on the symmetry of $\alpha_\res$.
Recall that we have exactly one of $\theta^{m_\alpha}(\alpha)\in\dot{\alpha}$ or $\theta^{m_\alpha}(\alpha)\in-\dot{\alpha}$; we have fixed a $\sigma_\alpha$ satisfying $\theta^{m_\alpha}(\alpha)=\sigma_\alpha(\alpha)$ or $-\sigma_\alpha(\alpha)$ according to it.
In the former case where $\theta^{m_\alpha}(\alpha)=\sigma_\alpha(\alpha)$, $\sigma_\alpha$ belongs to $\Gamma_{\alpha_\res}$, which implies that $\alpha_\res$ is asymmetric.
On the other hand, in the latter case where $\theta^{m_\alpha}(\alpha)=-\sigma_\alpha(\alpha)$, the Galois group of $F_\alpha/F_{\alpha_\res}$ is a proper subgroup of that of $F_\alpha/F_{\pm\alpha_\res}$ generated by $\sigma_\alpha^2$, which implies that $\alpha_\res$ is symmetric.
Thus, also recalling Lemma \ref{lem:symm-rest}, we have the following possibilities:
\begin{itemize}
    \item[(1)] $\alpha_\res$ is asymmetric ($\theta^{m_\alpha}(\alpha)=\sigma_\alpha(\alpha)$);
    \item[(2-I)] $\alpha_\res$ is symmetric ($\theta^{m_\alpha}(\alpha)=-\sigma_\alpha(\alpha)$) and unramified;
    \item[(2-II)] $\alpha_\res$ is symmetric ($\theta^{m_\alpha}(\alpha)=-\sigma_\alpha(\alpha)$) and ramified.
\end{itemize}

To make the notation lighter, we write $\varsigma_\alpha$ for $\sigma_\alpha^{-1}$ in the following.

\subsubsection{The case where $\alpha_\res$ is asymmetric}

We first consider the case where $\alpha_\res$ is asymmetric ($\theta^{m_\alpha}(\alpha)=\sigma_\alpha(\alpha)$).
In this case, $F_{\alpha}/F_{\alpha_\res}$ is a cyclic extension of degree $l_\alpha/m_\alpha$ and $F_{\alpha_{\res}}=F_{\pm\alpha_{\res}}$.
We write $g_\alpha:=[k_{\pm\alpha_\res}:\F_p]$.
\[
\xymatrix{
F_{\alpha_{\res}}\ar^-{l_\alpha/m_\alpha}_-{\langle\sigma_\alpha\rangle}@{-}[r]&F_{\alpha}\\
F_{\pm\alpha_{\res}}\ar@{-}[r]\ar@{=}[u]&F_{\pm\alpha}\ar@{=}[u]
}
\qquad
\xymatrix{
&k_{\alpha_{\res}}\ar^-{f_\alpha}@{-}[r]&k_{\alpha}\\
\F_p\ar^-{g_\alpha}@{-}[r]&k_{\pm\alpha_{\res}}\ar@{-}[r]\ar@{=}[u]&k_{\pm\alpha}\ar@{=}[u]
}
\]

The action of $[\eta]^{m_\alpha}$ on $V_{\ddot{\alpha}}$ preserves $V_{\dot{\alpha}}$ and $V_{-\dot{\alpha}}$.
Since $[\eta]^{m_\alpha}(X_{\alpha})$ belongs to $V_{\theta^{m_\alpha}(\alpha)}=V_{\sigma_{\alpha}(\alpha)}$, the induced action of $[\eta]^{m_\alpha}$ on $V_{\alpha}$ is $\varsigma_\alpha$-linear:
\[
\xymatrix{
V_{\alpha}\ar^-{\cong}[r]& V_{\dot{\alpha}} \ar^-{[\eta]^{m_\alpha}}[d]& X_{\alpha}\ar@{|->}[r]& \sum_{\sigma\in\Gamma/\Gamma_{\alpha}} \sigma(X_{\alpha})\ar@{|->}[d]\\
V_{\alpha}& V_{\dot{\alpha}} \ar^-{\cong}[l]& [\eta]^{m_\alpha}\circ\varsigma_\alpha(X_{\alpha})& \sum_{\sigma\in\Gamma/\Gamma_{\alpha}} [\eta]^{m_\alpha}\circ\sigma(X_{\alpha}) \ar@{|->}[l]
}
\]
More explicitly, if we let $\eta_{\alpha}$ (resp.\ $\eta_{-\alpha}$) be the element of $k_{\alpha}^{\times}$ such that $[\eta]^{m_\alpha}\circ\varsigma_\alpha(X_{\alpha})=\eta_{\alpha}X_{\alpha}$ (resp.\ $[\eta]^{m_\alpha}\circ\varsigma_\alpha(X_{-\alpha})=\eta_{-\alpha}X_{-\alpha}$), then $[\eta]$ is given by 
\[
x_{+}+x_{-}
\mapsto
\eta_{\alpha}\varsigma_\alpha(x_{+})+\eta_{-\alpha}\varsigma_\alpha(x_{-})
\]
as an element of $\Sp(V_{\ddot{\alpha}})\cong\Sp(k_{\alpha}\oplus k_{\alpha})$.
Hence, by Corollary \ref{cor:Gerardin}, we get
\begin{align*}
\Theta_{\omega_{\ddot{\alpha}}}([\eta]^{m_\alpha})
&=\sgn_{\F_{p}^{\times}}(\det(\eta_{\alpha}\circ\varsigma_\alpha \mid k_{\alpha}))\cdot |V_{\ddot{\alpha}}^{\eta^{m_\alpha}}|^{\frac{1}{2}}\\
&=\sgn_{\F_{p}^{\times}}(\det(\varsigma_{\alpha}\mid k_{\alpha}))\cdot\sgn_{k_{\alpha}^{\times}}(\eta_{\alpha})\cdot|V_{\ddot{\alpha}}^{\eta^{m_\alpha}}|^{\frac{1}{2}}.
\end{align*}

Since $\varsigma_\alpha$ is a generator of the Galois group of the cyclic extension $k_\alpha/k_{\alpha_\res}$ of degree $f_\alpha$, its determinant as an $k_{\alpha_\res}$-automorphism is given by $(-1)^{f_\alpha-1}$.
(Indeed, after splitting $k_\alpha\otimes_{k_{\alpha_\res}}\overline{\F}_p\cong \prod_{i=0}^{f_\alpha-1}\overline{\F}_p$, the element $\varsigma_\alpha$ is identified with the cyclic permutation matrix of size $f_\alpha$.)
Hence we have $\sgn_{\F_{p}^{\times}}(\det(\varsigma_{\alpha}\mid k_{\alpha}))=\sgn_{\F_{p}^{\times}}(-1)^{g_\alpha(f_\alpha-1)}$.
In summary, 
\[
\Theta_{\omega_{\ddot{\alpha}}}([\eta]^{m_\alpha})
=\sgn_{\F_{p}^{\times}}(-1)^{g_\alpha(f_\alpha-1)}\cdot\sgn_{k_{\alpha}^{\times}}(\eta_{\alpha})\cdot|V_{\ddot{\alpha}}^{\eta^{m_\alpha}}|^{\frac{1}{2}}.
\]

\subsubsection{The case where $\alpha_\res$ is symmetric}

We next consider the case where $\alpha_\res$ is symmetric ($\theta^{m_\alpha}(\alpha)=-\sigma_\alpha(\alpha)$).
In this case, $F_{\alpha}/F_{\pm\alpha_\res}$ is a cyclic extension of degree $l_\alpha/m_\alpha$ and $F_{\alpha_{\res}}/F_{\pm\alpha_{\res}}$ is quadratic.
\[
\xymatrix{
F_{\alpha_{\res}}\ar^-{l_\alpha/(2m_\alpha)}_-{\langle\sigma_\alpha^2\rangle}@{-}[r]&F_{\alpha}\\
F_{\pm\alpha_{\res}}\ar@{-}[r]\ar^-{2}_-{\langle\sigma_\alpha\rangle}@{-}[u]&F_{\pm\alpha}\ar@{=}[u]
}
\qquad
\xymatrix{
k_{\alpha_{\res}}\ar^-{f_\alpha}@{-}[r]&k_{\alpha}\\
k_{\pm\alpha_{\res}}\ar@{-}[r]\ar^-{\leq2}@{-}[u]&k_{\pm\alpha}\ar@{=}[u]
}
\]

The action of $[\eta]^{m_\alpha}$ on $V_{\ddot{\alpha}}$ swaps $V_{\dot{\alpha}}$ and $V_{-\dot{\alpha}}$.
Since $[\eta]^{m_\alpha}(X_{\alpha})$ belongs to $V_{\theta^{m_\alpha}(\alpha)}=V_{-\sigma_{\alpha}(\alpha)}$, the isomorphism from $V_{\alpha}$ to $V_{-\alpha}$ induced by $[\eta]^{m_\alpha}$ is $\varsigma_\alpha$-linear:
\[
\xymatrix{
V_{\alpha}\ar^-{\cong}[r]& V_{\dot{\alpha}} \ar^-{[\eta]^{m_\alpha}}[d]& X_{\alpha}\ar@{|->}[r]& \sum_{\sigma\in\Gamma/\Gamma_{\alpha}} \sigma(X_{\alpha})\ar@{|->}[d]\\
V_{-\alpha}& V_{-\dot{\alpha}} \ar^-{\cong}[l]& [\eta]^{m_\alpha}\circ\varsigma_\alpha(X_{\alpha})& \sum_{\sigma\in\Gamma/\Gamma_{\alpha}} [\eta]^{m_\alpha}\circ\sigma(X_{\alpha}) \ar@{|->}[l]
}
\]
More explicitly, if we let $\eta_{-\alpha}$ (resp.\ $\eta_{\alpha}$) be the element of $k_{\alpha}^{\times}$ such that $[\eta]^{m_\alpha}\circ\varsigma_\alpha(X_{\alpha})=\eta_{-\alpha}X_{-\alpha}$ (resp.\ $[\eta]^{m_\alpha}\circ\varsigma_\alpha(X_{-\alpha})=\eta_{\alpha}X_{\alpha}$), then $[\eta]^{m_\alpha}$ is given by 
\[
x_{+}+x_{-}
\mapsto
\eta_{\alpha}\varsigma_\alpha(x_{-})+\eta_{-\alpha}\varsigma_\alpha(x_{+})
\]
as an element of $\Sp(V_{\ddot{\alpha}})\cong\Sp(k_{\alpha}\oplus k_{\alpha})$.
Since this automorphism preserves the symplectic form $(x_{+}+x_{-},y_{+}+y_{-}) \mapsto\Tr_{k_{\alpha}/\F_{p}}(C\cdot(x_{+}y_{-}-x_{-}y_{+}))$ (see Section \ref{subsec:Heisen-structure}), we must have
\[
\Tr_{k_{\alpha}/\F_{p}}(C\cdot(x_{+}y_{-}-x_{-}y_{+}))
=
\Tr_{k_{\alpha}/\F_{p}}(C\cdot\eta_{\alpha}\eta_{-\alpha}\cdot\varsigma_\alpha(x_{-}y_{+}-x_{+}y_{-}))
\]
for any $x_{+},x_{-},y_{+},y_{-}\in k_{\alpha}$.
In other words, we have $\eta_{\alpha}\eta_{-\alpha}=-\varsigma_\alpha(C)\cdot C^{-1}$.

To compute $\Theta_{\omega_{\ddot{\alpha}}}([\eta]^{m_\alpha})$, we utilize G{\'e}rardin's formula (Proposition \ref{prop:Gerardin2}). 
Note that $[\eta]^{m_\alpha}$ is of finite prime-to-$p$ order, hence a semisimple element of $\Sp(V_{\ddot{\alpha}})$.
We let $T$ be an $\F_p$-rational maximal torus of $\Sp(V_{\ddot{\alpha}})$ containing $[\eta]^{m_\alpha}$ (note that such a $T$ may not be unique, so we fix one).
Following the notation of Section \ref{subsec:Gerardin}, we write $P(V_{\ddot{\alpha}}, T)$ for the set of weights of $T$ in $V_{\ddot{\alpha}}$, which is a finite set equipped with an action of $\Sigma_{\F_p}:=\pm\Gamma_{\F_p}$ (the notion of being symmetric or asymmetric for the weights are defined with respect to this action).
Then Proposition \ref{prop:Gerardin2} gives 
\begin{align}\label{eq:Gerardin}
\Theta_{\omega_{\ddot{\alpha}}}([\eta]^{m_\alpha})
=
(-1)^{l(V_{\ddot{\alpha}},T;[\eta]^{m_\alpha})}\cdot p^{N(V_{\ddot{\alpha}};[\eta]^{m_\alpha})}\cdot\chi^{T}([\eta]^{m_\alpha}).
\end{align}
Here, 
\begin{itemize}
\item
$l(V_{\ddot{\alpha}},T;[\eta]^{m_\alpha})\colonequals |\{\omega\in P(V_{\ddot{\alpha}},T)/\Gamma_{\F_{p}} \mid \text{$\epsilon([\eta]^{m_\alpha}) \neq1$ for an(y) $\epsilon\in\omega$}\}|$,
\item
$N(V_{\ddot{\alpha}};[\eta]^{m_\alpha})\colonequals \frac{1}{2}\dim_{\F_{p}}{V_{\ddot{\alpha}}^{\eta^{m_\alpha}}}$ (hence $p^{N(V_{\ddot{\alpha}};[\eta]^{m_\alpha})}=|V_{\ddot{\alpha}}^{\eta^{m_\alpha}}|^{\frac{1}{2}}$),
\item 
$\chi^{T}\colonequals \prod_{\Omega\in P(V,T)/\Sigma_{\F_{p}}}\chi^{T}_{\Omega}$, where
\[
\chi^{T}_{\Omega}
:=
\begin{cases}
\epsilon(t)^{\frac{1-q_{\Omega}}{2}}&\text{if an(y) $\omega\subset\Omega$ is asymmetric,}\\
\epsilon(t)^{\frac{1+q_{\Omega}}{2}}&\text{if an(y) $\omega\subset\Omega$ is symmetric,}
\end{cases}
\]
($q_{\Omega}\colonequals p^{\frac{1}{2}|\Omega|}$).

\end{itemize}

Therefore, our task is to find an $\F_p$-rational maximal torus $T$ of $\Sp(V_{\ddot{\alpha}})$ containing $[\eta]^{m_\alpha}$ and compute the set $P(V_{\ddot{\alpha}}, T)$ including its Galois action.
For this, noting that any semisimple conjugacy class of $\Sp(V_{\ddot{\alpha}})$ is uniquely determined by the multi-set of eigenvalues (Lemma \ref{lem:eigenvalue-conjugate}), we first determine the multi-set of eigenvalues of $[\eta]^{m_\alpha}$ as an $\F_p$-automorphism of $V_{\ddot{\alpha}}$.

Let us introduce a notation for convenience: for any finite-dimensional $K$-vector space $W$ (where $K$ is any field) and its $K$-endomorphism $\varphi$, we let $\mcE_{K}(\varphi\mid W)$ denote the multi-set of the eigenvalues of $\varphi$.\symdef{E-K-varphi-W}{$\mcE_{K}(\varphi\mid W)$}

The following lemma can be proved by a straightforward computation (cf.\ Lemma \ref{lem:eigen-s-th-root}):
\begin{lem}
Let $K$ be any field and $V=V_{1}\oplus V_{2}$ a finite-dimensional $K$-vector space equipped with isomorphisms $A_{1}\colon V_{1}\rightarrow V_{2}$ and $A_{2}\colon V_{2}\rightarrow V_{1}$.
If we put $A\colonequals A_{1}\oplus A_{2}$, then we have
\[
\det(T\cdot \id_{V}-A_{1}\oplus A_{2} \mid V)
=\det(T^{2}\cdot \id_{V_1}-A_{2}\circ A_{1} \mid V_{1}).
\]
In particular, $\mcE_K(A\mid V)=\mcE_K(A_2\circ A_1\mid V_1)^{1/2}$, where $(-)^{1/2}$ on the right-hand side denotes the multi-set of square roots of elements of $(-)$.
\end{lem}

By this lemma, we see
\[
\mcE_{\F_p}([\eta]^{m_\alpha}\mid k_{\alpha}\oplus k_{\alpha})
=
\mcE_{\F_p}((\eta_{\alpha}\circ\varsigma_\alpha)\circ(\eta_{-\alpha}\circ\varsigma_\alpha) \mid k_{\alpha})^{1/2}.
\]
Note that 
\[
(\eta_{\alpha}\circ\varsigma_\alpha)\circ(\eta_{-\alpha}\circ\varsigma_\alpha)
=\eta_{\alpha}\cdot\varsigma_\alpha(\eta_{-\alpha})\circ \varsigma_\alpha^{2}
=-\varsigma_\alpha(\eta_{-\alpha}C)/(\eta_{-\alpha}C)\circ\varsigma_\alpha^{2}
\]
and that this is a $k_{\alpha_\res}$-linear automorphism of $k_\alpha$.
Thus
\[
\mcE_{\F_p}((\eta_{\alpha}\circ\varsigma_\alpha)\circ(\eta_{-\alpha}\circ\varsigma_\alpha) \mid k_{\alpha})
=
\Gal(k_{\alpha_{\res}}/\F_p)\cdot\mcE_{k_{\alpha_{\res}}}(-\varsigma_\alpha(\eta_{-\alpha}C)/(\eta_{-\alpha}C)\circ\varsigma_\alpha^{2} \mid k_{\alpha}).
\]
Here, the right-hand side denotes the multi-set consisting of the elements of the form $\sigma(e)$, where $\sigma\in \Gal(k_{\alpha_{\res}}/\F_p)$ and $e\in \mcE_{k_{\alpha_{\res}}}(-\varsigma_\alpha(\eta_{-\alpha}C)/(\eta_{-\alpha}C)\circ\varsigma_\alpha^{2} \mid k_{\alpha})$ ($\sigma$ acts on $e$ by implicitly fixing an extension of $\sigma$ to $\overline{\F}_p$; the resulting multi-set does not depend on the choice of such extensions.)

We put $\gamma:=\varsigma_\alpha(\eta_{-\alpha}C)/(\eta_{-\alpha}C)\in k_\alpha^\times$.
Note that $\varsigma_\alpha^2$ is a generator of the Galois group of $k_{\alpha}/k_{\alpha_\res}$.
Under the identification
\[
k_\alpha\otimes_{k_{\alpha_\res}}\overline{\F}_p\cong \prod_{i=0}^{f_\alpha-1}\overline{\F}_p
\colon x\otimes y \mapsto ((\varsigma_\alpha^2)^{i}(x)\cdot y)_{i},
\]
the $(-\gamma)$-multiplication map on $k_\alpha$ is represented by the diagonal matrix 
\[
    \mathrm{diag}\bigl(-\gamma,(\varsigma_\alpha^2)(-\gamma),\dots,(\varsigma_\alpha^{2(f_\alpha-1)})(-\gamma)\bigr)
\]
and $\varsigma_\alpha^{2}$ is represented by the permutation matrix of length $f_\alpha$.
Hence, by Lemma \ref{lem:eigen-s-th-root}, we get
\begin{align*}
    \mcE_{k_{\alpha_{\res}}}(-\varsigma_\alpha(\eta_{-\alpha}C)/(\eta_{-\alpha}C)\circ\varsigma_\alpha^{2} \mid k_{\alpha})
    &=
    \{-\gamma\cdot(\varsigma_\alpha^2)(-\gamma)\cdots(\varsigma_\alpha^{2(f_\alpha-1)})(-\gamma)\}^{1/f_\alpha}\\
    &=
    \{\Nr_{k_\alpha/k_{\alpha_\res}}(-\gamma)\}^{1/f_\alpha}.
\end{align*}

In summary, we get
\[
    \mcE_{\F_p}([\eta]^{m_\alpha}\mid k_{\alpha}\oplus k_{\alpha})
    =
    \Gal(k_{\alpha_{\res}}/\F_p)\cdot\{\Nr_{k_\alpha/k_{\alpha_\res}}(-\gamma)\}^{\frac{1}{2f_\alpha}}.
\]

\noindent{\textbf{(I) Unramified case.}} 

We first consider the case where $\alpha_{\res}$ is unramified.
\[
\xymatrix{
k_{\alpha_{\res}}\ar^-{f_\alpha}@{-}[r]&k_{\alpha}\\
k_{\pm\alpha_{\res}}\ar@{-}[r]\ar^-{2}@{-}[u]&k_{\pm\alpha}\ar@{=}[u]
}
\]

In the following, we explain an algorithm to pick up an $\F_p$-rational maximal torus $T$ of $\Sp(V_{\ddot{\alpha}})$ containing $[\eta]^{m_\alpha}$.

Let us write $\delta:=\Nr_{k_\alpha/k_{\alpha_\res}}(-\gamma)$.
As $\gamma=\varsigma_\alpha(\eta_{-\alpha}C)/(\eta_{-\alpha}C)$, we have $\Nr_{k_\alpha/k_{\pm\alpha_\res}}(-\gamma)=1$.
This implies that $\delta\in k_{\alpha_\res}^1$.
Hence, by Lemma \ref{lem:finite-field} (1), $\delta$ has a square root in $k_{\alpha_\res}$.
If we let $\beta\in k_{\alpha_\res}$ be a square root of $\delta$, then we have
\[
    \mcE_{\F_p}([\eta]^{m_\alpha}\mid k_{\alpha}\oplus k_{\alpha})
    =
    \Gal(k_{\alpha_{\res}}/\F_p)\cdot\{\beta\}^{\frac{1}{f_\alpha}}\sqcup \Gal(k_{\alpha_{\res}}/\F_p)\cdot\{-\beta\}^{\frac{1}{f_\alpha}}.
\]

\begin{description}
\item[(1) The case where $\Nr_{k_{\alpha_\res}/k_{\pm\alpha_\res}}(\beta)=-1$]

We first consider the case where $\Nr_{k_{\alpha_\res}/k_{\pm\alpha_\res}}(\beta)=-1$; in other words, we have $\sigma_\alpha(\beta)=-\beta^{-1}$.
Hence,
\begin{align*}
    \mcE_{\F_p}([\eta]^{m_\alpha}\mid k_{\alpha}\oplus k_{\alpha})
    &=
    \Gal(k_{\alpha_{\res}}/\F_p)\cdot\{\beta\}^{\frac{1}{f_\alpha}}\sqcup \Gal(k_{\alpha_{\res}}/\F_p)\cdot\{-\beta\}^{\frac{1}{f_\alpha}}\\
    &=
    \Gal(k_{\alpha_{\res}}/\F_p)\cdot\{\beta\}^{\frac{1}{f_\alpha}}\sqcup \Gal(k_{\alpha_{\res}}/\F_p)\cdot\{\beta^{-1}\}^{\frac{1}{f_\alpha}}.
\end{align*}
We define field extensions $k_1,\ldots,k_r$ of $k_{\alpha_\res}$ so that
\[
    k_{\alpha_\res}[X]/(X^{f_\alpha}-\beta)
    \cong
    \bigoplus_{i=1}^r k_i.
\]
Note that the polynomial $X^{f_\alpha}-\beta$ is reduced since $f_\alpha$ is prime-to-$p$ by assumption ($f_\alpha$ divides $l_\alpha$, hence divides $l$, which is assumed to be prime to $p$).
Also note that, by Lemma \ref{lem:finite-field} (1), each $k_i$ is contained in $k_\alpha$.
If we let $(x_i)_i\in \bigoplus_{i=1}^r k_i$ be the image of $X\in k_{\alpha_\res}[X]/(X^{f_\alpha}-\beta)$ under the above identification, we have
\[
    \Gal(k_{\alpha_{\res}}/\F_p)\cdot\{\beta\}^{\frac{1}{f_\alpha}}
    =
    \bigsqcup_{i=1}^{r}\{\sigma(x_i) \mid \sigma\in\Gal(k_i/\F_p)\}.
\]
Hence,
\[
    \mcE_{\F_p}([\eta]^{m_\alpha}\mid k_{\alpha}\oplus k_{\alpha})
    =
    \bigsqcup_{i=1}^{r}\{\sigma(x_i)^{\pm1} \mid \sigma\in\Gal(k_i/\F_p)\}.
\]

Therefore, if we choose an $\F_p$-rational maximal torus $T$ of $\Sp(V_{\ddot{\alpha}})$ to be $\prod_{i=1}^{r}k_i^\times$, then its element $(x_i)_i$ has the same multi-set of eigenvalues as $[\eta]^{m_\alpha}$ (see Lemma \ref{lem:Sp-tori}).

The partition of $P(V_{\ddot{\alpha}},T)$ into $\Gamma_{\F_p}$-orbit is given by 
\[
    \bigsqcup_{i=1}^{r}\{\sigma(-)\mid \sigma\in\Gal(k_i/\F_p)\}
    \sqcup
    \bigsqcup_{i=1}^{r}\{\sigma(-)^{-1}\mid \sigma\in\Gal(k_i/\F_p)\}
\]
as described in Lemma \ref{lem:Sp-tori} (note that each $\Gamma_{\F_p}$-orbit is asymmetric).
In particular, the number $l(V_{\ddot{\alpha}},T;[\eta]^{m_\alpha})$ is always even, hence the sign $(-1)^{l(V_{\ddot{\alpha}},T;[\eta]^{m_\alpha})}$ in the formula \eqref{eq:Gerardin} is trivial.
Since $\epsilon(-)^{\frac{1-q_\Omega}{2}}$ for each $\Omega=\{\sigma(-)^{\pm1}\mid \sigma\in\Gal(k_i/\F_p)\}$ is nothing but the unique quadratic character of $k_i^\times$, we have $\chi_T=\prod_{i=1}^{r}\sgn_{k_i^\times}$.
In summary, we get
\[
\Theta_{\omega_{\ddot{\alpha}}}([\eta]^{m_\alpha})
=
\prod_{i=1}^{r}\sgn_{k_i^\times}(x_i)\cdot |V_{\ddot{\alpha}}^{\eta^{m_\alpha}}|^{\frac{1}{2}}.
\]

\item[(2) The case where $\Nr_{k_{\alpha_\res}/k_{\pm\alpha_\res}}(\beta)=1$ and $f_\alpha$ is odd]

We next consider the case where $\Nr_{k_{\alpha_\res}/k_{\pm\alpha_\res}}(\beta)=1$ and $f_\alpha$ is odd.
Let us again define field extensions $k_1,\ldots,k_r$ of $k_{\alpha_\res}$ so that
\[
    k_{\alpha_\res}[X]/(X^{f_\alpha}-\beta)
    \cong
    \bigoplus_{i=1}^r k_i
\]
and define $(x_i)_i$ to be the image of $X$.
By Lemma \ref{lem:finite-field} (1), each $k_i$ is contained in $k_{\alpha}$.
Similarly, we define field extensions $k'_1,\ldots,k'_s$ of $k_{\alpha_\res}$
\[
    k_{\alpha_\res}[X]/(X^{f_\alpha}+\beta)
    \cong
    \bigoplus_{j=1}^{s} k'_j
\]
and define $(x'_j)_j$ to be the image of $X$.
Then we have 
\begin{align*}
    &\mcE_{\F_p}([\eta]^{m_\alpha}\mid k_{\alpha}\oplus k_{\alpha})\\
    &=
    \Gal(k_{\alpha_{\res}}/\F_p)\cdot\{\beta\}^{\frac{1}{f_\alpha}}\sqcup \Gal(k_{\alpha_{\res}}/\F_p)\cdot\{-\beta\}^{\frac{1}{f_\alpha}}\\
    &=
    \bigsqcup_{i=1}^{r}\{\sigma(x_i) \mid \sigma\in\Gal(k_i/\F_p)\}\sqcup \bigsqcup_{j=1}^{s}\{\sigma(x'_i) \mid \sigma\in\Gal(k'_j/\F_p)\}.
\end{align*}
Since $f_\alpha$ is odd, each $x_i\in k_i\subset k_\alpha$ automatically belongs to $k_{\alpha}^{1}$ by Lemma \ref{lem:finite-field} (2); this implies that $x_i\in k_i^1$.
Similarly, we also get $x'_j\in (k'_j)^1$.

Therefore, if we choose an $\F_p$-rational maximal torus $T$ of $\Sp(V_{\ddot{\alpha}})$ to be $\prod_{i=1}^{r}k_i^1\times \prod_{j=1}^{s}(k'_j)^1$, then its element $((x_i)_i,(x'_j)_j)$ has the same multi-set of eigenvalues as $[\eta]^{m_\alpha}$ (see Lemma \ref{lem:Sp-tori}).

The partition of $P(V_{\ddot{\alpha}},T)$ into $\Gamma_{\F_p}$-orbit is given by 
\[
    \bigsqcup_{i=1}^{r}\{\sigma(-)\mid \sigma\in\Gal(k_i/\F_p)\}
    \sqcup
    \bigsqcup_{j=1}^{s}\{\sigma(-)\mid \sigma\in\Gal(k'_j/\F_p)\}
\]
as described in Lemma \ref{lem:Sp-tori} (each $\Gamma_{\F_p}$-orbit is symmetric).
Note that $\epsilon([\eta]^{m_\alpha})=1$ for some $\epsilon\in P(V_{\ddot{\alpha}},T)$ if and only if $\beta=\pm1$, which is equivalent to that $\delta=1$.
(This is furthermore equivalent to that $V_{\ddot{\alpha}}^{\eta^{m_\alpha}}\neq0$.)
In this case, there exists a unique $\Gamma_{\F_p}$-orbit of such weights, hence the number $l(V_{\ddot{\alpha}},T;[\eta]^{m_\alpha})$ is given by $r+s-1$.
Otherwise, we have $l(V_{\ddot{\alpha}},T;[\eta]^{m_\alpha})=r+s$.
Since $\epsilon(-)^{\frac{1+q_\Omega}{2}}$ for each $\Omega=\{\sigma(-)\mid \sigma\in\Gal(k_i/\F_p)\}$ is nothing but the unique quadratic character of $k_i^1$ (and the same is true for $(k'_j)^1$), we have $\chi_T=\prod_{i=1}^{r}\sgn_{k_i^1}\cdot \prod_{j=1}^{s}\sgn_{(k'_j)^1}$.
In summary, we get
\[
\Theta_{\omega_{\ddot{\alpha}}}([\eta]^{m_\alpha})
=
(-1)^{r+s-n_\alpha(\eta)}\cdot\prod_{i=1}^{r}\sgn_{k_i^1}(x_i)\cdot \prod_{j=1}^{s}\sgn_{{k'_j}^1}(x'_j)\cdot |V_{\ddot{\alpha}}^{\eta^{m_\alpha}}|^{\frac{1}{2}},
\]
where $n_\alpha(\eta):=1$ if $V_{\ddot{\alpha}}^{\eta^{m_\alpha}}\neq0$ and $n_\alpha(\eta):=0$ if $V_{\ddot{\alpha}}^{\eta^{m_\alpha}}=0$.

\item[(3) The case where $\Nr_{k_{\alpha_\res}/k_{\pm\alpha_\res}}(\beta)=1$ and $f_\alpha$ is even]

We finally consider the case where $\Nr_{k_{\alpha_\res}/k_{\pm\alpha_\res}}(\beta)=1$ and $f_\alpha$ is even.
We first note that, again by Lemma \ref{lem:finite-field} (1), we can find a square root $\beta^{\frac{1}{2}}$ of $\beta\in k_{\alpha_\res}$ in $k_{\alpha_\res}$.
Let us write 
\[
  \Gal(k_{\alpha_{\res}}/\F_p)\cdot\{\beta\}^{\frac{1}{f_\alpha}}
  =
  \Gal(k_{\alpha_{\res}}/\F_p)\cdot\{\beta^{\frac{1}{2}}\}^{\frac{1}{f_\alpha^\circ}}\sqcup\Gal(k_{\alpha_{\res}}/\F_p)\cdot\{-\beta^{\frac{1}{2}}\}^{\frac{1}{f_\alpha^\circ}},
\]
where $f_\alpha^\circ=\frac{f_\alpha}{2}$.
Since $\Nr_{k_{\alpha_\res}/k_{\pm\alpha_\res}}(\beta)=1$, we must have $\Nr_{k_{\alpha_\res}/k_{\pm\alpha_\res}}(\beta^{\frac{1}{2}})=\pm1$.

\begin{itemize}
\item If $\Nr_{k_{\alpha_\res}/k_{\pm\alpha_\res}}(\beta^{\frac{1}{2}})=-1$, then we can apply the discussion in the case (1) to find a subtorus of $\Sp(V_{\ddot{\alpha}})$ with an element whose multi-set of eigenvalues is $\Gal(k_{\alpha_{\res}}/\F_p)\cdot\{\beta\}^{\frac{1}{f_\alpha}}$.
\item If $\Nr_{k_{\alpha_\res}/k_{\pm\alpha_\res}}(\beta^{\frac{1}{2}})=1$ and $f_\alpha^\circ$ is odd, then we can apply the discussion in the case (2) to find a subtorus of $\Sp(V_{\ddot{\alpha}})$ with an element whose multi-set of eigenvalues is $\Gal(k_{\alpha_{\res}}/\F_p)\cdot\{\beta\}^{\frac{1}{f_\alpha}}$.
\item If $\Nr_{k_{\alpha_\res}/k_{\pm\alpha_\res}}(\beta^{\frac{1}{2}})=1$ and $f_\alpha^\circ$ is even, then we repeat the discussion so far by furthermore replacing $\beta^{\frac{1}{2}}$ and $f_\alpha^\circ$ with $\beta^{\frac{1}{4}}$ and $\frac{f_\alpha^\circ}{2}$. Repeating this procedure, we can finally find a subtorus of $\Sp(V_{\ddot{\alpha}})$ with an element whose multi-set of eigenvalues is $\Gal(k_{\alpha_{\res}}/\F_p)\cdot\{\beta\}^{\frac{1}{f_\alpha}}$.
\end{itemize}

By applying the same consideration to $\Gal(k_{\alpha_{\res}}/\F_p)\cdot\{-\beta\}^{\frac{1}{f_\alpha}}$, in total, we can find an $\F_p$-rational maximal torus $T$ of $\Sp(V_{\ddot{\alpha}})$ with an element whose multi-set of eigenvalues is the same as $[\eta]^{m_\alpha}$.
\end{description}

As this algorithm shows, the computation of a maximal torus ``$T$'' is particularly simple when $f_\alpha$ is odd (because then only the cases (1) and (2) happen).
In fact, from a practical point of view, even assuming $f_{\alpha}=1$ is sufficient to cover a wide range of cases.
For example, if $\theta$ is an involution ($l=2$), the degree of $F_{\alpha}/F_{\pm\alpha_\res}$ is at most $2$. Thus, if $F_{\alpha_\res}/F_{\pm\alpha_\res}$ is quadratic, then $F_{\alpha}/F_{\alpha_\res}$ is necessarily trivial.
    

For our application, let us examine the case where $f_\alpha=1$.
Note that, in this case, we have $k_{\alpha}=k_{\alpha_\res}$ and $\delta=N_{k_\alpha/k_{\alpha_\res}}(-\gamma)=-\gamma$.

\begin{description}
\item[(1) The case where $\Nr_{k_{\alpha_\res}/k_{\pm\alpha_\res}}(\beta)=-1$]

In this case, we have
\begin{itemize}
    \item $r=1$ and $k_1=k_{\alpha}$,
    \item $x_1=\beta=\delta^{\frac{1}{2}}$.
\end{itemize}
Hence we get 
\[
    \Theta_{\omega_{\ddot{\alpha}}}([\eta]^{m_\alpha})
    =
    \sgn_{k_{\alpha}^\times}(\beta)\cdot |V_{\ddot{\alpha}}^{\eta^{m_\alpha}}|^{\frac{1}{2}}.
\]
Note that, if we let $q_\alpha:=|k_\alpha|$ and $q_\alpha^\circ:=|k_{\alpha_\res}|$, then 
\[
    \sgn_{k_{\alpha}^\times}(\beta)
    =\beta^{\frac{q_\alpha-1}{2}}
    =(\beta^{q_\alpha^\circ+1})^{\frac{q_\alpha^\circ-1}{2}}
    =(-1)^{\frac{q_\alpha^\circ-1}{2}}
\]
(we used that $\Nr_{k_{\alpha_\res}/k_{\pm\alpha_\res}}(\beta)=\beta^{q_\alpha^\circ+1}=-1$ in the last equality).
On the other hand, 
\begin{align*}
    \sgn_{k_{\alpha}^1}(\gamma)
    &=\sgn_{k_{\alpha}^1}(-1)\cdot \sgn_{k_{\alpha}^1}(-\gamma)\\
    &=(-1)^{\frac{q_\alpha^\circ+1}{2}}\cdot \delta^{\frac{q_\alpha^\circ+1}{2}}\\
    &=(-1)^{\frac{q_\alpha^\circ+1}{2}}\cdot \beta^{q_\alpha^\circ+1}
    =(-1)^{\frac{q_\alpha^\circ-1}{2}}
\end{align*}
(we again used that $\Nr_{k_{\alpha_\res}/k_{\pm\alpha_\res}}(\beta)=-1$ in the last equality).
By furthermore recalling that $\gamma=\varsigma_\alpha(\eta_{-\alpha}C)/(\eta_{-\alpha}C)$ (hence, $\sgn_{k_{\alpha}^\times}(\eta_{-\alpha}C)=\sgn_{k_{\alpha}^1}(\gamma)$), we finally obtain
\[
    \Theta_{\omega_{\ddot{\alpha}}}([\eta]^{m_\alpha})
    =
    \sgn_{k_{\alpha}^\times}(\eta_{-\alpha}C)\cdot |V_{\ddot{\alpha}}^{\eta^{m_\alpha}}|^{\frac{1}{2}}.
\]

\item[(2) The case where $\Nr_{k_{\alpha_\res}/k_{\pm\alpha_\res}}(\beta)=1$]

In this case, we have
\begin{itemize}
    \item $r=s=1$,
    \item $k_1=k'_1=k_{\alpha}$,
    \item $x_1=\beta=\delta^{\frac{1}{2}}$ and $x'_1=-\beta=-\delta^{\frac{1}{2}}$.
\end{itemize}
Hence we get 
\[
    \Theta_{\omega_{\ddot{\alpha}}}([\eta]^{m_\alpha})
    =
    (-1)^{n_\alpha(\eta)}\cdot\sgn_{k_{\alpha}^1}(\beta)\cdot\sgn_{k_{\alpha}^1}(-\beta)\cdot |V_{\ddot{\alpha}}^{\eta^{m_\alpha}}|^{\frac{1}{2}},
\]
where $n_\alpha(\eta)=1$ if $V_{\ddot{\alpha}}^{\eta^{m_\alpha}}\neq0$ and $n_\alpha(\eta)=0$ if $V_{\ddot{\alpha}}^{\eta^{m_\alpha}}=0$.
By noting that
\[
    \sgn_{k_{\alpha}^1}(\beta)\cdot\sgn_{k_{\alpha}^1}(-\beta)
    =\sgn_{k_{\alpha}^1}(-\beta^2)
    =\sgn_{k_{\alpha}^1}(\gamma)
    =\sgn_{k_{\alpha}^\times}(\eta_{-\alpha}C),
\]
we obtain
\[
    \Theta_{\omega_{\ddot{\alpha}}}([\eta]^{m_\alpha})
    =
    (-1)^{n_\alpha(\eta)}\cdot\sgn_{k_{\alpha}^\times}(\eta_{-\alpha}C)\cdot |V_{\ddot{\alpha}}^{\eta^{m_\alpha}}|^{\frac{1}{2}}.
\]
\end{description}

We furthermore note that, in the former case ($\Nr_{k_{\alpha_\res}/k_{\pm\alpha_\res}}(\beta)=-1$), the multi-set of eigenvalues $\mcE_{\F_p}([\eta]^{m_\alpha}\mid k_\alpha\oplus k_\alpha)$ cannot contain $1$.
In other words, $|V_{\ddot{\alpha}}^{\eta^{m_\alpha}}|^{\frac{1}{2}}$ is in fact trivial.
Keeping this in mind, we can write the result as
\[
    \Theta_{\omega_{\ddot{\alpha}}}([\eta]^{m_\alpha})
    =
    (-1)^{n_\alpha(\eta)}\cdot\sgn_{k_{\alpha}^\times}(\eta_{-\alpha}C)\cdot |V_{\ddot{\alpha}}^{\eta^{m_\alpha}}|^{\frac{1}{2}}
\]
regardless of whether $\Nr_{k_{\alpha_\res}/k_{\pm\alpha_\res}}(\beta)=-1$ or $1$.

\noindent{\textbf{(II) Ramified case.}} 

We next consider the case where $\alpha_{\res}$ is ramified (note that $f_\alpha$ is odd by Lemma \ref{lem:symm-rest}).
In this case, noting that $\gamma=\varsigma_\alpha(\eta_{-\alpha}C)/(\eta_{-\alpha}C)$, we have
\[
    \Nr_{k_\alpha/k_{\alpha_\res}}(-\gamma)
    =(-1)^{f_\alpha}\cdot \Nr_{k_\alpha/k_{\alpha_\res}}(\gamma)
    =-1.
\]
Hence we get 
\[
    \mcE_{\F_p}([\eta]^{m_\alpha}\mid k_{\alpha}\oplus k_{\alpha})
    =
    \Gal(k_{\alpha_{\res}}/\F_p)\cdot\{-1\}^{\frac{1}{2f_\alpha}}.
\]
This means that the semisimple conjugacy class of the element $[\eta]^{m_\alpha}\in\Sp(k_\alpha\oplus k_\alpha)$ does not depend on the choice of $\eta$.
In other words, $\Theta_{\omega_{\ddot{\alpha}}}([\eta]^{m_\alpha})$ is a constant determined only by $\alpha$.

\begin{rem}
It is possible to explicate $\Theta_{\omega_{\ddot{\alpha}}}([\eta]^{m_\alpha})$ by the same argument as in the unramified case, but we do not insist on it in this paper.
\end{rem}

\begin{rem}\label{rem:asym-sym-triv}   
    Note that asymmetric root $\alpha$ can be restricted to a symmetric root only when $l_\alpha$ is even.
    In particular, when $l$ is odd, $\alpha_\res$ is always asymmetric whenever $\alpha$ is asymmetric.
\end{rem}

\subsection{Twisted characters of Weil representations: symmetric roots}\label{subsec:twisted-Weil-sym}

We next compute $\Theta_{\omega_{\ddot{\alpha}}}([\eta]^{m_{\alpha}})$ in the case where $\alpha\in\Xi_{\sym}$.
Note that $\alpha$ must be unramified by Lemma \ref{lem:ram-empty}.
Recall that $V_{\ddot{\alpha}}=V_{\dot{\alpha}}\cong V_{\alpha}$ and that $V_{\alpha}$ is a $1$-dimensional $k_{\alpha}$-vector spaces, which is identified with $k_{\alpha}$ by fixing nonzero elements $X_{\alpha}\in V_{\alpha}$ (see Section \ref{subsec:Heisen-structure}).

We carry out a case-by-case computation on the ramification of $\alpha_\res$.
Recall that we have fixed $\sigma_\alpha$ satisfying $\theta^{m_\alpha}(\alpha)=\sigma_\alpha(\alpha)$ and also $\tau_\alpha$ satisfying $\tau_\alpha(\alpha)=-\alpha$.
Also recall that $\alpha_\res$ is unramified if and only if $f_\alpha=[k_\alpha:k_{\alpha_\res}]$ is odd (Lemma \ref{lem:symm-rest}).
So we are divided into the following two cases:
\begin{enumerate}
    \item[(I)] $\alpha_\res$ is symmetric unramified ($f_\alpha$ is odd);
    \item[(II)] $\alpha_\res$ is symmetric ramified ($f_\alpha$ is even).
\end{enumerate}

To make the notation lighter, we write $\varsigma_\alpha$ for $\sigma_\alpha^{-1}$ in the following.

\[
\xymatrix{
F_{\alpha_{\res}}\ar^-{l_\alpha/m_\alpha}_-{\langle\sigma_{\alpha}\rangle}@{-}[r]&F_{\alpha}\\
F_{\pm\alpha_{\res}}\ar@{-}[r]\ar^-{2}_-{\langle\tau_{\alpha}\rangle}@{-}[u]&F_{\pm\alpha}\ar^-{2}_-{\langle\tau_{\alpha}\rangle}@{-}[u]
}
\qquad
\xymatrix{
k_{\alpha_{\res}}\ar^-{f_\alpha}@{-}[r]&k_{\alpha}\\
k_{\pm\alpha_{\res}}\ar@{-}[r]\ar^-{\leq2}@{-}[u]&k_{\pm\alpha}\ar^-{2}@{-}[u]
}
\]

Since $[\eta]^{m_\alpha}(X_{\alpha})$ belongs to $V_{\theta^{m_\alpha}(\alpha)}=V_{\sigma_{\alpha}(\alpha)}$, the induced action of $[\eta]^{m_\alpha}$ on $V_{\alpha}$ is $\varsigma_\alpha$-linear (note that $[\eta]^{m_\alpha}$ is $k_{\alpha}$-linear on $V_{\dot{\alpha}}$):
\[
\xymatrix{
V_{\alpha}\ar^-{\cong}[r]& V_{\dot{\alpha}} \ar^-{[\eta]^{m_\alpha}}[d]& X_{\alpha}\ar@{|->}[r]& \sum_{\sigma\in\Gamma/\Gamma_{\alpha}} \sigma(X_{\alpha})\ar@{|->}[d]\\
V_{\alpha}& V_{\dot{\alpha}} \ar^-{\cong}[l]& [\eta]^{m_\alpha}\circ\varsigma_\alpha(X_{\alpha})& \sum_{\sigma\in\Gamma/\Gamma_{\alpha}} [\eta]^{m_\alpha}\circ\sigma(X_{\alpha}) \ar@{|->}[l]
}
\]
More explicitly, if we let $\eta_{\alpha}$ be the element of $k_{\alpha}^{\times}$ such that $[\eta]^{m_\alpha}\circ\varsigma_\alpha(X_{\alpha})=\eta_{\alpha}X_{\alpha}$, then $[\eta]^{m_\alpha}$ is given by 
\[
    x\mapsto \eta_\alpha\varsigma_\alpha(x)
\]
as an element of $\Sp(V_{\ddot{\alpha}})\cong\Sp(k_{\alpha})$.
Note that, since $[\eta]^{m_\alpha}$ preserves the symplectic form as described in Section \ref{subsec:Heisen-structure}, we must have
\begin{align}\label{eq:sym-symp-preserve}
\Tr_{k_{\alpha}/\F_{p}}(C\cdot \eta_{\alpha}\varsigma_{\alpha}(x)\tau_{\alpha}(\eta_{\alpha}\varsigma_{\alpha}(y)))
=\Tr_{k_{\alpha}/\F_{p}}(C\cdot x\tau_{\alpha}(y))
\end{align}
for any $x,y\in k_{\alpha}$.
In other words, we necessarily have that $\eta_{\alpha}\tau_{\alpha}(\eta_{\alpha})=\varsigma_{\alpha}(C)/C$.

To compute $\Theta_{\omega_{\ddot{\alpha}}}([\eta]^{m_\alpha})$, we follow the same strategy as in the case where $\alpha$ is asymmetric and $\alpha_\res$ is symmetric, i.e., appeal to G{\'e}rardin's formula (Proposition \ref{prop:Gerardin2}) by computing the eigenvalues of the action $[\eta]^{m_\alpha}$ on $k_\alpha$.
By Lemma \ref{lem:eigen-s-th-root},
\begin{align*}
    \mcE_{k_{\alpha_\res}}(\eta_{\alpha}\circ\varsigma_{\alpha}\mid k_\alpha)
    =\{\eta_{\alpha}\cdot\varsigma_{\alpha}(\eta_{\alpha})\cdots\varsigma_{\alpha}^{f_\alpha-1}(\eta_{\alpha})\}^{1/f_\alpha}
    =\{\Nr_{k_\alpha/k_{\alpha_\res}}(\eta_\alpha)\}^{1/f_\alpha}.
\end{align*}
Hence, 
\[
    \mcE_{\F_p}(\eta_{\alpha}\circ\varsigma_{\alpha}\mid k_\alpha)
    =\Gal(k_{\alpha_\res}/\F_p)\cdot\{\Nr_{k_\alpha/k_{\alpha_\res}}(\eta_\alpha)\}^{1/f_\alpha}.
\]

\noindent{\textbf{(I) Unramified case.}} 

We first consider the case where $\alpha_{\res}$ is unramified.
\[
\xymatrix{
k_{\alpha_{\res}}\ar^-{\text{$f_\alpha$: odd}}@{-}[r]&k_{\alpha}\\
k_{\pm\alpha_{\res}}\ar@{-}[r]\ar^-{2}@{-}[u]&k_{\pm\alpha}\ar^-{2}@{-}[u]
}
\]

Our computation here is similar to that in the case where $\alpha$ is asymmetric and $\alpha_\res$ is symmetric unramified.
We put $\beta:=\Nr_{k_\alpha/k_{\alpha_\res}}(\eta_\alpha)$.
We define extensions $k_1,\ldots,k_r$ of $k_{\alpha_\res}$ so that
\[
    k_{\alpha_\res}[X]/(X^{f_\alpha}-\beta)
    \cong
    \bigoplus_{i=1}^r k_i
\]
and also define $(x_i)_i$ to be the image of $X$.
Since $\eta_{\alpha}\tau_{\alpha}(\eta_{\alpha})=\varsigma_{\alpha}(C)/C$, we see that $\Nr_{k_\alpha/k_{\alpha_\res}}(\eta_\alpha\tau_\alpha(\eta_\alpha))=1$.
In other words, $\beta$ belongs to $k_{\alpha_\res}^{1}$.
Hence, by Lemma \ref{lem:finite-field} (2), any $f_\alpha$-th root of $\beta$ belongs to $k_\alpha^{1}$.
This implies that each $k_i$ is contained in $k_\alpha$ and also that $x_i$ belongs to $k_\alpha^{1}$.

Then we have 
\begin{align*}
    \Gal(k_{\alpha_\res}/\F_p)\cdot\{\Nr_{k_\alpha/k_{\alpha_\res}}(\eta_\alpha)\}^{1/f_\alpha}
    =
    \bigsqcup_{i=1}^{r}\{\sigma(x_i) \mid \sigma\in\Gal(k_i/\F_p)\}.
\end{align*}
Therefore, if we choose an $\F_p$-rational maximal torus $T$ of $\Sp(V_{\ddot{\alpha}})$ to be $\prod_{i=1}^{r}k_i^1$, then its element $(x_i)_i$ has the same multi-set of eigenvalues as $[\eta]^{m_\alpha}$ (see Lemma \ref{lem:Sp-tori}).

The partition of $P(V_{\ddot{\alpha}},T)$ into $\Gamma_{\F_p}$-orbit is given by 
\[
    \bigsqcup_{i=1}^{r}\{\sigma(-)\mid \sigma\in\Gal(k_i/\F_p)\}
\]
as described in Lemma \ref{lem:Sp-tori} (note that each $\Gamma_{\F_p}$-orbit is symmetric).
Note that $\epsilon([\eta]^{m_\alpha})=1$ for some $\epsilon\in P(V_{\ddot{\alpha}},T)$ if and only if $\beta=1$.
(This is furthermore equivalent to that $V_{\ddot{\alpha}}^{\eta^{m_\alpha}}\neq0$.)
In this case, there exists a unique $\Gamma_{\F_p}$-orbit of such weights, hence the number $l(V_{\ddot{\alpha}},T;[\eta]^{m_\alpha})$ is given by $r-1$.
Otherwise, we have $l(V_{\ddot{\alpha}},T;[\eta]^{m_\alpha})=r$.
Since $\epsilon(-)^{\frac{1+q_\Omega}{2}}$ for each $\Omega=\{\sigma(-)\mid \sigma\in\Gal(k_i/\F_p)\}$ is nothing but the unique quadratic character of $k_i^1$, we have $\chi_T=\prod_{i=1}^{r}\sgn_{k_i^1}$.
In summary, we get
\[
    \Theta_{\omega_{\ddot{\alpha}}}([\eta]^{m_\alpha})
    =
    (-1)^{r-n_\alpha(\eta)}\cdot\prod_{i=1}^{r}\sgn_{k_i^1}(x_i)\cdot |V_{\ddot{\alpha}}^{\eta^{m_\alpha}}|^{\frac{1}{2}},
\]
where $n_\alpha(\eta)=1$ if $V_{\ddot{\alpha}}^{\eta^{m_\alpha}}\neq0$ and $n_\alpha(\eta)=0$ if $V_{\ddot{\alpha}}^{\eta^{m_\alpha}}=0$.

Now, similarly to the case where $\alpha$ is asymmetric, let us examine the particular case where $f_\alpha=1$.
In this case, we have $r=1$, $k_1=k_{\alpha_\res}=k_\alpha$.
Thus $x_1=\beta=\eta_\alpha$, which implies that 
\[
    \Theta_{\omega_{\ddot{\alpha}}}([\eta]^{m_\alpha})
    =
    (-1)^{1-n_\alpha(\eta)}\cdot\sgn_{k_{\alpha}^{1}}(\eta_\alpha)\cdot|V_{\ddot{\alpha}}^{\eta^{m_\alpha}}|^{\frac{1}{2}}.
\]

\noindent{\textbf{(II) Ramified case.}} 

We next consider the case where $\alpha_{\res}$ is ramified.
Recall that $f_\alpha$ must be even in this case (say $f_\alpha=2f_\alpha^\circ$):
\[
\xymatrix{
k_{\alpha_{\res}}\ar^-{f_\alpha=2f_\alpha^\circ}@{-}[r]&k_{\alpha}\\
k_{\pm\alpha_{\res}}\ar^-{f_\alpha^\circ}@{-}[r]\ar@{=}[u]&k_{\pm\alpha}\ar^-{2}@{-}[u]
}
\]

Note that $\tau_\alpha=\sigma_\alpha^{f_\alpha^\circ}=\varsigma_\alpha^{f_\alpha^\circ}$.
Hence, 
\begin{align*}
    \Nr_{k_\alpha/k_{\alpha_\res}}(\eta_\alpha)
    =\prod_{i=0}^{f_\alpha-1}\varsigma_\alpha^i(\eta_\alpha)
    &=\prod_{i=0}^{f_\alpha^\circ-1}\varsigma_\alpha^i(\eta_\alpha\cdot\tau_\alpha(\eta_\alpha))\\
    &=\prod_{i=0}^{f_\alpha^\circ-1}\varsigma_\alpha^i(\varsigma_\alpha(C)/C)
    =\varsigma_\alpha^{f_\alpha^\circ}(C)/C
    =\tau_\alpha(C)/C.
\end{align*}
Recalling that the constant $C$ satisfies $\tau_\alpha(C)=-C$, we get $\Nr_{k_\alpha/k_{\alpha_\res}}(\eta_\alpha)=-1$.
Thus 
\[
    \mcE_{k_{\alpha_\res}}(\eta_{\alpha}\circ\varsigma_{\alpha}\mid k_\alpha)
    =\Gal(k_{\alpha_\res}/\F_p)\cdot\{-1\}^{1/f_\alpha}.
\]
This means that, similarly to the case where $\alpha$ is asymmetric and $\alpha_\res$ is symmetric ramified, the semisimple conjugacy class of the element $[\eta]^{m_\alpha}\in\Sp(k_\alpha)$ does not depend on  $\eta$.
In other words, $\Theta_{\omega_{\ddot{\alpha}}}([\eta]^{m_\alpha})$ is a constant determined only by $\alpha$.

\subsection{Twisted characters of Weil representations: summary}\label{subsec:twisted-Weil-summary}

Let us summarize the results in Sections \ref{subsec:twisted-Weil-asym} and \ref{subsec:twisted-Weil-sym}.

In all cases, $\Theta_{\omega_{\ddot{\alpha}}}([\eta]^{m_\alpha})$ is given by
\[
    \Theta_{\omega_{\ddot{\alpha}}}([\eta]^{m_\alpha})
    =\epsilon_{\ddot{\alpha}}([\eta]^{m_\alpha})\cdot|V_{\ddot{\alpha}}^{\eta^{m_\alpha}}|^{\frac{1}{2}},
\]
where we let $\epsilon_{\ddot{\alpha}}([\eta]^{m_\alpha})$ denote the sign depending on the root $\alpha$ and the element $[\eta]^{m_\alpha}$ as described in Sections \ref{subsec:twisted-Weil-asym} and \ref{subsec:twisted-Weil-sym}.
We did not explicate the sign $\epsilon_{\ddot{\alpha}}([\eta]^{m_\alpha})$ in general, so here let us state the result just in the following way:

\begin{prop}\label{prop:twisted-Weil-product-ambiguous}
We have
\[
    \prod_{\ddot{\alpha}\in\Theta_{\bfS}\backslash\ddot{\Xi}}
    \Theta_{\omega_{\ddot{\alpha}}}([\eta]^{m_{\alpha}})
    =
    |V_{\eta}|^{\frac{1}{2}}
    \cdot
    \prod_{\ddot{\alpha}\in\Theta_{\bfS}\backslash\ddot{\Xi}} \epsilon_{\ddot{\alpha}}([\eta]^{m_\alpha}).
\]
\end{prop}

However, recall that our computation of each sign factor $\epsilon_{\ddot{\alpha}}([\eta]^{m_\alpha})$ becomes quite simple in the case where $f_\alpha=1$.
From a practical point of view (especially, towards an application to the twisted endoscopic character relation), let us rewrite the formula of Proposition \ref{prop:twisted-Weil-product-ambiguous} in this case where $f_\alpha=1$ holds for all $\alpha$. 

We introduce characters $\epsilon_{\alpha}\colon S\rightarrow\C^{\times}$ for $\alpha\in\Phi(\G,\bfS)\smallsetminus\Phi(\G,\bfS)_{\ram}$ as follows:
\[
\epsilon_{\alpha}(s)\colonequals 
\begin{cases}
\sgn_{k_{\alpha}^{\times}}(\overline{\alpha(s)}) & \text{if $\alpha\in\Phi_{\asym}(\G,\bfS)$,}\\
\sgn_{k_{\alpha}^{1}}(\overline{\alpha(s)}) & \text{if $\alpha\in\Phi_{\ur}(\G,\bfS)$.}
\end{cases}
\]\symdef{epsilon-alpha}{$\epsilon_{\alpha}$}

With the notation in Sections \ref{subsec:twisted-Weil-asym} and \ref{subsec:twisted-Weil-sym}, we have
\[
    \epsilon_{\ddot{\alpha}}([\eta]^{m_\alpha})
    =
    \begin{cases}
        \sgn_{\F_p^\times}(-1)^{g_\alpha(f_\alpha-1)}\cdot\sgn_{k_\alpha^\times}(\eta_\alpha)  & \text{$\alpha$: $\asym$ \& $\alpha_\res$: $\asym$,}\\
        (-1)^{n_\alpha(\eta)} \cdot \sgn_{k_\alpha^\times}(\eta_{-\alpha}C) & \text{$\alpha$: $\asym$ \& $\alpha_\res$: sym.\ unram,}\\
        \epsilon_{\ddot{\alpha}} & \text{$\alpha$: $\sym$ \& $\alpha_\res$: sym.\ ram,}\\
        (-1)^{1-n_\alpha(\eta)} \cdot \sgn_{k_\alpha^1}(\eta_{\alpha}) & \text{$\alpha$: $\sym$ \& $\alpha_\res$: sym.\ unram,}\\
        \epsilon_{\ddot{\alpha}} & \text{$\alpha$: $\sym$ \& $\alpha_\res$: sym.\  ram,}
    \end{cases}
\]\symdef{epsilon-alpha-ddot}{$\epsilon_{\ddot{\alpha}}$}
where $\epsilon_{\ddot{\alpha}}$ denotes the sign only determined by $\alpha$ in the case where $\alpha_\res$ is symmetric ramified.

Recall that we have fixed a topologically semisimple element $\ul{\eta}\in \t{S}$ (``a base point'', see Section \ref{subsec:twist-int}).
Let $\eta\in\t{S}$ be any topologically semisimple element; then we may write $\eta=s\ul{\eta}$ with $s\in S$.
In Section \ref{subsec:twisted-Weil-asym}, $\eta_\alpha$ is defined as the element of $k_\alpha^{\times}$ such that $[\eta]^{m_\alpha}\circ\varsigma_\alpha(X_{\alpha})=\eta_{\alpha}X_{\alpha}$.
Since $\ul{\eta}(-)\ul{\eta}^{-1}=\theta_{\bfS}(-)$ on $S$, 
\[
[\eta]^{m_\alpha}
=[s\ul{\eta}]^{m_\alpha}
=\prod_{i=0}^{m_\alpha-1}[\theta_{\bfS}^{i}(s)]\circ[\ul{\eta}]^{m_\alpha}.
\]
Hence
\begin{align*}
    [\eta]^{m_\alpha}\circ\varsigma_\alpha(X_{\alpha})
    &=\prod_{i=0}^{m_\alpha-1}[\theta_{\bfS}^{i}(s)]\circ[\ul{\eta}]^{m_\alpha}\circ\varsigma_\alpha(X_{\alpha})\\
    &=\prod_{i=0}^{m_\alpha-1}[\theta_{\bfS}^{i}(s)](\ul{\eta}_\alpha X_{\alpha})\\
    &=\prod_{i=0}^{m_\alpha-1}\theta^{-i}(\alpha)(s)\cdot\ul{\eta}_\alpha X_{\alpha}
    =\prod_{i=0}^{m_\alpha-1}\theta^{i}(\alpha)(s)\cdot\ul{\eta}_\alpha X_{\alpha}.
\end{align*}
This means that
\[
    \eta_\alpha
    =
    \prod_{i=0}^{m_\alpha-1}\theta^{i}(\alpha)(s)\cdot\ul{\eta}_\alpha.
\]
Therefore, when $\alpha$ is asymmetric and $\alpha_{\res}$ is asymmetric, we obtain
\begin{align*}
    \epsilon_{\ddot{\alpha}}([\eta]^{m_\alpha})
    =\sgn_{\F_{p}^{\times}}(-1)^{g_\alpha(f_\alpha-1)}\cdot\sgn_{k_{\alpha}^{\times}}(\ul{\eta}_{\alpha})\cdot\prod_{i=0}^{m_\alpha-1}\epsilon_{\theta^{i}(\alpha)}(s).
\end{align*}
By the same consideration, when $\alpha$ is asymmetric and $\alpha_{\res}$ is symmetric unramified,
\begin{align*}
    \epsilon_{\ddot{\alpha}}([\eta]^{m_\alpha})
    =(-1)^{n_\alpha(\eta)}\cdot\sgn_{k_{\alpha}^{\times}}(\ul{\eta}_{-\alpha}C)\cdot\prod_{i=0}^{m_\alpha-1}\epsilon_{\theta^{i}(\alpha)}(s).
\end{align*}
Also, when $\alpha$ is symmetric and $\alpha_{\res}$ is symmetric unramified,
\begin{align*}
    \epsilon_{\ddot{\alpha}}([\eta]^{m_\alpha})
    =(-1)^{1-n_\alpha(\eta)}\cdot\sgn_{k_{\alpha}^{1}}(\ul{\eta}_{\alpha})\cdot\prod_{i=0}^{m_\alpha-1}\epsilon_{\theta^{i}(\alpha)}(s).
\end{align*}
We define a constant $\epsilon_{\ddot{\alpha}}$ for $\alpha$ whose $\alpha_\res$ is asymmetric or symmetric ramified by
\[
    \epsilon_{\ddot{\alpha}}
    :=
    \begin{cases}
        \sgn_{\F_{p}^{\times}}(-1)^{g_\alpha(f_\alpha-1)}\cdot\sgn_{k_{\alpha}^{\times}}(\ul{\eta}_{\alpha}) & \text{$\alpha$: asym \& $\alpha_\res$: asym,}\\
        \sgn_{k_{\alpha}^{\times}}(\ul{\eta}_{-\alpha}C) & \text{$\alpha$: asym \& $\alpha_\res$: sym unram,}\\
        -\sgn_{k_{\alpha}^{1}}(\ul{\eta}_{\alpha}) & \text{$\alpha$: sym \& $\alpha_\res$: sym unram.}
    \end{cases}
\]
(Note that this constant depends only on $\alpha$ and the choice of the fixed base point $\ul{\eta}$ and does not depend on $\eta$.)
We define a constant $C_{\ul{\eta}}$ to be the product of $\epsilon_{\ddot{\alpha}}$'s:
\[
    C_{\ul{\eta}}
    :=
    \prod_{\ddot{\alpha}\in\Theta_{\bfS}\backslash\ddot{\Xi}}\epsilon_{\ddot{\alpha}}.
\]\symdef{C-eta}{$C_{\eta}$}
Also, we put 
\[
\t\epsilon_{\Xi}(s)
\colonequals 
\prod_{\begin{subarray}{c} \ddot{\alpha}\in\ddot{\Xi}\\ \alpha_{\res}: \, \asym/\ur\end{subarray}}\epsilon_{\alpha}(s).
\]\symdef{t-epsilon-Xi}{$\tilde{\epsilon}_{\Xi}$}
Then we get the following (recall the description of $V_{\eta}$ in Section \ref{subsec:descent-Yu}):

\begin{prop}\label{prop:twisted-Weil-product}
We have
\[
\prod_{\ddot{\alpha}\in\Theta_{\bfS}\backslash\ddot{\Xi}}
\Theta_{\omega_{\ddot{\alpha}}}([\eta]^{m_{\alpha}})
=
C_{\ul{\eta}}\cdot (-1)^{|\ddot{\Xi}_{\eta,\ur}|}\cdot |V_{\eta}|^{\frac{1}{2}}\cdot\t\epsilon_{\Xi}(s).
\]
\end{prop}

\begin{prop}\label{prop:rho-final}
If we write $\eta=s\ul{\eta}$ with an element $s\in S$, then we have
\[
    \Theta_{\t{\rho}}(\eta)
    =
    |V_{\eta_{0}}|^{\frac{1}{2}}\cdot
    \prod_{\ddot{\alpha}\in\Theta_{\bfS}\backslash\ddot{\Xi}}
    \Theta_{\omega_{\ddot{\alpha}}}([\eta_{0}]^{m_{\alpha}})
    \cdot\vartheta(s).
\]
Furthermore, when $f_\alpha=1$ for all $\alpha\in\Phi(\bfG,\bfS)$, we have    
\[
    \Theta_{\t{\rho}}(\eta)
    =
    C_{\ul{\eta}}\cdot |V_{\eta_{0}}|^{\frac{1}{2}}\cdot (-1)^{|\ddot{\Xi}_{\eta_{0},\ur}|}\cdot\t\epsilon_{\Xi}(s)\cdot\vartheta(s).
\]
\end{prop}

\begin{proof}
By the definition of $\t{\rho}$ and its twisted character (see Sections \ref{subsec:toral-sc} and \ref{subsec:twist-int}), 
\[
\Theta_{\t{\rho}}(\eta)
=\tr\bigl(\rho(s)\circ I_{\rho}^{\ul{\eta}}\bigr)
=\tr\bigl(\omega([s])\circ I_{\omega}^{\ul{\eta}}\bigr)\cdot\vartheta(s).
\]
By Corollary \ref{cor:twisted-HW-Yu} and Proposition \ref{prop:twisted-Weil-product}, we have
\begin{align*}
    \tr\bigl(\omega([s])\circ I_{\omega}^{\ul{\eta}}\bigr)
    =
    |V_{\eta_{0}}|^{\frac{1}{2}}\cdot
    \prod_{\ddot{\alpha}\in\Theta_{\bfS}\backslash\ddot{\Xi}}
    \Theta_{\omega_{\ddot{\alpha}}}([\eta_{0}]^{m_{\ddot{\alpha}}}).
\end{align*}
When the assumption on $f_\alpha$ is satisfied, 
\begin{align*}
    \prod_{\ddot{\alpha}\in\Theta_{\bfS}\backslash\ddot{\Xi}} \Theta_{\omega_{\ddot{\alpha}}}([\eta_0]^{m_{\ddot{\alpha}}})
    =
    C_{\ul{\eta}}\cdot(-1)^{|\ddot{\Xi}_{\eta_{0},\ur}|}\cdot \t\epsilon_{\Xi}(s')
\end{align*}
by Proposition \ref{prop:twisted-Weil-product}, where $s'\in S$ is the element satisfying $\eta_{0}=s'\ul{\eta}$.
Since $\eta_{0}$ commutes with $\eta_{+}$, we have $\eta_{+}s'\ul{\eta}=\eta_+\eta_0=\eta=s\ul{\eta}$, hence $\eta_{+}s'=s$.
This implies that $\epsilon_{\alpha}(s)=\epsilon_{\alpha}(s')$ for any $\alpha\in\ddot{\Xi}$ such that $\alpha_{\res}$ is asymmetric or symmetric unramified.
Thus we get the desired identity.
\end{proof}

\section{Twisted Adler--DeBacker--Spice character formula}\label{sec:TCF-final}

Now we return to the twisted character formula for toral supercuspidal representations.
From now on, $\eta$ again denotes $\delta_{<r}$ for a fixed elliptic regular semisimple element $\delta\in\t{S}$ as in Section \ref{sec:TCF1}.
Recall that $\eta=\eta_{0}\eta_{+}$ is a fixed topological Jordan decomposition.

In Theorem \ref{thm:TCF-pre}, we proved that
\[
\Theta_{\t{\pi}}(\delta)
=
\sum_{\begin{subarray}{c} g\in S\backslash G/G_{\eta} \\ {}^{g}\eta\in\t{S}\end{subarray}}
\Theta_{\t{\rho}}({}^{g}\eta)
\cdot|\t{\mfG}_{\G_{{}^{g}\eta_{0}}}(\vartheta,{}^{g}\eta_{+})|
\cdot\mfG_{\G_{{}^{g}\eta_{0}}}(\vartheta,{}^{g}\eta_{+})
\cdot
\hat{\mu}^{\G_{{}^{g}\eta}}_{X^{\ast}} \bigl(\log({}^{g}\delta_{\geq r})\bigr).
\]
Then we computed the term $\Theta_{\t{\rho}}({}^{g}\eta)$ in Section \ref{sec:HW-twisted}.
If we write ${}^{g}\eta=s_{g}\cdot\ul{\eta}\in\t{S}$ for any $g\in G$ satisfying ${}^{g}\eta\in\t{S}$, then, by Proposition \ref{prop:rho-final}, we have
\[
    \Theta_{\t{\rho}}({}^g\eta)
    =
    |V_{{}^g\eta_{0}}|^{\frac{1}{2}}\cdot
    \prod_{\ddot{\alpha}\in\Theta_{\bfS}\backslash\ddot{\Xi}}
    \Theta_{\omega_{\ddot{\alpha}}}([{}^g\eta_{0}]^{m_{\ddot{\alpha}}})
    \cdot\vartheta(s_g).
\]
When $f_\alpha=1$ for all $\alpha\in\Phi(\bfG,\bfS)$, we furthermore have
\[
    \Theta_{\t{\rho}}({}^{g}\eta)
    =
    C_{\ul{\eta}}\cdot(-1)^{|\ddot{\Xi}_{{}^{g}\eta_{0},\ur}|}\cdot|V_{{}^{g}\eta_{0}}|^{\frac{1}{2}}\cdot\t\epsilon_{\Xi}(s_{g})\cdot\vartheta(s_{g}).
\]
For simplicity, we will state the final form of a twisted character formula only in this case.
Hence let us suppose that $f_\alpha=1$ for all $\alpha\in\Phi(\bfG,\bfS)$ in the following.

\begin{lem}\label{lem:G-volume}
    For any $g\in G$ satisfying ${}^{g}\eta\in\t{S}$, the product $|V_{{}^g\eta_{0}}|^{\frac{1}{2}}\cdot|\t{\mfG}_{\G_{{}^g\eta_{0}}}(\vartheta,{}^g\eta_{+})|$ equals 
    \[
    |\mfg_{{}^g\eta,\x,0:0+}|^{-\frac{1}{2}}\cdot|\mfs^{\nat}_{0:0+}|^{\frac{1}{2}}\cdot|D^{\red}_{G_{{}^g\eta}}(X^{\ast})|^{\frac{1}{2}}\cdot|D^{\red}_{G_{{}^g\eta_{0}}}({}^g\eta_{+})|^{-\frac{1}{2}},
    \]
    where $D^{\red}$ is as in \cite[Definition 2.11]{DS18}.
\end{lem}

\begin{proof}
Let us shortly write $\eta':={}^g\eta$, $\eta'_{0}:={}^g\eta_{0}$, and $\eta'_{+}:={}^g\eta_{+}$.
We utilize \cite[Proposition 5.2.12]{AS09} by choosing $(\G,\G',\phi,\gamma)$ in \textit{loc.\ cit.} to be $(\G_{\eta'_{0}},\bfS^{\nat},\vartheta,\eta'_{+})$.
As we have $[\![\eta'_{+};\x,r(+)]\!]_{S^{\nat}}=S^{\nat}_{0+}$, $C_{S^{\nat}}^{(0+)}(\eta'_{+})=S^{\nat}$, and $C_{G_{\eta'_{0}}}^{(0+)}(\eta'_{+})=G_{\eta'_{0}}$, we get
\begin{multline*}
|(S^{\nat},G_{\eta'_{0}})_{\x,(r,s):(r,s+)}|^{\frac{1}{2}}
\cdot|\t{\mfG}_{\G_{\eta'_{0}}}(\vartheta,\eta'_{+})|\\
=
\bigl[ [\![\eta'_{+};\x,r]\!]_{G_{\eta'_{0}}} : S^{\nat}_{0+}G_{\eta'_{0},\x,s} \bigr]^{\frac{1}{2}}
\cdot
\bigl[ [\![\eta'_{+};\x,r+]\!]_{G_{\eta'_{0}}} : S^{\nat}_{0+}G_{\eta'_{0},\x,s+} \bigr]^{\frac{1}{2}}.
\end{multline*}
By \cite[Corollary 4.13]{DS18}, the right-hand side equals
\[
|\mfg_{\eta',\x,0:0+}|^{-\frac{1}{2}}\cdot|\mfs^{\nat}_{0:0+}|^{\frac{1}{2}}\cdot|D^{\red}_{G_{\eta'}}(X^{\ast})|^{\frac{1}{2}}\cdot|D^{\red}_{G_{\eta'_{0}}}(\eta'_{+})|^{-\frac{1}{2}}.
\]
Since $V_{\eta'_{0}}=(S^{\nat},G_{\eta'_{0}})_{\x,(r,s):(r,s+)}$ (see Section \ref{subsec:descent-Yu}), we get the assertion.
\end{proof}

Recall that the Fourier transform of an orbital integral depends on the choice of Haar measures.
We let $\hat{\mu}^{\G_{\eta}}_{\Wal,X^{\ast}}$ denote the Fourier transform of the orbital integral with respect to $X^{\ast}$ normalized via the canonical measure of Waldspurger (see \cite[Definition 4.6]{DS18}).\symdef{mu-hat-G-eta-Wal-X-ast}{$\hat{\mu}^{\G_{\eta}}_{\Wal,X^{\ast}}$}
Then, by (the proof of) \cite[Proposition 4.26]{DS18}, 
\[
\hat{\mu}^{\G_{\eta}}_{\Wal,X^{\ast}}
=
|(\mfs^{\nat},\mfg_{\eta})_{\x,(0,0):(0,0+)}|^{-\frac{1}{2}}\cdot
\hat{\mu}^{\G_{\eta}}_{X^{\ast}}
=
|\mfg_{\eta,\x,0:0+}|^{-\frac{1}{2}}\cdot|\mfs^{\nat}_{0:0+}|^{\frac{1}{2}}\cdot
\hat{\mu}^{\G_{\eta}}_{X^{\ast}}.
\]
Following \cite[Section 4.2]{Kal19}, we put
\[
\hat{\iota}^{\G_{\eta}}_{X^{\ast}}(-)
\colonequals 
|D^{\red}_{G_{\eta}}(X^{\ast})|^{\frac{1}{2}}\cdot|D^{\red}_{G_{\eta}}(-)|^{\frac{1}{2}}\cdot\hat{\mu}^{\G_{\eta}}_{\Wal,X^{\ast}}(-).
\]\symdef{hat-iota-G-eta-X-star}{$\hat{\iota}^{\G_{\eta}}_{X^{\ast}}$}
(The same normalization is applied to ${}^g\eta$ for any $g\in G$ satisfying ${}^g\eta\in\t{S}$.)

Recall from \cite[Section 4.5]{KS99} that the \textit{fourth absolute transfer factor} is defined by 
\[
\Delta_{\IV}^{\t\G}(\delta')
\colonequals |\det([\delta']-1 \mid \mfg/\mfs')|_{\ol{F}}^{\frac{1}{2}}.
\]
Here, $\delta'$ is any regular semisimple element of $\t{G}$, which is supposed to belong to an $F$-rational twisted maximal torus $\t{\bfS}'$.
We choose an isomorphism $[g_{\bfS'}]\colon(\bfS',\t{\bfS}')\xrightarrow{\cong}(\bfT,\t{\bfT})$ as in Section \ref{subsec:Steinberg} to transport $\delta'$ to an element $\nu\theta'$ of $\t{\bfT}=\bfT\theta$.
Since we are assuming that the set $\Phi_{\res}(\G,\bfT)$ contains no restricted root of type $2$ or $3$, \cite[Section 4.5]{KS99} gives
\[
\Delta_{\IV}^{\t\G}(\delta')
=
\prod_{\alpha\in\Phi_{\res}(\G,\T)}|N(\alpha)(\nu')-1|^{\frac{1}{2}}_{\ol{F}}.
\]
By noting this, we define $\Delta_{\IV}^{\t\G}(\delta')$ also for any semisimple element $\delta'\in\t{S}$ by
\[
\Delta_{\IV}^{\t\G}(\delta')
=
\prod_{\begin{subarray}{c}\alpha\in \Phi_{\res}(\G,\T) \\ N(\alpha)(\nu')\neq1\end{subarray}}|N(\alpha)(\nu')-1|^{\frac{1}{2}}_{\ol{F}},
\]
where $\nu'\in\T$ is the element such that $[g_{\bfS'}](\delta')=\nu'\theta$.

\begin{lem}\label{lem:tran-IV-descent}
We have 
\begin{align*}
\Delta_{\IV}^{\t{\G}}(\delta)
&=
\Delta_{\IV}^{\t\G}(\delta_{<r})\cdot\Delta_{\IV}^{\G_{\delta_{<r}}}(\delta_{\geq r})
=
|D^{\red}_{G_{\eta_{0}}}(\eta_{+})|^{\frac{1}{2}}\cdot|D^{\red}_{G_{\eta}}(\log(\delta_{\geq r}))|^{\frac{1}{2}}.
\end{align*}
\end{lem}

\begin{proof}
Let $\t\bfS'$ be the $F$-rational twisted maximal torus of $\t\G$ containing $\delta$.
Let $\nu\theta$, $\nu_{<r}\theta$, and $\nu_{\geq r}\theta$ be the images of $\delta$, $\delta_{<r}$, and $\delta_{\geq r}$ under the isomorphism $[g_{\bfS'}]$, respectively.
Then we have $\nu=\nu_{<r}\nu_{\geq r}$.
By noting that the valuation of $N(\alpha)(\nu_{<r})-1$ is smaller than $r$ whenever $N(\alpha)(\nu_{<r})\neq1$, we have
\[
|N(\alpha)(\nu)-1|^{\frac{1}{2}}_{\ol{F}}
=
\begin{cases}
|N(\alpha)(\nu_{<r})-1|^{\frac{1}{2}}_{\ol{F}} & \text{if $N(\alpha)(\nu_{<r})\neq1$,}\\
|N(\alpha)(\nu_{\geq r})-1|^{\frac{1}{2}}_{\ol{F}} & \text{if $N(\alpha)(\nu_{<r})=1$.}
\end{cases}
\]
Since $\{\alpha\in \Phi_{\res}(\G,\T) \mid N(\alpha)(\nu_{<r})=1\}$ is identified with the set $\Phi(\G_{\nu_{<r}\theta},\T^{\nat})$ (see Section \ref{subsec:Steinberg}) and 
\[
|N(\alpha)(\nu_{\geq r})-1|^{\frac{1}{2}}_{\ol{F}}
=|\alpha(\nu_{\geq r})^{l_{\alpha}}-1|^{\frac{1}{2}}_{\ol{F}}
=|\alpha(\nu_{\geq r})-1|^{\frac{1}{2}}_{\ol{F}}
\]
(use that $l_{\alpha}$ is prime to $p$), we get $\Delta_{\IV}^{\t\G}(\delta)=\Delta_{\IV}^{\t\G}(\delta_{<r})\cdot\Delta_{\IV}^{\G_{\delta_{<r}}}(\delta_{\geq r})$.

By applying the same argument to $\Delta_{\IV}^{\t\G}(\delta_{<r})$, we also have a decomposition $\Delta_{\IV}^{\t\G}(\delta_{<r})=\Delta_{\IV}^{\t\G}(\eta_{0})\cdot \Delta_{\IV}^{\G_{\eta_{0}}}(\eta_{+})$.
However, we have $|N(\alpha)(\nu_{0})-1|_{\ol{F}}=1$ whenever $N(\alpha)(\nu_{0})\neq1$ since $N(\alpha)(\nu_{0})$ is of prime-to-$p$ order.
Hence we get $\Delta_{\IV}^{\t\G}(\delta)=\Delta_{\IV}^{\G_{\eta_{0}}}(\eta_{+})\cdot\Delta_{\IV}^{\G_{\delta_{<r}}}(\delta_{\geq r})$.
This can be rewritten as $\Delta_{\IV}^{\t{\G}}(\delta)=|D^{\red}_{G_{\eta_{0}}}(\eta_{+})|^{\frac{1}{2}}\cdot|D^{\red}_{G_{\eta}}(\delta_{\geq r})|^{\frac{1}{2}}$ by \cite[Remark 2.12]{DS18}.
By also noting that $|D^{\red}_{G_{\eta}}(\delta_{\geq r})|=|D^{\red}_{G_{\eta}}(\log(\delta_{\geq r}))|$, we obtain the assertion.
\end{proof}

We define the \textit{normalized Harish-Chandra character} $\Phi_{\t{\pi}}$ by
\[
\Phi_{\t{\pi}}(\delta)\colonequals \Delta_{\IV}^{\t{\G}}(\delta)\cdot\Theta_{\t{\pi}}(\delta).
\]\symdef{Phi-pi-tilde}{$\Phi_{\tilde{\pi}}$}

\begin{thm}\label{thm:TCF-final}
    We write ${}^{g}\eta=s_{g}\cdot\ul{\eta}\in\t{S}$ for each $g\in G$ satisfying ${}^{g}\eta\in\t{S}$.
    Then
    \[
    \Phi_{\t{\pi}}(\delta)
    =
    C_{\ul{\eta}}\cdot(-1)^{|\ddot{\Xi}_{\eta_{0},\ur}|}\cdot 
    \!\!\!\!\!
    \sum_{\begin{subarray}{c} g\in S\backslash G/G_{\eta} \\ {}^{g}\eta\in\t{S}\end{subarray}}
    \t\epsilon_{\Xi}(s_{g})\cdot\vartheta(s_{g})
    \cdot\mfG_{\G_{{}^{g}\eta_{0}}}(\vartheta,{}^{g}\eta_{+})
    \cdot\hat{\iota}^{\G_{{}^{g}\eta}}_{X^{\ast}}(\log({}^{g}\delta_{\geq r})).
    \]
\end{thm}

\begin{proof}
    By the discussion so far (i.e., by Theorem \ref{thm:TCF-pre} and Proposition \ref{prop:rho-final}), the unnormalized character $\Theta_{\t{\pi}}(\delta)$ is equal to
    \[
    \sum_{\begin{subarray}{c} g\in S\backslash G/G_{\eta} \\ {}^{g}\eta\in\t{S}\end{subarray}}
    \Theta_{\t{\rho}}({}^{g}\eta)
    \cdot|\t{\mfG}_{\G_{{}^{g}\eta_{0}}}(\vartheta,{}^{g}\eta_{+})|
    \cdot\mfG_{\G_{{}^{g}\eta_{0}}}(\vartheta,{}^{g}\eta_{+})
    \cdot
    \hat{\mu}^{\G_{{}^{g}\eta}}_{X^{\ast}} (\log({}^{g}\delta_{\geq r})),
    \]
    where $\Theta_{\t{\rho}}({}^{g}\eta)$ is given by $C_{\ul{\eta}}\cdot(-1)^{|\ddot{\Xi}_{{}^{g}\eta_{0},\ur}|}\cdot|V_{{}^{g}\eta_{0}}|^{\frac{1}{2}}\cdot\t\epsilon_{\Xi}(s_{g})\cdot\vartheta(s_{g})$.
    By Lemma \ref{lem:G-volume}, each summand equals
    \begin{multline*}
    |\mfg_{{}^g\eta,\x,0:0+}|^{-\frac{1}{2}}\cdot|\mfs^{\nat}_{0:0+}|^{\frac{1}{2}}\cdot|D^{\red}_{G_{{}^g\eta}}(X^{\ast})|^{\frac{1}{2}}\cdot|D^{\red}_{G_{{}^g\eta_{0}}}({}^g\eta_{+})|^{-\frac{1}{2}}\\
    \cdot
    C_{\ul{\eta}}\cdot(-1)^{|\ddot{\Xi}_{{}^{g}\eta_{0},\ur}|}\cdot\t\epsilon_{\Xi}(s_{g})\cdot\vartheta(s_{g})
    \cdot\mfG_{\G_{{}^{g}\eta_{0}}}(\vartheta,{}^{g}\eta_{+})
    \cdot
    \hat{\mu}^{\G_{{}^{g}\eta}}_{X^{\ast}} (\log({}^{g}\delta_{\geq r})).
    \end{multline*}
    By the above discussion on the normalization of the Fourier transform of the orbital integral, this equals
    \begin{multline*}
    |D^{\red}_{G_{{}^g\eta}}(\log({}^g\delta_{\geq r}))|^{-\frac{1}{2}}\cdot|D^{\red}_{G_{{}^g\eta_{0}}}({}^g\eta_{+})|^{-\frac{1}{2}}\\
    \cdot
    C_{\ul{\eta}}\cdot(-1)^{|\ddot{\Xi}_{{}^{g}\eta_{0},\ur}|}\cdot\t\epsilon_{\Xi}(s_{g})\cdot\vartheta(s_{g})
    \cdot\mfG_{\G_{{}^{g}\eta_{0}}}(\vartheta,{}^{g}\eta_{+})
    \cdot
    \hat{\iota}^{\G_{{}^{g}\eta}}_{X^{\ast}} (\log({}^{g}\delta_{\geq r})).
    \end{multline*}
    Lemma \ref{lem:tran-IV-descent} implies that the product of the first two factors equals $\Delta_{\IV}^{\t{\G}}({}^{g}\delta)^{-1}=\Delta_{\IV}^{\t{\G}}(\delta)^{-1}$.
    By finally noting that $|\ddot{\Xi}_{{}^{g}\eta_{0},\ur}|=|\ddot{\Xi}_{\eta_{0},\ur}|$ for any $g\in G$ satisfying ${}^{g}\eta\in\t{S}$, we get the desired formula.
\end{proof}

\begin{rem}
    We do not try to compute the term $\mfG_{\G_{{}^{g}\eta_{0}}}(\vartheta,{}^{g}\eta_{+})$ in this paper.
    The point is that this invariant is being considered with respect to the `descended' (hence no longer twisted) group $\G_{{}^{g}\eta_{0}}$, so we can apply the computation in \cite{DS18} to explicitly determine it.
    In fact, the resulting formula can be also interpreted in terms of certain endoscopic invariant such as the (second) absolute transfer factors, Kottwitz signs, and so on (\cite[Section 4.7]{Kal19}).
    This observation played an important role in Kaletha's work \cite{Kal19} on the standard endoscopic character relation for toral supercuspidal representations.
    We investigate analogous computation in the twisted setting in our subsequent paper \cite{Oi23-TECR}.
\end{rem}

\appendix
\section{Some facts on Heisenberg--Weil representations}\label{sec:twisted-HW}

\subsection{Decomposition formula of twisted characters}\label{subsec:twisted-HW}
Let us consider an abstract situation where the following data are given:
\begin{itemize}
\item
a finite-dimensional symplectic space $V$ over $\F_{p}$, where $p\neq2$, 
\item
mutually orthogonal finite-dimensional symplectic subspaces $V^{i}_{j}$ where $i=0,\ldots,r$ and $j=0,\ldots,l_{i}$ satisfying
\[
V=\bigoplus_{i=0}^{r}\bigoplus_{j=0}^{l_{i}}V^{i}_{j},
\]
\item
a symplectic automorphism $\iota$ of $V$ such that $\iota\colon V^{i}_{j}\xrightarrow{\sim} V^{i}_{j+1}$ for any $0\leq i\leq r$ and $0\leq j \leq l_{i}$ (here we put $V^{i}_{l_{i}+1}\colonequals V^{i}_{0}$ for convenience),
\item
a nontrivial character $\vartheta$ of $\F_{p}$.
\end{itemize}

We write $\bbH(V^{i}_{j})$ for the finite Heisenberg group associated to the symplectic space $V^{i}_{j}$ over $\F_{p}$.
More precisely, $\bbH(V^{i}_{j})$ is defined to be the set $V^{i}_{j}\times\F_{p}$ equipped with a multiplication law given by
\[
(v_{1},z_{1})\cdot(v_{2},z_{2})\colonequals \bigl(v_{1}+v_{2},z_{1}+z_{2}+\frac{1}{2}\langle v_{1},v_{2}\rangle\bigr),
\]
where $\langle-,-\rangle$ denotes the symplectic form on $V^{i}_{j}$.
Then, according to the Stone--von Neumann theorem, we have an irreducible representation $\omega^{i}_{j}$ of $\Sp(V^{i}_{j})\ltimes\bbH(V^{i}_{j})$ with central character $\vartheta$ (called a \textit{Heisenberg--Weil representation}), which is unique up to isomorphism unless $\Sp(V^{i}_{j})\cong\SL_{2}(\F_{3})$.
We let $W^{i}_{j}$ denote the representation space of $\omega^{i}_{j}$:
\[
\omega^{i}_{j}\colon \Sp(V^{i}_{j})\ltimes\bbH(V^{i}_{j})\rightarrow\GL_{\C}(W^{i}_{j}).
\]

Let $\bbH(V)$ denote the Heisenberg group associated to $V$.
Then note that $\bbH(V)$ is isomorphic to the central product of $\bbH(V^{i}_{j})$ for $0\leq i \leq r$ and $0\leq j \leq l_{i}$, i.e., the quotient of the product group $\prod_{i,j}\bbH(V^{i}_{j})$ by the central subgroup 
\[
\Bigl\{(0,z^{i}_{j})_{i,j}\in\prod_{i,j}\bbH(V^{i}_{j})
\,\Big\vert\,
\sum_{i,j}z^{i}_{j}=0 \Bigr\}.
\]
If we put $W\colonequals \bigotimes_{i,j}W^{i}_{j}$, then $W$ realizes the Heisenberg--Weil representation of $\Sp(V)\ltimes\bbH(V)$ with central character $\vartheta$, for which we write $\omega$.
Furthermore, on the subgroup
\[
\Bigl(\prod_{i,j}\Sp(V^{i}_{j})\Bigr)\ltimes\bbH(V)\subset\Sp(V)\ltimes\bbH(V),
\]
the representation $\omega$ is isomorphic to $\bigotimes_{i,j}\omega^{i}_{j}$ (see \cite[2.5]{Ger77} for the details).

Since $\iota$ is a symplectic automorphism of $V$, an isomorphism
\[
\iota_{\ast}\colon\Sp(V)\ltimes\bbH(V)\xrightarrow{\sim}\Sp(V)\ltimes\bbH(V);\quad
\bigl(g,(v,z)\bigr)\mapsto \bigl({}^{\iota}g,(\iota(v),z)\bigr)
\]
is naturally induced, where ${}^{\iota}g\colonequals \iota\circ g\circ\iota^{-1}$.
Then, since $\iota$ acts on the center part of $\bbH(V)$ identically, the $\iota_{\ast}$-twist of the representation $\omega$ (let us write $\omega^{\iota}$) is again the Heisenberg--Weil representation of $\Sp(V)\ltimes\bbH(V)$ with central character $\vartheta$.
Hence, by the uniqueness part of the Stone--von Neumann theorem, $\omega$ and $\omega^{\iota}$ are isomorphic as representations of $\Sp(V)\ltimes\bbH(V)$.
Our aim in this section is to construct an intertwiner $\omega\xrightarrow{\sim}\omega^{\iota}$ explicitly by using the symplectic decomposition $V=\bigoplus_{i=0}^{r}\bigoplus_{j=0}^{l_{i}}V^{i}_{j}$ and express the associated twisted character in terms of the intertwiner.

Since the symplectic isomorphism $\iota$ maps $V^{i}_{j}$ to $V^{i}_{j+1}$, the automorphism $\iota_{\ast}$ of $\Sp(V)\ltimes\bbH(V)$ induces
\[
\iota_{\ast}\colon
\Sp(V^{i}_{j})\ltimes\bbH(V^{i}_{j})\xrightarrow{\sim}\Sp(V^{i}_{j+1})\ltimes\bbH(V^{i}_{j+1});\quad
\bigl(g,(v,z)\bigr)\mapsto \bigl({}^{\iota}g,(\iota(v),z)\bigr).
\]
Therefore the subgroup
\[
\Bigl(\prod_{i,j}\Sp(V^{i}_{j})\Bigr)\ltimes\bbH(V)\subset\Sp(V)\ltimes\bbH(V)
\]
is preserved under $\iota_{\ast}$.
As $\omega$ and $\omega^{\iota}$ are irreducible as representations of the subgroup $(\prod_{i,j}\Sp(V^{i}_{j}))\ltimes\bbH(V)$ (or even $\bbH(V)$), any intertwiner between $\omega$ and $\omega^{\iota}$ as representations of $(\prod_{i,j}\Sp(V^{i}_{j}))\ltimes\bbH(V)$ is automatically an intertwiner as representations of $\Sp(V)\ltimes\bbH(V)$.

Let us consider the representation $\omega^{i,\iota}_{j+1}$ given by the pull-back of $\omega^{i}_{j+1}$ via $\iota_{\ast}$:
\[
\omega^{i,\iota}_{j+1}\colon \Sp(V^{i}_{j})\ltimes\bbH(V^{i}_{j})\xrightarrow{\iota_{\ast}}\Sp(V^{i}_{j+1})\ltimes\bbH(V^{i}_{j+1})\rightarrow\GL_{\C}(W^{i}_{j+1}).
\]
Here, similarly to the notation $V^{i}_{l_{i}+1}\colonequals V^{i}_{0}$, we put $\omega^{i}_{l_{i}+1}\colonequals \omega^{i}_{0}$ for convenience.
Then, since $\iota$ preserves the center part of $\bbH(V^{i}_{j})$ identically, $\omega^{i,\iota}_{j+1}$ is the Heisenberg--Weil representation of $\Sp(V^{i}_{j})\ltimes\bbH(V^{i}_{j})$ with central character $\vartheta$.
In particular, by the uniqueness part of the Stone--von Neumann theorem, $\omega^{i,\iota}_{j+1}$ is isomorphic to $\omega^{i}_{j}$ as a representation of $\Sp(V^{i}_{j})\ltimes\bbH(V^{i}_{j})$.
Let us fix an intertwiner $I^{i}_{j}$ between these two representations, i.e., an isomorphism
\[
I^{i}_{j}\colon (\omega^{i}_{j},W^{i}_{j})\xrightarrow{\sim}(\omega^{i,\iota}_{j+1},W^{i}_{j+1})
\]
making the following diagram commutative for any $(g,h)\in\Sp(V^{i}_{j})\ltimes\bbH(V^{i}_{j})$:
\[
\xymatrix{
W^{i}_{j}\ar^-{I^{i}_{j}}[r]\ar_-{\omega^{i}_{j}(g,h)}[d]&W^{i}_{j+1}\ar^-{\omega^{i,\iota}_{j+1}(g,h)=\omega^{i}_{j+1}(\iota_{\ast}(g,h))}[d]\\
W^{i}_{j}\ar_-{I^{i}_{j}}[r]&W^{i}_{j+1}
}
\]

If we put $V^{i}\colonequals \bigoplus_{j=0}^{l_{i}}V^{i}_{j}$, then $\bbH(V^{i})$ is isomorphic to the central product of $\bbH(V^{i}_{j})$ for $0\leq j \leq l_{i}$.
The automorphism $\iota_{\ast}$ of $\Sp(V)\ltimes\bbH(V)$ preserves the subgroup 
\[
\Bigl(\prod_{j=0}^{l_{i}}\Sp(V^{i}_{j})\Bigr)\ltimes\bbH(V^{i})
\]
and its action is described on this subgroup by
\[
\iota_{\ast}\colon
\bigl((g_{0}\ldots,g_{l_{i}}),(v,z)\bigl)
\mapsto
\bigl(({}^{\iota}g_{l_{i}},{}^{\iota}g_{0},\ldots,{}^{\iota}g_{l_{i}-1}), (\iota(v),z)\bigr).
\]
We put $W^{i}\colonequals \bigotimes_{j=0}^{l_{i}}W^{i}_{j}$ and define an $\C$-linear automorphism $I^{i}$ on $W^{i}$ by
\[
I^{i}\colon
v_{0}\otimes\cdots\otimes v_{l_{i}-1}\otimes v_{l_{i}}\mapsto I^{i}_{l_{i}}(v_{l_{i}})\otimes I^{i}_{0}(v_{0})\otimes\cdots\otimes I^{i}_{l_{i}-1}(v_{l_{i}-1}).
\]
We write $(\omega^{i},W^{i})$ for the tensored Heisenberg--Weil representation $(\bigotimes_{j}\omega^{i}_{j},\bigotimes_{j}W^{i}_{j})$ of $\Sp(V^{i})\ltimes\bbH(V^{i})$.
Then we can easily check that $I^{i}$ gives an intertwiner between $(\omega^{i},W^{i})$ and its $\iota_{\ast}$-twist $(\omega^{i,\iota},W^{i})$ as representation of $(\prod_{j=0}^{l_{i}}\Sp(V^{i}_{j}))\ltimes\bbH(V^{i})$, that is, the following diagram is commutative for any $((g_{0},\ldots,g_{l_{i}}),h)\in(\prod_{j=0}^{l_{i}}\Sp(V^{i}_{j}))\ltimes\bbH(V^{i}):$
\[
\xymatrix{
W^{i}\ar^-{I^{i}}[r]\ar_-{\omega^{i}((g_{0},\ldots,g_{l_{i}}),h)}[d]&W^{i}\ar^-{\omega^{i,\iota}((g_{0},\ldots,g_{l_{i}}),h)=\omega(\iota_{\ast}((g_{0},\ldots,g_{l_{i}}),h))}[d]\\
W^{i}\ar_-{I^{i}}[r]&W^{i}
}
\]

Now we define a $\C$-linear isomorphism $I$ of $W=\bigotimes_{i=0}^{r}W^{i}$ by $I\colonequals \bigotimes_{i=0}^{r}I^{i}$, i.e., 
\[
I\colon 
v^{0}\otimes\cdots\otimes v^{r}\mapsto I^{0}(v^{0})\otimes\cdots\otimes I^{r}(v^{r}).
\]
Then $I$ is an intertwiner between $\omega\,(=\bigotimes_{i}\omega^{i})$ and its $\iota_{\ast}$-twist $\omega^{\iota}\,(=\bigotimes_{i}\omega^{i,\iota})$.

\begin{lem}\label{lem:tr}
Let $W'_{0},\ldots,W'_{l}$ be finite-dimensional $\C$-vector spaces equipped with $\C$-linear isomorphisms $I'_{j}\colon W'_{j}\xrightarrow{\sim}W'_{j+1}$ for $1\leq j\leq l$, where we put $W'_{l+1}\colonequals W'_{0}$.
We define an automorphism $I'$ of $W'\colonequals W'_{0}\otimes\cdots\otimes W'_{l}$ by 
\[
I'\colon
v_{0}\otimes\cdots\otimes v_{l}\mapsto I'_{l}(v_{l})\otimes I'_{0}(v_{0})\otimes\cdots\otimes I'_{l-1}(v_{l-1}).
\]
Then we have
\[
\tr(I'\mid W')
=
\tr(I'_{l}\circ\cdots\circ I'_{0} \mid W'_{0}).
\]
\end{lem}

\begin{proof}
We take a $\C$-basis $\{e^{(0)}_{1},\ldots,e^{(0)}_{n}\}$ of $W'_{0}$ and define a $\C$-basis $\{e^{(i)}_{1},\ldots,e^{(i)}_{n}\}$ of each $W'_{i}$ ($1\leq i \leq l$) by $e_{j}^{(i)}\colonequals I'_{i-1}\circ\cdots\circ I'_{0}(e_{j}^{(0)})$.
Then, by the definition of the trace, we have
\[
\tr(I'\mid W')
=
\sum_{1\leq j_{0},\ldots,j_{l}\leq n} \langle I'(e^{(0)}_{j_{0}}\otimes\cdots\otimes e^{(l)}_{j_{l}}), e^{(0)}_{j_{0}}\otimes\cdots\otimes e^{(l)}_{j_{l}}\rangle_{W'},
\]
where $\langle-,-\rangle_{W'}$ denotes the standard $\C$-bilinear pairing on $W'\times W'$ given by
\[
\langle e^{(0)}_{j_{0}}\otimes\cdots\otimes e^{(l)}_{j_{l}}, e^{(0)}_{j'_{0}}\otimes\cdots\otimes e^{(l)}_{j'_{l}}\rangle_{W'}
=
\delta_{j_{0},j'_{0}}\cdots\delta_{j_{l},j'_{l}}
\]
for any $1\leq j_{0},\ldots,j_{l}\leq n$ and $1\leq j'_{0},\ldots,j'_{l}\leq n$, where $\delta_{-,-}$ denotes the Kronecker delta.
By the definition of $I'$, we have
\[
I'(e^{(0)}_{j_{0}}\otimes\cdots\otimes e^{(l)}_{j_{l}})
=
(I_{l}'\circ\cdots\circ I_{0}' (e^{0}_{j_{l}}))\otimes e^{(1)}_{j_{0}}\otimes\cdots\otimes e^{(l)}_{j_{l-1}}
\]
Hence the summand of the above formula for the trace of $I'$ is not zero only when $j_{0}=\cdots=j_{l}$.
Moreover, in this case (let us put $j\colonequals j_{0}=\cdots=j_{l}$), we have
\[
\langle I'(e^{(0)}_{j_{0}}\otimes\cdots\otimes e^{(l)}_{j_{l}}), e^{(0)}_{j_{0}}\otimes\cdots\otimes e^{(l)}_{j_{l}}\rangle_{W'}
=
\langle I'_{l}\circ\cdots\circ I'_{0}(e^{(0)}_{j}), e^{(0)}_{j}\rangle_{W'_{0}},
\]
where $\langle-,-\rangle_{W'_{0}}$ denotes the standard $\C$-bilinear pairing on $W'_{0}\times W'_{0}$ satisfying $\langle e^{(0)}_{j}, e^{(0)}_{j'}\rangle_{W'_{0}}=\delta_{j,j'}$ for any $1\leq j\leq n$ and $1\leq j'\leq n$.
Thus we get 
\[
\tr(I'\mid W')
=
\sum_{j=1}^{n} \langle I'_{l}\circ\cdots\circ I'_{0}(e^{(0)}_{j}), e^{(0)}_{j}\rangle_{W'_{0}}.
\]
The right-hand side is nothing but the trace of $I'_{0}\circ\cdots\circ I'_{l}$ on $W'_{0}$.
\end{proof}

\begin{prop}\label{prop:twisted-HW}
Let $g\colonequals (g^{i}_{j})_{i,j}\in\prod_{i,j}\Sp(V^{i}_{j})$.
Then the trace of $\omega(g)\circ I$ on $W$ is given by
\[
\prod_{i=0}^{r}\tr\Bigl(\omega^{i}_{0}\bigl(g^{i}_{0}\circ\iota_{\ast}(g_{l_{i}}^{i})\circ\cdots\circ\iota_{\ast}^{l_{i}}(g_{1}^{i})\bigr)\circ I^{i}_{l_{i}}\circ\cdots\circ I^{i}_{0} \,\Big\vert\, W^{i}_{0} \Bigr)
\]
\end{prop}

\begin{proof}
We put $g^{i}\colonequals (g^{i}_{j})_{j}$.
Recall that $W=\bigotimes_{i=0}^{r}W^{i}$, $\omega(g)=\bigotimes_{i=0}^{r}\omega^{i}(g^{i})$, and $I=\bigotimes_{i=0}^{r}I^{i}$.
Hence we have
\[
\tr(\omega(g)\circ I\mid W)
=
\prod_{i=0}^{r}\tr(\omega^{i}(g^{i})\circ I^{i}\mid W^{i}).
\]

Let us compute each $\tr(\omega^{i}(g^{i})\circ I^{i}\mid W^{i})$.
Recall that $W^{i}=\bigotimes_{j=0}^{l_{i}}W^{i}_{j}$, $\omega^{i}(g^{i})=\bigotimes_{j=0}^{l_{i}}\omega^{i}_{j}(g^{i}_{j})$, and an automorphism $I^{i}$ of $W^{i}$ is defined by
\[
I^{i}\colon
v_{0}\otimes\cdots\otimes v_{l_{i}-1}\otimes v_{l_{i}}\mapsto I^{i}_{l_{i}}(v_{l_{i}})\otimes I^{i}_{0}(v_{0})\otimes\cdots\otimes I^{i}_{l_{i}-1}(v_{l_{i}-1}).
\]
Hence the automorphism $\omega^{i}(g^{i})\circ I^{i}$ of $W^{i}$ is given by
\[
v_{0}\otimes\cdots\otimes v_{l_{i}-1}\otimes v_{l_{i}}\mapsto I^{i,\prime}_{l_{i}}(v_{l_{i}})\otimes I^{i,\prime}_{0}(v_{0})\otimes\cdots\otimes I^{i,\prime}_{l_{i}-1}(v_{l_{i}-1}),
\]
where we put
\[
I^{i,\prime}_{j}\colonequals \omega^{i}_{j+1}(g^{i}_{j+1})\circ I^{i}_{j} \colon W^{i}_{j}\xrightarrow{\sim} W^{i}_{j+1}.
\]
Thus, by Lemma \ref{lem:tr}, we get
\[
\tr(\omega^{i}(g^{i})\circ I^{i} \mid W^{i})
=
\tr(I^{i,\prime}_{l_{i}}\circ\cdots\circ I^{i,\prime}_{0} \mid W^{i}_{0}).
\]

Then, by the intertwining property of $I^{i}_{j}$, i.e., $\omega^{i}_{j+1}(\iota_{\ast}(-))\circ I^{i}_{j}=I^{i}_{j}\circ\omega^{i}_{j}(-)$,
\begin{align*}
I^{i,\prime}_{l_{i}}\circ\cdots\circ I^{i,\prime}_{0}
&= \bigl(\omega^{i}_{l_{i}+1}(g^{i}_{l_{i}+1})\circ I^{i}_{l_{i}}\bigr)\circ\cdots\circ\bigl(\omega^{i}_{1}(g^{i}_{1})\circ I^{i}_{0}\bigr)\\
&=\omega^{i}_{0}\bigl(g^{i}_{0}\circ\iota_{\ast}(g_{l_{i}}^{i})\circ\cdots\circ\iota_{\ast}^{l_{i}}(g_{1}^{i})\bigr)\circ I^{i}_{l_{i}}\circ\cdots\circ I^{i}_{0}.
\end{align*}
Hence we get the assertion.
\end{proof}

\subsection{G\'erardin's character formulas of Weil representations}\label{subsec:Gerardin}
In this subsection, we review character formulas of Weil representations established by G\'erardin \cite{Ger77}.
We let $V$ be a finite-dimensional vector space over $\F_{p}$ equipped with a symplectic pairing $\langle-,-\rangle\colon V\times V\rightarrow\F_{p}$.
Let $(\omega_{V},W_{V})$ be the Heisenberg--Weil representation of $\Sp(V)\ltimes\bbH(V)$ with central character $\vartheta\colon\F_{p}\hookrightarrow\C^{\times}$.

We introduce some notation following \cite[Section 4]{Ger77}.
Suppose that $T$ is an $\F_{p}$-rational maximal torus of $\Sp(V)$.
(The symbol $\Sp(V)$ loosely denotes both finite symplectic group and also an algebraic group over $\F_p$.)
Then, since $T$ acts on $V$, we have a decomposition
\[
V_{\overline{\F}_{p}} \, (\colonequals V\otimes_{\F_{p}}\overline{\F}_{p}) =\bigoplus_{\epsilon\in P(V,T)} V_{\overline{\F}_{p}}^{\epsilon},
\]
where $P(V,T)$ denotes the set of weights of $T$ in $V_{\overline{\F}_{p}}$ and $V_{\overline{\F}_{p}}^{\epsilon}$ denotes the weight space with respect to $\epsilon\in P(V,T)$.
As the action of $T$ on $V$ is $\F_{p}$-rational, the set $P(V,T)$ is equipped with an action of $\Gamma_{\F_{p}}=\Gal(\overline{\F}_{p}/\F_{p})$.
Furthermore, by putting $\Sigma_{\F_{p}}\colonequals \Gamma_{\F_{p}}\times\{\pm1\}$, $\Sigma_{\F_{p}}$ also acts on $P(V,T)$ ($-1$ acts via $\epsilon\mapsto-\epsilon$).
We say that a $\Gamma_{\F_{p}}$-orbit $\omega$ in $P(V,T)$ is symmetric (resp.\ asymmetric) if $-\omega=\omega$ (resp.\ $-\omega\neq\omega$).
For each $\Omega\in P(V,T)/\Sigma_{\F_{p}}$ such that $\epsilon\in \Omega$, we define a quadratic character $\chi^{T}_{\Omega}$ of $T(\F_{p})$ by
\[
\chi^{T}_{\Omega}(t)
\colonequals 
\begin{cases}
\epsilon(t)^{\frac{1-q_{\Omega}}{2}}&\text{if an(y) $\omega\subset\Omega$ is asymmetric,}\\
\epsilon(t)^{\frac{1+q_{\Omega}}{2}}&\text{if an(y) $\omega\subset\Omega$ is symmetric,}
\end{cases}
\]\symdef{chi-T-Omega}{$\chi^{T}_{\Omega}$}
for $t\in T(\F_{p})$, where we put $q_{\Omega}\colonequals p^{\frac{1}{2}|\Omega|}$.
We define a quadratic character $\chi^{T}$ of $T(\F_{p})$ by
\[
\chi^{T}\colonequals \prod_{\Omega\in P(V,T)/\Sigma_{\F_{p}}}\chi^{T}_{\Omega}.
\]\symdef{chi-T}{$\chi^{T}$}

\begin{prop}[{\cite[Corollary 4.8.1]{Ger77}}]\label{prop:Gerardin2}
For any $t\in T(\F_p)$, we have
\[
\Theta_{\omega_{V}}(t)
=
(-1)^{l(V,T;t)}\cdot p^{N(V;t)}\cdot\chi^{T}(t),
\]
where 
\begin{itemize}
\item
$l(V,T;t)\colonequals |\{\omega\in P(V,T)/\Gamma_{\F_{p}} \mid \text{$\epsilon(t) \neq1$ for an(y) $\epsilon\in\omega$}\}|$,
\item
$N(V;t)\colonequals \frac{1}{2}\dim_{\F_{p}}{V^{t}}$ (hence $p^{N(V;t)}=|V^{t}|^{\frac{1}{2}}$).
\end{itemize}
\end{prop}\symdef{l-V-T-t}{$l(V,T;t)$}\symdef{N-V-t}{$N(V;t)$}

Practically, in order to apply Proposition \ref{prop:Gerardin2} to computing $\Theta_{\omega_{V}}(g)$ for  a given semisimple element $g\in\Sp(V)$, we have to pick an $\F_{p}$-rational maximal torus $T$ of $\Sp(V)$ containing $g$ and analyze the structure of the set of weights $P(V,T)$ including its Galois action.
The following lemmas are useful for this:

\begin{lem}\label{lem:eigenvalue-conjugate}
Let $g,t\in\Sp(V)$ be semisimple elements.
If $g$ and $t$ have the same (multi-)sets of eigenvalues, then they are $\Sp(V)$-conjugate.
\end{lem}

\begin{proof}
The proof of this lemma should be standard, but we explain it for the sake of completeness.
Note that $\Sp(V)\subset \Sp(V_{\overline{\F}_{p}})\subset\GL(V_{\overline{\F}_{p}})$.
The assumption that $g$ and $t$ have the same eigenvalues implies that $g$ and $t$ are conjugate in $\GL(V_{\overline{\F}_{p}})$.
It is known that this furthermore implies that $g$ and $t$ are conjugate in $\Sp(V_{\overline{\F}_{p}})$ (for example, see \cite[275 page, Exercises 2.15 (ii)]{SS70}).

Let $x\in \Sp(V_{\overline{\F}_{p}})$ be an element such that $g=xtx^{-1}$.
Then we have $xtx^{-1}=g=\sigma(g)=\sigma(x)t\sigma(x)^{-1}$ for any $\sigma\in \Gamma_{\F_{p}}$.
In other words, by putting $H$ to be the centralizer of $t$ in $\Sp(V_{\overline{\F}_{p}})$, we have $\sigma(x)^{-1}x\in H$.
Hence we obtain a $1$-cocycle $z_{\sigma}\in Z^{1}(\Gamma_{\F_{p}}, H)$ given by $z_{\sigma}=\sigma(x)^{-1}x$.
Since $\Sp(V_{\overline{\F}_{p}})$ is simply-connected, $H$ is connected (for example, see \cite[Section 2.11]{Hum95}).
Thus we have $H^{1}(\Gamma_{\F_{p}}, H)=1$ by Lang's theorem.
This means that there exists an element $h\in H$ satisfying $z_{\sigma}=\sigma(h)^{-1}h$.
In particular, $xh^{-1}$ is $\F_{p}$-rational, i.e., an element of $\Sp(V)$.
As we have $(xh^{-1})t(xh^{-1})^{-1}=xtx^{-1}=g$, we obtain the assertion.
\end{proof}

\begin{lem}\label{lem:Sp-tori}
Let $2n\colonequals \dim_{\F_{p}}(V)$.
Let $k^{\circ}_{1},\ldots,k^{\circ}_{l}$ and $k^{\circ}_{l+1},\ldots,k^{\circ}_{r}$ be finite extensions of $\F_{p}$ satisfying $[k^{\circ}_{1}:\F_{p}]+\cdots +[k^{\circ}_{r}:\F_{p}]= n$.
Let $k_{i}$ be the quadratic extension of $k_{i}^{\circ}$ for $1\leq i\leq l$.
Then there exists an $\F_{p}$-maximal torus $T$ of $\Sp(V)$ of the form
\[
\prod_{i=1}^{l} \Ker(\Nr_{k_{i}/k_{i}^{\circ}}\colon \Res_{k_{i}/\F_{p}}\Gm\rightarrow\Res_{k_{i}^{\circ}/\F_{p}}\Gm)\times\prod_{i=l+1}^{r}\Res_{k_i^\circ/\F_p}\Gm.
\]
Moreover, we have the following:
\begin{enumerate}
\item
The set of weights $P(V,T)$ is of the form $\bigsqcup_{i=1}^{l}\Omega_{i}\sqcup \bigsqcup_{j=1}^{r}\pm\omega_i$, where 
\begin{itemize}
\item
$\Omega_{j}$ is a finite set of order $[k_{l}:\F_{p}]$ on which $\Gal(k_{l}/\F_{p})$ acts simply transitively and the unique nontrivial element of $\Gal(k_{l}/k_{l}^{\circ})$ acts via negation,
\item
$\omega_{i}$ is a finite set of order $[k_{l}^\circ:\F_{p}]$ on which $\Gal(k^\circ_{l}/\F_{p})$ acts simply transitively.
\end{itemize}
\item
If $t=(t_{1},\ldots,t_{l},t_{l+1},\ldots,t_{r})\in T(\F_{p})\cong \prod_{i=1}^{l} k_{i}^{1}\times\prod_{i=l+1}^{r}(k^\circ_i)^\times$, then the (multi-)set of eigenvalues of $t$ as an $\F_p$-automorphism of $V$ is given by
\[
\bigsqcup_{i=1}^{l}\{\sigma(t_{i}) \mid \sigma\in\Gal(k_{i}/\F_{p})\}
\sqcup
\bigsqcup_{i=l+1}^{r}\{\sigma(t_{i})^{\pm1} \mid \sigma\in\Gal(k^\circ_{i}/\F_{p})\}.
\]
\end{enumerate}
\end{lem}

\begin{proof}
If we let $\tau_{i}$ be the unique nontrivial element of $\Gal(k_{i}/k_{i}^{\circ})$, then we can define an $\F_{p}$-symplectic form on $k_{i}$ by
\[
(x,y)\mapsto \Tr_{k_{i}/\F_{p}}(x\tau_{i}(y)-\tau_{i}(x)y).
\]
Since the multiplication action of $k_{i}^{1}$ on $k_{i}$ preserves this symplectic form, we see that the symplectic group $\Sp(k_{i})$ (as an algebraic group over $\F_{p}$) contains an $\F_{p}$-rational torus $\Ker(\Nr_{k_{i}/k_{i}^{\circ}}\colon \Res_{k_{i}/\F_{p}}\Gm\rightarrow\Res_{k_{i}^{\circ}/\F_{p}}\Gm)$.
Since its rank is given by $[k_{i}^{\circ}:\F_{p}]$, which is the half of $\dim_{\F_{p}}(k_{i})$, it is a maximal torus.
Similarly, the action of $(k_{i}^{\circ})^\times$ on $k_{i}\times k_{i}$ given by $(x,y)\mapsto(zx,z^{-1}y)$ preserves the symplectic form represented by $(\begin{smallmatrix}0&1\\1&0\end{smallmatrix})$, hence $\Sp(k_{i})$ also contains an $\F_{p}$-rational torus $\Res_{k_i^\circ/\F_p}\Gm$.
Noting that $\Sp(V)\cong\Sp_{2n}$ contains $\prod_{i=1}^{r}\Sp(k_i)$, we see that the product $\prod_{i=1}^{l} \Ker(\Nr_{k_{i}/k_{i}^{\circ}}\colon \Res_{k_{i}/\F_{p}}\Gm\rightarrow\Res_{k_{i}^{\circ}/\F_{p}}\Gm)\times\prod_{i=l+1}^{r}\Res_{k_i^\circ/\F_p}\Gm$ is realized in $\Sp(V)$ as an $\F_p$-rational maximal torus.
The remaining assertions immediately follows from this explicit realization.
\end{proof}

We also introduce another formula of G\'erardin.

\begin{prop}[{\cite[Theorem 4.9.1 (a), (c)]{Ger77}}]\label{prop:Gerardin}
Let $g\in\Sp(V)$.
\begin{enumerate}
\item
Suppose that $g$ has no nonzero fixed point in $V$.
If $V'$ is a maximal $g$-invariant totally isotropic subspace of $V$, then we have
\[
\Theta_{\omega_{V}}(g)
=\sgn_{\F_{p}^{\times}}\bigl((-1)^{\frac{\dim{V_{0}}}{2}}\cdot\det(g\mid V')\cdot\det(g-1\mid V_{0})\bigr),
\]
where $V_{0}\colonequals V^{\prime\perp}/V'$.
\item
Suppose that $g$ fixes pointwise a line $L\subset V$.
If $V_{0}$ is a $g$-invariant subspace of $L^{\perp}$ such that $L^{\perp}=L\oplus V_{0}$, then we have
\[
\Theta_{\omega_{V}}(g)
=\Theta_{\omega_{V_{0}}}(g)\sum_{v\in V_{0}^{\perp}/L}\vartheta(\langle gv,v\rangle),
\]
where $\omega_{V_{0}}$ is a Heisenberg--Weil representation of $\Sp(V_{0})\ltimes\bbH(V_{0})$ with central character $\vartheta$.
\end{enumerate}
\end{prop}

\begin{lem}\label{lem:polarization}
Let $V=V_{+}\oplus V_{-}$ be a polarization of $V$ (i.e., $V_{+}$ and $V_{-}$ are totally isotropic subspaces).
Let $g\in\Sp(V)$ be a semisimple element and we suppose that $V_{+}$ and $V_{-}$ are invariant under $g$.
Then, for any line $L_{+}\subset V_{+}$ fixed by $g$ pointwise, there exist $g$-invariant decompositions $V_{+}=L_{+}\oplus M_{+}$ and $V_{-}=L_{-}\oplus M_{-}$ such that $L_{-}$ is a line fixed by $g$ pointwise and we have $L_{+}^{\perp}=V_{+}\oplus M_{-}$ and $L_{-}^{\perp}=V_{-}\oplus M_{+}$.
\end{lem}

\begin{proof}
Let $l$ be a nonzero element of the line $L_{+}$.
We put $M_{-}\colonequals \{v\in V_{-} \mid \langle l,v\rangle=0\}$.
Since $V=V_{+}\oplus V_{-}$ is a polarization, we can find an element $w\in V_{-}$ satisfying $\langle l,w\rangle\neq0$.
Note that, as $g$ is semisimple, the order of $g$ (say $p'$) is prime to $p$.
Thus, since $g$ stabilizes the subspace $V_{-}$, the averaged element $w'\colonequals \frac{1}{p'}\sum_{i=0}^{p'-1}g^{i}(w)$ belongs to $V_{-}$.
Moreover, $w'$ is $g$-invariant and satisfies $\langle l,w'\rangle=\langle l,w\rangle\neq0$.
By putting $L_{-}\colonequals \F_{p}w'$, we get $V_{-}=L_{-}\oplus M_{-}$.
By applying the same construction to $V_{+}$ using $w'$ instead of $l$, we get $V_{+}=L_{+}\oplus M_{+}$.
These decompositions satisfy the conditions as desired.
\end{proof}

\begin{cor}\label{cor:Gerardin}
Let $g\in\Sp(V)$ be a semisimple element and we suppose that we have a $g$-invariant polarization $V=V_{+}\oplus V_{-}$ of $V$.
Then we have
\[
\Theta_{\omega_{V}}(g)
=\sgn_{\F_{p}^{\times}}(\det(g\mid V_{+}))\cdot |V^{g}|^{\frac{1}{2}}.
\]
\end{cor}

\begin{proof}
When $V$ has no nonzero point fixed by $g$, then Proposition \ref{prop:Gerardin} (1) can be applied to $V'=V_{+}$.
Then, as $V=V_{+}\oplus V_{-}$ is a polarization, we have $V^{\prime\perp}=V_{+}$, hence $V_{0}=0$.
Thus we get
\[
\Theta_{\omega_{V}}(g)
=
\sgn_{\F_{p}^{\times}}(\det(g\mid V_{+})).
\]

We next suppose that $V$ has a nonzero point $v$ fixed by $g$.
We write $v=v_{+}+v_{-}$ according to the polarization $V=V_{+}\oplus V_{-}$ ($v_{+}\in V_{+}$ and $v_{-}\in V_{-}$).
Then, since the decomposition $V=V_{+}\oplus V_{-}$ is $g$-invariant, both of $v_{+}$ and $v_{-}$ are fixed by $g$.
Since $v\neq0$, either $v_{+}$ or $v_{-}$ is not zero.
We may assume that $v_{+}$ is not zero without loss of generality.
Then $L_{+}\colonequals \F_{p}v_{+}\subset V_{+}$ is a line fixed by $g$ pointwise.

We take $g$-invariant decompositions $V_{+}=L_{+}\oplus M_{+}$ and $V_{-}=L_{-}\oplus M_{-}$ as in Lemma \ref{lem:polarization}.
We use Proposition \ref{prop:Gerardin} (2) with $L=L_{+}$; then $L^{\perp}=V_{+}\oplus M_{-}=L_{+}\oplus M_{+}\oplus M_{-}$, hence $V_{0}$ can be taken to be $M_{+}\oplus M_{-}$.
Hence we get
\[
\Theta_{\omega_{V}}(g)
=
\Theta_{\omega_{V_{0}}}(g)
\sum_{v\in V_{0}^{\perp}/L}\vartheta(\langle gv,v\rangle).
\]
Since $V_{0}^{\perp}$ is given by $L_{+}\oplus L_{-}$ and $g$ fixes $L_{-}$ pointwise, we have
\[
\sum_{v\in V_{0}^{\perp}/L}\vartheta(\langle gv,v\rangle)
=\sum_{v\in L_{-}}\vartheta(\langle gv,v\rangle)
=\sum_{v\in L_{-}}\vartheta(\langle v,v\rangle)
=\sum_{v\in L_{-}}1
=p.
\]

Then the same argument can be applied to $V_{0}$.
By repeating this procedure and using the result on the case where $V$ has no nonzero point fixed by $g$, which is already proved in the first paragraph, we eventually get
\[
\Theta_{\omega_{V}}
=
\sgn_{\F_{p}^{\times}}(\det(g\mid V_{+}/V_{+}^{g}))\cdot p^{\dim_{\F_{p}}(V_{+}^{g})}.
\]
By noting that 
\[
\det(g\mid V_{+})
=\det(g\mid V_{+}^{g})\cdot \det(g\mid V_{+}/V_{+}^{g})
=\det(g\mid V_{+}/V_{+}^{g})
\]
and $p^{\dim_{\F_{p}}(V_{+}^{g})}=|V_{+}^{g}|=|V^{g}|^{\frac{1}{2}}$, we get the desired result.
\end{proof}

\begin{rem}
    In recent years, it has come to be recognized that G\'erardin's paper \cite{Ger77} contains an unfortunate typo.
    This was pointed out by Loren Spice and eventually became one of triggers which led to the idea of reexamining Yu's construction itself (the work of Fintzen--Kaletha--Spice \cite{FKS23}).
    The typo is something happened to appear accidentally only in the statement of a theorem (Theorem 2.4 (b); see \cite[Remark 3.2]{Fin21-Compos} and \cite[footnote comment on 339 page]{Fin21-Ann} for details), so it does not affect any of the subsequent computations in the paper, especially, the character formulas cited in this section.
\end{rem}

\subsection{Odds and Ends}
Here we collect some elementary lemmas related to computation of the characters of Heisenberg--Weil representations.

\begin{lem}\label{lem:eigen-s-th-root}
    Let $K$ be any field and we put $V:=K^{\oplus n}$.
    Let $r\in\Z_>0$ be a positive integer dividing $n$ (say $n=rs$).
    Let $D=\mathrm{diag}(a_1,\ldots,a_n)$ be a diagonal matrix and $\varphi:=(\begin{smallmatrix}0&I_{n-1}\\1&0\end{smallmatrix})$, i.e., $\varphi$ is the permutation matrix of length $n$.
    Then we have 
    \[
        \mcE_K(D\cdot\varphi^r \mid V)
        =
        \bigsqcup_{i=1}^{r}\{a_{i}\cdot a_{i+r}\cdot a_{i+2r} \cdots a_{i+(s-1)r}\}^{1/s},
    \]
    where each summand on the right-hand side denotes the multi-set of $s$-th roots of $a_{i}\cdot a_{i+r}\cdot a_{i+2r} \cdots a_{i+(s-1)r}$.
\end{lem}

\begin{proof}
    Let $p(X)\in K[X]$ be the characteristic polynomial of $D\cdot\varphi^r$.
    If we regard $a_1\ldots,a_n$ as formal variables, then $p(X)$ can be thought of as an element of $K[X,a_1\ldots,a_n]$.
    Note that 
    \[
        (D\cdot\varphi^r)^s
        =
        D\cdot (\varphi^{r}D\varphi^{-r})\cdot (\varphi^{2r}D\varphi^{-2r})\cdots (\varphi^{r(s-1)}D\varphi^{-r(s-1)});
    \]
    this is a diagonal matrix whose diagonal entries consist of $a_{i}\cdot a_{i+r}\cdot a_{i+2r} \cdots a_{i+(s-1)r}$ ($1\leq i \leq r$), where each $a_{i}\cdot a_{i+r}\cdot a_{i+2r} \cdots a_{i+(s-1)r}$ appears $s$ times.
    This implies that $p(X)$ at least divides $\prod_{i=1}^{r}(X^s-a_{i}\cdot a_{i+r}\cdot a_{i+2r} \cdots a_{i+(s-1)r})^{s}$.
    By noting that each $(X^s-a_{i}\cdot a_{i+r}\cdot a_{i+2r} \cdots a_{i+(s-1)r})$ is irreducible (as an element of $K[X,a_1\ldots,a_n]$) and must appear in $p(X)$ at least once, we conclude that
    \[
        p(X)=\prod_{i=1}^{r}(X^s-a_{i}\cdot a_{i+r}\cdot a_{i+2r} \cdots a_{i+(s-1)r})
    \]
    by counting the degrees.
    This completes the proof.
\end{proof}

\begin{lem}\label{lem:finite-field}
    Let $\F_q$ be a finite field of odd characteristic.
    Let $\F_{q^{2f}}$ be a finite extension of $\F_q$ of degree $2f$, where $f$ is prime to $p$.
    \begin{enumerate}
    \item Any $2f$-th root of any element of $\F_{q^2}^{1}$ lies in $\F_{q^{2f}}$.
    \item If $f$ is odd, any $f$-th root of any element of $\F_{q^2}^{1}$ lies in $\F_{q^{2f}}^{1}$.
    \end{enumerate}
    
\end{lem}

\begin{proof}
    We first show (1).
    Let $x\in \F_{q^2}^{1}$.
    Since $\F_{q^{2f}}^\times$ is cyclic of order $q^{2f}-1$, there exists a $2f$-th root of $x$ in $\F_{q^{2f}}$ if and only if $x^{\frac{q^{2f}-1}{2f}}=1$.
    As $\F_{q^2}^{1}$ is cyclic of order $q+1$, the order of $x$ divides $q+1$.
    Hence, it suffices to show that $\frac{q^{2f}-1}{2f}$ is divisible by $q+1$.
    (Note that, since $\F_{q^{2f}}^{1}$ contains all $2f$-th roots of unity, if $x$ has a $2f$-th root in $\F_{q^{2f}}^{1}$, then all roots are in $\F_{q^{2f}}^{1}$.)

    We write $2f=2^{r}s$, where $r\in\Z_{>0}$ and $s\in\Z_{>0}$ is an odd integer.
    Then we have
    \[
        q^{2f}-1
        =q^{2^{r}s}-1
        =(q^{2^{r-1}s}+1)\cdot(q^{2^{r-2}s}+1)\cdots(q^{s}+1)\cdot(q^{s}-1).
    \]
    Note that
    \begin{itemize}
        \item $(q^{2^{r-1}s}+1), \ldots, (q^{2s}+1)$ are divisible by $2$ since $q$ is odd, hence $(q^{2^{r-1}s}+1)\cdots(q^{2s}+1)$ is divisible by $2^{r-1}$;
        \item $(q^s+1)$ is divisible by $(q+1)$ since $s$ is odd;
        \item $(q^{s}-1)=(q-1)\cdot(q^{s-1}+\cdots+q+1)$;
        \begin{itemize}
            \item $(q-1)$ is divisible by $2$ since $q$ is odd;
            \item $(q^{s-1}+\cdots+q+1)$ is divisible by $s$ since $s$ is prime to $p$ by assumption on $f$.
        \end{itemize}
    \end{itemize}
    Hence, we see that $q^{2f}-1$ is divisible by $2^{r}s(q+1)=2f(q+1)$ as desired.

    We next show (2).
    Let $x\in \F_{q^2}^{1}$.
    Since $\F_{q^{2f}}^{1}$ is cyclic of order $q^{f}+1$, there exists an $f$-th root of $x$ in $\F_{q^{2f}}^{1}$ if and only if $x^{\frac{q^{f}+1}{f}}=1$.
    As $\F_{q^2}^{1}$ is cyclic of order $q+1$, the order of $x$ divides $q+1$.
    Hence, it suffices to show that $\frac{q^{f}+1}{f}$ is divisible by $q+1$.
    Since $f$ is odd, we have 
    \[
    q^{f}+1=(q+1)\cdot(q^{f-1}-q^{f-2}+\cdots-q+1).
    \]
    As $f$ is prime to $-q$, the second factor $q^{f-1}-q^{f-2}+\cdots-q+1$ is divisible by $f$.    
\end{proof}

\newpage
\input{symbols.ind}

\end{document}